\documentclass[10pt, oneside]{amsart}
\usepackage[utf8]{inputenc}
\pdfoutput=1

\title{Operads in Unstable Global Homotopy Theory}
\author{Miguel Barrero}
\address{IMAPP, Radboud University Nijmegen, The Netherlands}
\email{m.barrero@math.ru.nl}

\usepackage[english]{babel}
\usepackage{csquotes}
\usepackage{mathtools}
\usepackage[T1]{fontenc}%useful for dealing with diacritics and fonts
\usepackage{amsmath, amsthm, amsfonts, mathrsfs, amsfonts, amssymb}
\usepackage{tikz-cd}
\usepackage{quiver}
\usepackage[shortlabels]{enumitem}%to remove the bullets from itemize and more enumerate options

\usepackage[a4paper, textheight = 25cm, textwidth = 16cm, top = 2.5cm, centering]{geometry} %Page formatting
\usepackage[final]{microtype}

\usepackage[backend=biber, style=alphabetic, maxnames=5,minnames=3, maxalphanames=5, doi=false, isbn=false, url=false, eprint=false]{biblatex}
\usepackage{chngcntr}
\counterwithout{equation}{section}
\usepackage{hyperref}%hyperlinks

\usepackage{pict2e}%Adjunction arrows
\makeatletter
\newcommand{\adjunction}[4]{%
  #1\colon #2%
  \mathrel{\vcenter{%
    \offinterlineskip\m@th
    \ialign{%
      \hfil$##$\hfil\cr
      \longrightharpoonup\cr
      \noalign{\kern-.3ex}
      \smallbot\cr
      \longleftharpoondown\cr
    }%
  }}%
  #3 \noloc #4%
}
\newcommand{\longrightharpoonup}{\relbar\joinrel\rightharpoonup}
\newcommand{\longleftharpoondown}{\leftharpoondown\joinrel\relbar}
\newcommand\noloc{%
  \nobreak
  \mspace{6mu plus 1mu}
  {:}
  \nonscript\mkern-\thinmuskip
  \mathpunct{}
  \mspace{2mu}
}
\newcommand{\smallbot}{%
  \begingroup\setlength\unitlength{.15em}%
  \begin{picture}(1,1)
  \roundcap
  \polyline(0,0)(1,0)
  \polyline(0.5,0)(0.5,1)
  \end{picture}%
  \endgroup
}
\makeatother

\usetikzlibrary{decorations.markings}
\tikzset{double line with arrow/.style args={#1,#2}{decorate,decoration={markings,%
mark=at position 0 with {\coordinate (ta-base-1) at (0,1pt);
\coordinate (ta-base-2) at (0,-1pt);},
mark=at position 1 with {\draw[#1] (ta-base-1) -- (0,1pt);
\draw[#2] (ta-base-2) -- (0,-1pt);
}}}}
\tikzset{Equal/.style={-,double line with arrow={-,-}}}
\newtheorem{thm}{Theorem}[section]
\newtheorem{coro}[thm]{Corollary}
\newtheorem{lemm}[thm]{Lemma}
\newtheorem{prop}[thm]{Proposition}
\newtheorem{thmintro}{Theorem}

\theoremstyle{definition}
\newtheorem{defi}[thm]{Definition}
\newtheorem*{question}{Question}
\theoremstyle{remark}
\newtheorem{constr}[thm]{Construction}
\theoremstyle{remark}
\newtheorem{cond}[thm]{Condition}
\theoremstyle{remark}
\newtheorem{rem}[thm]{Remark}
\theoremstyle{remark}
\newtheorem{ex}[thm]{Example}

\DeclareMathAlphabet{\mathpzc}{OT1}{pzc}{m}{it}

\newcommand{\Topcat}{\underline{\smash{\mathrm{Top}}}}
\newcommand{\Setcat}{\underline{\mathrm{Set}}}
\newcommand{\id}{\mathrm{id}}
\newcommand{\RR}{\mathbb{R}}
\newcommand{\NN}{\mathbb{N}}
\DeclareMathOperator{\Ima}{Im}
\DeclareMathOperator{\tel}{tel}
\DeclareMathOperator{\Map}{Map}
\DeclareMathOperator{\ev}{ev}
\newcommand{\cat}{\mathscr{C}}
\newcommand{\dat}{\mathscr{D}}
\newcommand{\Spc}{\underline{\smash{\mathpzc{Spc}}}}

\DeclareMathOperator{\Hom}{Hom}
\DeclareMathOperator{\End}{End}
\DeclareMathOperator*{\colim}{colim}

\DeclareMathOperator{\Fun}{Fun}
\newcommand{\Fu}{\underline{\Fun}}

\newcommand{\GTopcat}{\underline{\smash{\mathrm{GTop}}}}
\newcommand{\KTopcat}{\underline{\smash{\mathrm{KTop}}}}

\newcommand{\GSpc}{\underline{\mathrm{G}}\underline{\smash{\mathpzc{Spc}}}}
\newcommand{\SnSpc}{\underline{\smash{\Sigma_n}}\underline{\smash{\mathpzc{Spc}}}}

\newcommand{\Skcatfull}{\underline{\smash{\Sigma_k}}\underline{\smash{\cat}}}
\newcommand{\Gcat}{\underline{\mathrm{G}}}
\newcommand{\Sncat}{\underline{\Sigma_n}}
\newcommand{\Lcat}{\underline{\mathrm{L}}}
\newcommand{\Opn}{\mathcal{O}_n}
\newcommand{\Ug}{\mathcal{U}_G}
\newcommand{\Uk}{\mathcal{U}_K}
\newcommand{\Ul}{\mathcal{U}_L}
\newcommand{\Op}{\mathcal{O}}

\newcommand{\Ld}{\mathcal{LD}}
\newcommand{\Pop}{\mathcal{P}}
\newcommand{\Popn}{\mathcal{P}_n}

\DeclareMathOperator{\conf}{Conf}
\DeclareMathOperator{\Comm}{Comm}
\newcommand{\FF}{\mathcal{F}}
\newcommand{\fat}{\mathscr{F}}

\newcommand{\sob}{\Sigma_\ast\textrm{-}\cat}
\newcommand{\algop}{\underline{\smash{\mathpzc{Alg}}}(\Op)}
\newcommand{\algcomm}{\underline{\smash{\mathpzc{Alg}}}(\Comm)}
\newcommand{\Uop}{U_{\algop}}
\newcommand{\Fop}{F_{\algop}}
\newcommand{\algpop}{\underline{\smash{\mathpzc{Alg}}}(\Pop)}
\newcommand{\Upop}{U_{\algpop}}
\newcommand{\Fpop}{F_{\algpop}}
\newcommand{\Hcof}{\mathpzc{Hcof}}

\newcommand{\OpeSpc}{\mathcal{OP}\textrm{-}\Spc}
\newcommand{\OpeTop}{\mathcal{OP}\textrm{-}\Topcat}
\newcommand{\KK}{\mathpzc{K}}

\newcommand{\OmGTopcat}{\underline{\smash{\mathrm{(O(m)}\times \mathrm{G)} \mathrm{Top}}}}
\newcommand{\Is}{\mathsf{I}}
\newcommand{\Js}{\mathsf{J}}
\newcommand{\Ks}{\mathsf{K}}
\DeclareMathOperator{\sh}{sh}
\newcommand{\reg}{\textrm{-reg}}

\begin{document}
\begin{abstract}
    We study operads in unstable global homotopy theory, which is the homotopy theory of spaces with compatible actions by all compact Lie groups. We show that the theory of these operads works remarkably well, as for example it is possible to give a model structure for the category of algebras over any such operad. We define global $E_\infty$-operads, a good generalization of $E_\infty$-operads to the global setting, and we give a rectification result for algebras over them.
\end{abstract}
\maketitle
\tableofcontents

\section{Introduction}

Operads were first introduced by May \cite{may1972geometry} to study infinite loop spaces. Since then they have found uses in many areas of mathematics, including algebra, higher category theory, geometry, and mathematical physics. In general, an operad codifies a collection of operations of varying arity in a symmetric monoidal category.

An algebra over an operad $\Op$ is a representation of the abstract operations that the operad encodes as actual operations in some object. For example, an algebra over the commutative operad $\Comm$ is a commutative monoid in the given symmetric monoidal category. Another important example is that of an \emph{$E_\infty$-operad}, which encodes a binary operation that is unital, associative and commutative but only up to all higher homotopies.

One area that has seen increased interest in the last decade is equivariant homotopy theory. It is dedicated to studying the homotopy theory of spaces with an action by a topological group $G$. One can construct operads in the category of $G$-spaces, and this yields a theory that is remarkably different to the non-equivariant case. Unlike in the non-equivariant case, where all $E_\infty$-operads are equivalent, there are multiple possible non-equivalent notions of what an $E_\infty$-operad in $G$-spaces could be, all of which are non-equivariantly $E_\infty$-operads. For example, there is the naive one, an $E_\infty$-operad in spaces given the trivial $G$-action. This is however not the best choice when one wants to study objects like equivariant infinite loop spaces or equivariant spectra with some multiplicative structure.

Instead the better choice is to look at both the $G$-action and the $\Sigma_n$-action on each $\Opn$ at the same time. An \emph{$E_\infty$-$G$-operad} is an operad in $G$-spaces where each $\Opn$ is a universal space for the family of \emph{graph subgroups} of $G \times \Sigma_n$. Algebras over an $E_\infty$-$G$-operad have more structure than algebras over a naive $E_\infty$-operad in $G$-spaces.

In this paper we look at operads in the setting of \emph{unstable global equivariant homotopy theory}. This is the homotopy theory of spaces which have simultaneous and compatible actions by all compact Lie groups. There are important equivariant constructions, like equivariant $K$-theory spectra and equivariant Thom spectra, that can be understood as a single globally equivariant object. We work with the model for unstable global homotopy theory based on \emph{orthogonal spaces}, introduced by Schwede \cite{global}. An orthogonal space can be thought of as the unstable analog of an orthogonal spectrum. There are some similarities between the theories of operads in the global equivariant setting and the $G$-equivariant setting for a single group $G$, but operads in the global equivariant setting are technically better behaved.

An orthogonal space has an \emph{underlying $K$-space} for each compact Lie group $K$. We study orthogonal spaces through these compatible $K$-actions for each $K$. A morphism of orthogonal spaces is said to be a \emph{global equivalence} if it is an equivalence of underlying $K$-spaces for each compact Lie group $K$. There is a model structure in the category of orthogonal spaces with these global equivalences as the weak equivalences, called the \emph{global model structure}.%This is not strictly true, but true for a good definition of underlying K-spaces.

A natural question to consider is whether one can construct a model structure on $\algop$ the category of algebras over a given operad $\Op$ in orthogonal spaces using the global model structure on the underlying category.

\begin{question}
Does the forgetful functor create the weak equivalences and fibrations of some model structure on the category of algebras over a given operad $\Op$?
\end{question}

The first place where this question was examined for a general category was in Spitzweck's PhD thesis \cite{spitzweck2001operads}, which provided some conditions under which this is true. The main technical point there was a factorization for pushout diagrams in the category of algebras over an operad. A different approach was used by Berger and Moerdijk \cite{Berger02axiomatichomotopy}. Pavlov and Scholbach \cite{PavlovSchAdmissibility} studied this question most extensively with full generality, and White and Yau \cite{WhiteYau} studied an analogous question for semi-model structures. We use a different factorization for pushout diagrams given by Sagave and Schlichtkrull \cite{sagave}, originally from Elmendorf and Mandell \cite{ElmendorfMandell}.

The first main result that we obtain is that the desired model structure exists for any operad in orthogonal spaces.

\begin{thmintro}[Theorem~\ref{thmoperadmodelcat}]
\label{thmintrooperadmodelcat}
Let $\Op$ be any operad in $(\Spc, \boxtimes)$ the category of orthogonal spaces, with the positive global model structure and the symmetric monoidal structure given by the box product. Then there is a cofibrantly generated model category structure on $\algop$ the category of algebras over $\Op$, where the forgetful functor $U_{\algop}$ creates the weak equivalences and fibrations, and sends cofibrations in $\algop$ to $h$-cofibrations in $\Spc$.
\end{thmintro}

This result is surprising, in that it holds for all operads. Such a result generally holds for all operads if the category is nice enough,  for example symmetric spectra based on simplicial sets, see the work of Harper \cite{Harper_2009}. One relevant property there is that all simplicial sets are cofibrant. Since not all orthogonal spaces are cofibrant, the approach of \cite{Harper_2009} does not apply to the present case.

Instead we use that the box product of orthogonal spaces is fully homotopical. By definition, this means that the box product of two global equivalences is a global equivalence, without any cofibrancy assumptions. This in turn removes any cofibrancy assumptions on the operad in Theorem~\ref{thmintrooperadmodelcat}.

Theorem~\ref{thmintrooperadmodelcat} was proven by Schwede \cite{global} for the specific case of the commutative operad $\Comm$. Algebras over $\Comm$ are the commutative monoids in orthogonal spaces with respect to the box product, which are usually called \emph{ultra-commutative monoids}, and they have a very rich structure. We generalize the result in \cite{global} to any operad. To accomplish this we need to use several different technical results and tools.

Some of these technical results deal with the $\Sigma_n$-objects in the category of orthogonal spaces, and so we study them in detail. We consider more generally orthogonal spaces which have an additional action by a fixed compact Lie group $G$, which we call \emph{$G$-orthogonal spaces}. Thus, the underlying $K$-space of a $G$-orthogonal space is a $(K \times G)$-space.

We define the notion of a $G$-global equivalence between $G$-orthogonal spaces, which takes into account both the $G$-action, and the action by each compact Lie group $K$. We also study various properties of these $G$-orthogonal spaces. In Appendix~\ref{appendixmodel} we give a model structure for $G$-orthogonal spaces which has the $G$-global equivalences as weak equivalences. Since $G$ is any compact Lie group, the results of Appendix~\ref{appendixmodel} are new in this generality.

This notion of "globally equivariant objects" with an additional action by a fixed group $G$ was studied extensively by Lenz in the context of Algebraic $K$-theory \cite{lenz2021gglobal}. There, various model structures were given for a discrete group $G$ not necessarily finite. Orthogonal spaces and orthogonal spectra with a $G$-action for a compact Lie group $G$ were also studied from the global point of view by Schwede \cite{schwede2021global}, however no model structure was defined there.

Our second main result is a characterization of morphisms of operads in orthogonal spaces that induce a Quillen equivalence between the respective categories of algebras.

\begin{thmintro}[Theorem~\ref{thmquillen}]
\label{thmintroquillen}
Let $g \colon \Op \to \Pop$ be a morphism of operads in $(\Spc, \boxtimes)$ the category of orthogonal spaces, with the positive global model structure and the symmetric monoidal structure given by the box product. Then the extension and restriction adjunction $(g_!, g^\ast)$ is a Quillen equivalence between the respective categories of algebras if and only if for each $n \geqslant 0$ the morphism $g_n \colon\Opn \to \Popn$ is a $\Sigma_n$-global equivalence.
\end{thmintro}

As was the case with Theorem~\ref{thmintrooperadmodelcat}, this result applies in full generality, to any morphism between any two operads. For a morphism $g$ between "nice" operads, it is enough to require that the morphisms $g_n$ are weak equivalences in the underlying category to obtain a Quillen equivalence, as shown by Spitzweck \cite{spitzweck2001operads}. However, for arbitrary operads in orthogonal spaces Theorem~\ref{thmintroquillen} does \emph{not} hold if each $g_n$ is merely a global equivalence, it additionally needs to be a $\Sigma_n$-global equivalence. In particular, if $\Op$ is a topological $E_\infty$-operad given the trivial global structure, $\algop$ is not Quillen equivalent to the category of ultra-commutative monoids. Thus, as in the $G$-equivariant setting mentioned at the beginning, this naive global $E_\infty$-operad is not the best one to consider.

Instead, Theorem~\ref{thmintroquillen} suggest a good notion of what an $E_\infty$-operad in the global equivariant sense should be. We define a \emph{global $E_\infty$-operad} to be an operad $\Op$ in $(\Spc, \boxtimes)$ such that each $\Opn$ is $\Sigma_n$-globally equivalent to $\ast$, the one-point orthogonal space. Then the naive global $E_\infty$-operads of the previous paragraph are not actually global $E_\infty$-operads. For any global $E_\infty$-operad $\Op$, Theorem~\ref{thmintroquillen} implies that the category of algebras over $\Op$ is Quillen equivalent to the category of ultra-commutative monoids. Thus, any algebra over a global $E_\infty$-operad can be rectified to an ultra-commutative monoid, and so these algebras also encode the highest possible level of commutativity.

In this article we provide several examples of global $E_\infty$-operads. Some of these are global analogs of classical operads in (equivariant) homotopy theory. These include a global version of the little disks operad and the Steiner operad, which are constructed in a similar way to the little disks and Steiner $G$-operads associated to a $G$-universe for a compact Lie group $G$.

In the $G$-equivariant case, there is a whole hierarchy of non-equivalent operads between a naive $E_\infty$-operad in $G$-spaces and an $E_\infty$-$G$-operad. These in-between operads are called \emph{$N_\infty$-operads}, and were introduced by Blumberg and Hill \cite{BLUMBERG2015658}. They codify various levels of commutativity, by imposing the existence of certain additive transfers/multiplicative norms. In the global setting, there is also a hierarchy of operads between the naive global $E_\infty$-operads and the global $E_\infty$-operads. These operads in orthogonal spaces are the global analogs of $N_\infty$-operads. We provide a classification of them in \cite{globalNin}.

\subsection{Structure of this paper}

In Section~\ref{sectionfirst} we begin by recalling the basic properties of operads as defined in any symmetric monoidal category. We then introduce unstable global homotopy theory, to put in context the questions that we examine. We also give plenty of examples of operads in orthogonal spaces, to build some intuition.

In Section~\ref{sectionGorthogonal}, we study $G$-orthogonal spaces. We begin by defining the $G$-global equivalences, and checking their basic properties. We then look at how $G$-global equivalences interact with taking $G$-orbits and with the box product. Lastly we introduce the $h$-cofibrations of $G$-orthogonal spaces, which are used in the proofs of our main results, Theorem~\ref{thmintrooperadmodelcat} and Theorem~\ref{thmintroquillen}, presented in Section~\ref{sectionmainresults}.

In Section~\ref{sectionEinfinity} we introduce global $E_\infty$-operads, and check that several of the examples of global operads given in Section~\ref{sectionfirst} are global $E_\infty$-operads.

There is a model structure on $G$-orthogonal spaces with the $G$-global equivalences as the weak equivalences. For completeness, we present the construction of this model structure in Appendix~\ref{appendixmodel}. We do not need this model structure to prove our main theorems.

\subsection{Notation and conventions}

We introduce here various mathematical and notational conventions that are used throughout this article.

Whenever we talk about a space we are referring to a compactly generated weak Hausdorff topological space. We use $\Topcat$ to denote the category of such spaces. In the rare cases where we refer to a general topological space, we do so explicitly. We underline the names used for specific categories, like $\Setcat$ or $\Topcat$, but not the variables like $\cat$. In particular, $\Gcat$ denotes the one-object groupoid associated to a group $G$.

We often use $i_l$ to refer to the boundary map $i_l \colon \partial D^l \rightarrow D^l$ in $\Topcat$, for each $l \geq 0$. Similarly we use $j_l$ for the inclusion $j_l \colon D^l \cong D^l \times \{0\} \to D^l \times [0, 1]$ for $l \geq 0$.

We use $\times$ for the categorical product, $\boxtimes$ for the box product of orthogonal spaces introduced in Remark~\ref{remproductsinspc}, and $\otimes$ for the tensor product in a generic symmetric monoidal category.

In this article we only consider compact Lie groups, and closed subgroups of them. By default, an inner product space refers to a real inner product space, finite-dimensional unless stated otherwise, and for a compact Lie group $G$, a $G$-representation means an orthogonal $G$-representation in an inner product space, also finite-dimensional unless stated otherwise.

A \emph{complete $G$-universe} is a countably-infinite-dimensional orthogonal $G$-representation with non-zero fixed points, and such that for each finite-dimensional $G$-representation $V$, a countably infinite direct sum of copies of $V$ embeds $G$-equivariantly into $\Ug$. We denote a complete $G$-universe by $\Ug$.

We write $\Sigma_n$ for the symmetric group on $n$ elements. By default, group actions are left group actions. Sometimes we turn a right action into a left action and vice versa by acting via the inverse, without saying so explicitly.

Let $G$ be a compact Lie group. For any set $\FF$ of closed subgroups of $G$, we say that a morphism $f \colon X \rightarrow Y$ of $G$-spaces is an \emph{$\FF$-equivalence} (respectively an \emph{$\FF$-fibration}) if for any $H \in \FF$ the restriction of $f$ to the $H$-fixed points $f^H \colon X^H \rightarrow Y^H$ is a weak homotopy equivalence (respectively a Serre fibration). For each set $\FF$ of closed subgroups of $G$ there is a cofibrantly generated model structure on $\GTopcat$ the category of $G$-spaces, with the $\FF$-equivalences as weak equivalences and the $\FF$-fibrations as fibrations (see~\cite[Proposition~B.7]{global}). We refer to the cofibrations of this model structure as \emph{$\FF$-cofibrations}. If the set $\FF$ is the set of all closed subgroups of $G$ we instead use $G$-equivalences, $G$-fibrations, and $G$-cofibrations to refer to these classes of morphisms.

Given two compact Lie groups $K$ and $G$ we refer often to the set of \emph{graph subgroups} of $K \times G$, denoted by $\fat(K, G)$ and defined in Definition~\ref{defigraphsub}. We generally use $\Gamma$ to denote a graph subgroup. Whenever we also need to refer to the continuous homomorphism $\phi$ associated to the graph subgroup, we use $\Gamma_\phi$ to denote the graph subgroup.

Finally, when we talk about small objects in a category with respect to a class of morphisms, we follow the conventions of~\cite[Section~2.1.1]{hovey2007model}. We use the letters $\Is$, $\Js$ and $\Ks$ to denote sets of generating (acyclic) cofibrations of various model categories.

\subsection{Acknowledgements}

This article is a modified and expanded version of my Master Thesis. I would like to thank Markus Hausmann, my Master Thesis supervisor, for suggesting the topic and for his guidance. I would also like to thank Magdalena Kędziorek for her advice and many helpful comments, Stefan Schwede for suggesting various improvements and one of the examples of Section~\ref{sectionEinfinity}, Tommy Lundemo and Eva Höning for reading an earlier version of this paper, and the referee for helpful suggestions.

\section{Background}
\label{sectionfirst}

\subsection{Operads}

Let $\cat$ be a cocomplete symmetric monoidal category, where the tensor product preserves all small colimits in both variables. We follow the exposition of \cite{fresse} to define operads in $\cat$. Let $\sob$ denote the category of \emph{symmetric sequences} in $\cat$, these are sequences $\{X(n)\}_{n \in \NN}$ of objects of $\cat$  where each $X(n)$ has a right $\Sigma_n$-action. So explicitly, $\sob$ is the functor category $\Fun(\coprod_{n \in \NN} \Sncat, \cat)$.

One can define a \emph{composition monoidal structure} on $\sob$, denoted by $\circ$ (see \cite[2.2.1 and 2.2.2]{fresse}). Then an \emph{operad} in $\cat$ is just a monoid in $(\sob, \circ)$. An operad $\Op$ in $\cat$ gives a monad $\FF(\Op)$ on $\cat$ (see \cite[2.1.1 and 2.2.1]{fresse}). An \emph{algebra} over this operad is defined as an algebra over the monad $\FF(\Op)$. We use $\algop$ to denote the category of algebras over $\Op$, and write $F_{\algop}$ and $U_{\algop}$ for the adjoint free and forgetful functors between $\cat$ and $\algop$.

From now on, let $\cat$ additionally be a cofibrantly generated model category. Given an operad $\Op$ in $\cat$ we want to lift the model structure of $\cat$ through the forgetful functor $U_{\algop} \colon \algop \to \cat$. That is, we want to consider the class of those morphisms which $U_{\algop}$ sends to weak equivalences, and the class of those sent to fibrations, and ask the question of whether these two classes determine a model structure on $\algop$. If they do, we say that the operad $\Op$ is \emph{admissible}.

The result \cite[Lemma~2.3]{schshi} gives conditions under which one can lift a model structure to the category of algebras over a monad. Let $\Is$ and $\Js$ denote sets of generating cofibrations and acyclic cofibrations of $\cat$ respectively. Set $\Is_\Op=F_{\algop}(\Is)$ and $\Js_\Op=F_{\algop}(\Js)$, and let $\Is_\Op\reg$ and $\Js_\Op\reg$ denote the regular $\Is_\Op$-cofibrations and regular $\Js_\Op$-cofibrations in $\algop$. Those are the transfinite compositions of cobase changes in $\algop$ of morphisms in $\Is_\Op$ and $\Js_\Op$. Applying~\cite[Lemma~2.3]{schshi} to the monad $\FF(\Op)$ associated to an operad $\Op$, and using that $\FF(\Op)$ always preserves filtered colimits (\cite[Proposition~2.4.1]{fresse}) one obtains the following result.

\begin{lemm}[{\cite[Lemma~2.3]{schshi} applied to operads}]
\label{lemmschsh}
Let $\Op$ be an operad in $\cat$. Assume that the sources of morphisms in $\Is_\Op$ and $\Js_\Op$ are small with respect to $\Is_\Op\reg$ and $\Js_\Op\reg$ respectively, and that every morphism in $\Js_\Op\reg$ is a weak equivalence in $\cat$. Then $\algop$ is a cofibrantly generated model category where $U_{\algop}$ creates the weak equivalences and fibrations and $\Is_\Op$ and $\Js_\Op$ are generating sets of cofibrations and acyclic cofibrations.
\end{lemm}

We have the following refinement of the result above, inspired by and similar to~\cite[Proposition~11.1.14]{fresse}.

\begin{thm}
\label{thmschsh2}
Let $\cat$ be a symmetric monoidal category which is also a cofibrantly generated model category with sets of generating cofibrations and acyclic cofibrations $\Is$ and $\Js$ respectively, and such that the monoidal product preserves all small colimits in each variable. Let $\Hcof$ be a class of morphisms in $\cat$ which satisfies the following:

\begin{enumerate}[a)]
    \item $\Hcof$ is closed under retracts and transfinite compositions.
    \item The sources of morphisms in $\Is$ and $\Js$ are small with respect to $\Hcof$.
    \item A map which is a transfinite composition of morphisms that are both in $\Hcof$ and are weak equivalences, is a weak equivalence.
\end{enumerate}

Fix any operad $\Op$ in $\cat$, and assume that for each pushout in $\algop$ of the form
\begin{equation}
\label{schsh2pushout}
% https://tikzcd.yichuanshen.de/#N4Igdg9gJgpgziAXAbVABwnAlgFyxMJZABgBpiBdUkANwEMAbAVxiRADEB9YAHR4Fs6OABZoAXgGNgAQQYBzAL4AKPgHk0ASmUANDSAWl0mXPkIoAjOSq1GLNl14ChoyTPnK1m5QE09Bo9h4BERk5tb0zKyIINL6hiAYgaZElmHUEXbRAEL61jBQcvBEoABmAE4Q-EhkIDgQSJY2kfbcfIIi4lKyiio86lpKWH7x5ZXV1HVIAEz+IKNViFMT9YgAzOm2UXNxpRULjZNrs-NI67UrxAoUCkA
\begin{tikzcd}
F_{\algop}(X) \arrow[r, "F_{\algop}(i)"] \arrow[d] & F_{\algop}(Y) \arrow[d] \\
A \arrow[r, "f"]                                                           & B \arrow[lu, phantom, "\ulcorner", very near start]                       
\end{tikzcd}
\end{equation}
the following hold:
\begin{enumerate}
    \item If $i\in \Is$ then $U_{\algop}(f)$ is in $\Hcof$.
    \item If $i \in \Js$ then $U_{\algop}(f)$ is a weak equivalence.
\end{enumerate}
Then the conditions of Lemma~\ref{lemmschsh} are satisfied, so $\algop$ is a cofibrantly generated model category where $U_{\algop}$ creates the weak equivalences and fibrations and $\Is_\Op$ and $\Js_\Op$ are generating sets of cofibrations and acyclic cofibrations of $\algop$. Furthermore $U_{\algop}$ sends cofibrations to morphisms in $\Hcof$.
\end{thm}

Note that in the conditions of Theorem~\ref{thmschsh2} the class of morphisms $\Hcof$ is not required to contain all the cofibrations of $\cat$. In our application of Theorem~\ref{thmschsh2} in Subsection~\ref{subsectionTheoremI}, we take $\Hcof$ to be the class of $h$-cofibrations, the morphisms with the homotopy extension property, hence the notation. In most settings, including the model of unstable global homotopy theory we use, the class of $h$-cofibrations does contain all cofibrations, but as mentioned this is not necessary.

\begin{proof}[Proof of Theorem~\ref{thmschsh2}]
We have to check that the sources of morphisms in $\Is_\Op$ and $\Js_\Op$ are small with respect to $\Is_\Op\reg$ and $\Js_\Op\reg$ respectively, and that every morphism in $\Js_\Op\reg$ is a weak equivalence in $\cat$. Then by Lemma~\ref{lemmschsh} the claim follows.

By~\cite[Proposition~4.3.2]{borceux}, $U_{\algop}$ preserves filtered colimits. Morphisms in $\Is_\Op\reg$ are transfinite compositions of cobase changes of morphisms with the form $F_{\algop}(i)$ for $i \in \Is$, as in Diagram~(\ref{schsh2pushout}). Therefore by our assumptions, $U_{\algop}$ sends morphisms in $\Is_\Op\reg$ to $\Hcof$.

Let $X$ be the source of a morphisms in $\Is$. By b) $X$ is $\kappa$-small with respect to $\Hcof$ for some cardinal~$\kappa$. Let $\lambda$ be a $\kappa$-filtered ordinal, and $V \colon \lambda \to \algop$ a $\lambda$-sequence which lands in $\Is_\Op\reg$. Then
\begin{align*}
    & \colim_\lambda \Hom_{\algop}(F_{\algop}(X), V) \cong \colim_\lambda \Hom_\cat(X, U_{\algop} \circ V) \cong \Hom_\cat (X, \colim_\lambda U_{\algop} \circ V ) \cong \\ & \Hom_\cat (X, U_{\algop}(\colim_\lambda V)) \cong \Hom_{\algop} (F_{\algop}(X), \colim_\lambda V)
\end{align*}
and $F_{\algop}(X)$ is $\kappa$-small with respect to $\Is_\Op\reg$. In the second isomorphism we are using that $U_{\algop}$ sends morphisms in $\Is_\Op\reg$ to $\Hcof$ and that $X$ is $\kappa$-small with respect to $\Hcof$.

Let $T$ denote the class of morphisms with the right lifting property with respect to $\Is_\Op$. By adjointness these are precisely those morphisms which $U_{\algop}$ sends to acyclic fibrations. By adjointness again $\Js_\Op$ has the left lifting property with respect to $T$, and then so does $\Js_\Op\reg$.

Let $f$ be a morphism of $\algop$ which has the left lifting property with respect to $T$. Use the small object argument for $\Is_\Op$ on $f$ to obtain that $f$ is a retract of $h \in \Is_\Op\reg$, and therefore $U_{\algop}(f)$ is a retract of $U_{\algop}(h) \in \Hcof$. Therefore $U_{\algop}$ sends $\Js_\Op\reg$ to $\Hcof$, and so we can repeat the previous argument to obtain that sources of morphisms in $\Js_\Op$ are small with respect to $\Js_\Op\reg$.

For the second condition, let $f\in \Js_\Op\reg$ be the transfinite composition of morphisms $f_\alpha$, such that each $f_\alpha$ is a cobase change (in $\algop$) of $F_{\algop}(j_\alpha)$ for $j_\alpha \in \Js$. The morphism $F_{\algop}(j_\alpha)$ has the left lifting property with respect to $T$, and then so does each $f_\alpha$. Then by the previous discussion $U_{\algop}(f_\alpha)$ is in $\Hcof$, and it is a weak equivalence by the hypothesis of the theorem. Since $U_{\algop}$ preserves transfinite compositions, $U_{\algop}(f)$ is a transfinite composition of morphisms that are both in $\Hcof$ and are weak equivalences, and so $U_{\algop}(f)$ is a weak equivalence.
\end{proof}

\begin{rem}
Usually in a category with both a model structure and a monoidal structure, two compatibility conditions are required (see for example \cite[Definition~3.1]{schshi}). These are the pushout product axiom, and the requirement that the unit is cofibrant. These are not necessary to prove Theorem~\ref{thmschsh2}, but something similar to the pushout product axiom is usually needed to actually check that condition (2) of Theorem~\ref{thmschsh2} holds in practice.
\end{rem}

\begin{rem}
The question of under which conditions on a category and on an operad one can lift the model structure to the category of algebras over said operad has been studied in a few different places and with diverse methods. As far as we can tell, the first place where this was examined in generality was in Spitzweck's PhD thesis \cite{spitzweck2001operads}. It was proven there that a model structure can be lifted assuming that the category $\cat$ satisfies the monoid axiom and that the operad is cofibrant (\cite[Theorem~4]{spitzweck2001operads}).

The condition that an operad is cofibrant in the model structure on operads constructed in \cite{spitzweck2001operads} is quite restrictive. It is stronger than asking that each $\Opn$ is cofibrant, or even $\Sigma_n$-cofibrant. In addition in the setting of global homotopy theory, it is not enough to look at the category of algebras over a cofibrant replacement of an operad (in the usual sense), as this gives the wrong homotopy theory, which we show in Remark~\ref{remcofibrantreplacement}.
\end{rem}

\subsection{Unstable global homotopy theory}
\label{subsectionunstable}

Unstable global homotopy theory is the homotopy theory of spaces which have simultaneous and compatible actions by all compact Lie groups. A model for this is the category of orthogonal spaces, studied in detail in Chapter 1 of \cite{global}. 

\begin{defi}[Definition~1.1.1 of \cite{global}]
Let $\Lcat$ be the $\Topcat$-enriched category where the objects are finite-dimensional real inner product spaces, and the morphisms are the linear isometric embeddings between them.

An \emph{orthogonal space} is a $\Topcat$-enriched functor $\Lcat \to \Topcat$. We use $\Spc$ to denote the $\Topcat$-enriched category of orthogonal spaces. Note the similarity of this definition to the definition of orthogonal spectra as enriched functors.
\end{defi}
\vspace{3mm}

If we have a compact Lie group $K$, and $V$ is a $K$-representation, then $X(V)$ inherits a $K$-action, where $k \in K$ acts via $X(k)$. In this sense, orthogonal spaces have actions by all compact Lie groups. For each compact Lie group $K$ we fix a complete $K$-universe $\Uk$ for the rest of this article. Let $s(\Uk)$ denote the poset of finite-dimensional subrepresentations of $\Uk$. Then we can associate to any orthogonal space the following $K$-space
\[X(\Uk) = \colim_{V \in s(\Uk)} X(V),\]
which we call the \emph{underlying $K$-space}. This yields a functor
\[(-)(\Uk) \colon \Spc \to \KTopcat.\]

A \emph{global equivalence} of orthogonal spaces is roughly speaking a morphism which for each compact Lie group $K$, induces $K$-equivalences on suitable homotopy colimits of finite-dimensional $K$-representations and equivariant embeddings between them. The precise definition can be found in \cite[Definition~1.1.2]{global}. It is also a special case of the definition of a $G$-global equivalence that we give in Definition~\ref{defiGglobal} of this article, with $G=e$.

An orthogonal space $X$ is said to be \emph{closed} if for each linear isometric embedding $\psi$ the map $X(\psi)$ is a closed embedding. A morphism between closed orthogonal spaces is a global equivalence if and only if for each compact Lie group $K$ the induced map on the underlying $K$-spaces is a $K$-equivalence, see~\cite[Proposition~1.1.17]{global}, or Proposition~\ref{propglobalcolim} below for the analogous result for $G$-global equivalences.

\begin{rem}
\label{remproductsinspc}
There are two symmetric monoidal structures on $\Spc$ which are relevant for us. The first one is the categorical product, denoted by $\times$, which is computed levelwise. The second is the box product, denoted by $\boxtimes$. It is constructed as a Day convolution product in \cite[Section~1.3]{global}. The unit of both is the terminal one-point constant orthogonal space $\ast$.

The box product can also be defined via a universal property. For each $X, Y \in \Spc$, consider the orthogonal spaces $Z$ with a bimorphism $(X, Y) \to Z$ from $X$ and $Y$. Then the box product of $X$ and $Y$ is an orthogonal space $X \boxtimes Y$ and a bimorphism $i \colon (X, Y) \to X \boxtimes Y$ of orthogonal spaces which is initial among such bimorphisms with source $(X, Y)$.
\end{rem}

\begin{rem}
\label{remmodelstructuresinspc}
On $\Spc$ there are two cofibrantly generated model structures whose weak equivalences are precisely the global equivalences; the global model structure~\cite[Theorem~1.2.21]{global} and the positive global model structure~\cite[Theorem~1.2.23]{global}. In the rest of this article we only consider the positive global model structure.
\end{rem}

\subsection{Examples of operads in unstable global homotopy theory}
\label{subsectionexamples}

In this paper, we study operads in orthogonal spaces with respect to the box product. This subsection is mostly devoted to showcasing several examples, which we study in more detail in Section~\ref{sectionEinfinity}.

\begin{rem}
\label{remoperadboxtocat}
For orthogonal spaces $X, Y \in \Spc$, we can construct a bimorphism $(X, Y) \to X \times Y$ via
\begin{equation}
\label{eqrhodef}
    % https://tikzcd.yichuanshen.de/#N4Igdg9gJgpgziAXAbVABwnAlgFyxMJZABgBpiBdUkANwEMAbAVxiRAA0AKANQEoACADqC8AW3j8AmpwDqvEAF9S6TLnyEUAJnJVajFmy7chgiGmZx+ck2InTjwsxasCAvP07sbWcZcm8eEycmSzlFXRgoAHN4IlAAMwAnCFEkMhAcCCQARmp6ZlZEDk5hfBw6AH1sgWFbPxLBMsrNeQUKBSA
\begin{tikzcd}
X(V) \times Y(W) \arrow[rr, "X(\iota_1) \times Y(\iota_2)"] &  & X(V \oplus W) \times Y(V \oplus W) = (X \times Y)(V \oplus W).
\end{tikzcd}
\end{equation}
This bimorphism yields a morphism of orthogonal spaces
\[\rho_{X, Y} \colon X \boxtimes Y \to X \times Y\]
which is natural in $X$ and $Y$. This means that the identity functor is a lax symmetric monoidal functor from $(\Spc, \times)$ to $(\Spc, \boxtimes)$, or equivalently an oplax symmetric monoidal functor from $(\Spc, \boxtimes)$ to $(\Spc, \times)$.

Therefore given an operad $\Op$ in $(\Spc, \times)$, the natural transformation $\rho$ gives an operad in $(\Spc, \boxtimes)$, with the same $\Opn$ for all $n \geq 0$. We denote this resulting operad in $(\Spc, \boxtimes)$ by $\Op^\boxtimes$.
\end{rem}\vspace{3mm}

\begin{ex}[Constant operads in $\Spc$ obtained from topological operads]
\label{exconstantoperad}
For any $X \in \Topcat$, we can consider the constant orthogonal space $\overline{X}$, that is the constant functor $\Lcat \to \Topcat$ with value $X$. This means that for any group $K$ the underlying $K$-space of $\overline{X}$ is just $X$ with the trivial $K$-action.

Any operad in spaces $\Op$ induces a constant operad $\overline{\Op}$ in $\Spc$, such that $(\overline{\Op})_n= \overline{\Opn}$. 
\end{ex}

\begin{constr}[Operad from a functor to topological operads]
\label{constrfunctoroperad}
Given a continuous functor $F$ from $\Lcat$ to the category of operads in spaces, $\OpeTop$, we can obtain from it an operad in $(\Spc, \times)$ by permuting the functoriality on $\Lcat$ with the functoriality on $\coprod_{n \in \NN} \Sncat$. Thus we obtain objects
\[\Op_{F, n} = (-)_n \circ F \in \SnSpc\]
The operadic structure on each $F(V)$ gives rise to an operadic structure on these $\Op_{F, n}$ with respect to the categorical product of orthogonal spaces. Then the natural morphism $\rho$ from $\boxtimes$ to $\times$ turns $\Op_F$ into an operad with respect to the box product, $\Op^\boxtimes_F$. 

If the functor $F$ is constant, then $\Op_F$ is just the constant operad of Example~\ref{exconstantoperad}. However this process becomes interesting when the actions of the linear isometric embeddings are non-trivial, as we discuss below.
\end{constr}

\begin{ex}[Little disks]
\label{exdisks}
For each inner product space $V$ consider the topological operad $\Ld(V)$ of little disks in $V$. These assemble into a continuous functor
\[\Ld \colon \Lcat \to \OpeTop\]
We understand an element of $\Ld(V)_n$ as a set of $n$ center points $v_i \in D(V)$ in the open unit disk of $V$ and $n$ radii $r_i$ that parametrize a rectilinear embedding of $n$ copies of $D(V)$ into itself. Then for a linear isometric embedding $\psi \colon V \to W$, the map $\Ld(\psi) \colon \Ld(V) \to \Ld(W)$ acts by sending each $v_i$ to $\psi(v_i)$.

By Construction~\ref{constrfunctoroperad} we obtain an operad $\Ld^\times$ in $(\Spc, \times)$, and $\Ld^\boxtimes$, shortened to $\Ld$, in $(\Spc, \boxtimes)$. The operad given by the underlying spaces of each $\Ld_n$ is precisely the $E_\infty$-operad of spaces obtained as the colimit of the little disks operads for $\RR^m$. Similarly, for a compact Lie group $K$, the underlying $K$-space of $\Ld_n$ is exactly $\Ld(\Uk)_n$, the $n$th space of the $K$-equivariant little disks operad for the complete $K$-universe $\Uk$, described for example in \cite[Definition~3.11 (ii)]{BLUMBERG2015658}.

\end{ex}

Analogously, there is a global version of the Steiner operad.

\begin{ex}[Steiner operad]
\label{exsteiner}
 For each inner product space $V$ let $R(V)$ be the space of distance reducing topological embeddings $f \colon V \to V$, where distance reducing means that $\lVert f(x) - f(y) \rVert \leqslant \lVert x - y \rVert$. This is a continuous functor $\Lcat \to \Topcat$, where for each linear isometric embedding $\psi \colon V \to W$ and $f \in R(V)$, the embedding $R(\psi)(f) \colon W \to W$ is given by
\[(\psi \circ f \circ \psi^{-1})\oplus \Ima(\psi)^\bot \colon \Ima(\psi) \oplus \Ima(\psi)^\bot \to \Ima(\psi) \oplus \Ima(\psi)^\bot\]

A Steiner path for $V$ is a map $h \colon [0, 1] \to R(V)$ such that $h(1)=\id_V$. Let $\KK(V)_n$ be the space of tuples $(h_1, \dots, h_n)$ of $n$ Steiner paths such that the images of the $h_i(0)$ are disjoint. These form $\KK(V)$ the Steiner operad for $V$, and these assemble into a continuous functor $F \colon \Lcat \to \OpeTop$, which gives the Steiner operad $\KK$ in $(\Spc, \boxtimes)$. As in the case of the little disks, for a compact Lie group $K$ the underlying $K$-space of $\KK_{\:n}$ is exactly $\KK(\Uk)_n$, the $n$th space of the $K$-equivariant Steiner operad for the complete $K$-universe $\Uk$, described for example in \cite[Definition~3.11 (iv)]{BLUMBERG2015658}. The underlying non-equivariant operad of $\KK$ is an $E_\infty$-operad in spaces.
\end{ex}

\begin{ex}[Endomorphism operads]
The symmetric monoidal category $(\Spc, \boxtimes)$ is closed (see \cite[Remark~C.12]{global}). Let $\Hom$ denote the internal $\Hom$ functor of $\Spc$. For each $X \in \Spc$ we can consider the endomorphism operad $\End(X)$ in $(\Spc, \boxtimes)$, where $\End(X)_n = \Hom(X^{\boxtimes n}, X)$.
\end{ex}

\begin{ex}
\label{exRoperads}
For each compact Lie group $K$ the underlying $K$-space functor
\[(-)(\Uk) \colon \Spc \to \KTopcat\]
has a right adjoint $R_K$, constructed in \cite[Construction 1.2.25]{global}. Being a right adjoint, $R_K$ is strong monoidal with respect to the categorical products in $\KTopcat$ and $\Spc$. Therefore if $\Op$ is any operad in $(\KTopcat, \times)$ then $R_K(\Op)$ is an operad in $(\Spc, \times)$, and by Remark~\ref{remoperadboxtocat} we also obtain an operad in $(\Spc, \boxtimes)$.
\end{ex}

\section{\texorpdfstring{$G$}{G}-orthogonal spaces}
\label{sectionGorthogonal}

In order to talk about operads in $\Spc$, we need to know more about the structure of orthogonal spaces which have an action by the symmetric group $\Sigma_n$. In this section we study orthogonal spaces with an additional action by a general compact Lie group $G$. We denote by $\GSpc$ the category of continuous functors from $\Gcat$ to $\Spc$, which we call \emph{$G$-orthogonal spaces}.

\begin{defi}
\label{defigraphsub}
Let $K$ and $G$ be compact Lie groups. A closed subgroup $\Gamma \leqslant K \times G$ is a \emph{graph subgroup} if $\Gamma \cap (\{e_K\} \times G) = \{e_{K\times G}\}$. We denote the set of graph subgroups of $K \times G$ by $\fat(K, G)$. They are called graph subgroups because for any $\Gamma \in \fat(K, G)$ there is a closed subgroup $H \leqslant K$ and a continuous homomorphism $\phi \colon H \to G$ such that $\Gamma$ is precisely the graph of $\phi$.
\end{defi}

Let $i_l$ denote the boundary map $i_l \colon \partial D^l \rightarrow D^l$ in $\Topcat$, for each $l \geq 0$. We use this notation throughout the paper. Given a $G$-orthogonal space $X$, a compact Lie group $K$, and a $K$-representation $V$, then $X(V)$ has a $(K \times G)$-action.

\begin{defi}[$G$-global equivalence]
\label{defiGglobal}
For a compact Lie group $G$, a morphism $f$ of $\GSpc$ is a \emph{$G$-global equivalence} if for each compact Lie group $K$, each graph subgroup $\Gamma \in \fat(K, G)$, each $K$-representation $V$ and $l \geqslant 0$, the following statement holds. For any lifting problem
\begin{equation*}
% https://tikzcd.yichuanshen.de/#N4Igdg9gJgpgziAXAbVABwnAlgFyxMJZABgBpiBdUkANwEMAbAVxiRAB1206AnPRgAQARAHoMQAX1LpMufIRRkAjFVqMWbUeKkzseAkSXlV9Zq0QgAGgAoAagEoRnNAAssk6SAx75h0iupTDQsATTtHZzdJVRgoAHN4IlAAMx4IAFskIxAcCCQAZh0QVIykACZqXILA9XNi8KcuKKKSzMQyHLzECrUzNk5GVzoPFLS2jqrEbJw6LHELFwgIAGtoiSA
\begin{tikzcd}
\partial D^l \arrow[r, "\alpha"] \arrow[d, "i_l", hook] & X(V)^\Gamma \arrow[d, "f(V)^\Gamma"] \\
D^l \arrow[r, "\beta"]                                    & Y(V)^\Gamma                       
\end{tikzcd}
\end{equation*}
there is a $K$-equivariant linear isometric embedding $\psi\colon V \to W$ into a $K$-representation $W$ such that there exists a morphism $\lambda \colon D^l \to X(W)^\Gamma$ which satisfies that in the following diagram
\begin{equation*}
% https://tikzcd.yichuanshen.de/#N4Igdg9gJgpgziAXAbVABwnAlgFyxMJZABgBpiBdUkANwEMAbAVxiRAB1206AnPRgAQARAHoMQAX1LpMufIRRkAjFVqMWbUeKkzseAkSXlV9Zq0QgAGgAoAagEoRnNAAssk6SAx75RAEzG1KYaFjYA6o7Obh66cgYoASpB6uYgAJrWEU5c0TpesvoKyEZJamZsGQ7Zru4SqjBQAObwRKAAZjwQALZIZCA4EEgBZSEc7IyudDEgHd291ANIRv10WOIWLhAQANbTsz2Iw4uIAMzJ5aHWztiRObWe+0sLg6fUDFhgqVB0cC4NIOdRpwGHQugAjb57ToHM79F4AFkBqTamVuNShc0Qy2OAFY8o9EDjnkhESNUhlrlg0bkKBIgA
\begin{tikzcd}
\partial D^l \arrow[r, "\alpha"] \arrow[d, "i_l", hook] & X(V)^\Gamma \arrow[r, "X(\psi)^\Gamma"] & X(W)^\Gamma \arrow[d, "f(W)^\Gamma"] \\
D^l \arrow[rru, "\lambda", dashed] \arrow[r, "\beta"']     & Y(V)^\Gamma \arrow[r, "Y(\psi)^\Gamma"] & Y(W)^\Gamma      \end{tikzcd}
\end{equation*}
the upper left triangle commutes, and the lower right triangle commutes up to homotopy relative to $\partial D^l$.
\end{defi}
\vspace{3mm}

Note that for $G= e$ this is just the definition of global equivalence from \cite[Definition~1.1.2]{global} mentioned in Subsection~\ref{subsectionunstable}. As it was the case for global equivalences, this definition is meant to capture that for each compact Lie group $K$ taking some suitable homotopy colimit of $K$-representations yields an $\fat(K, G)$-equivalence. We can make this more explicit with the following proposition, analogous to~\cite[Proposition~1.1.7]{global}.

\begin{defi}
For a compact Lie group $K$, we say that a nested sequence $\{V_i\}_{i \in \NN}$ of $K$-representations
\[V_0 \subset V_1 \subset \dots \subset V_i \subset \dots\]
is \emph{exhaustive} if for each $K$-representation $V$ there is an equivariant linear isometric embedding of $V$ into some $V_i$.
\end{defi}

\begin{prop}
\label{propglobaltel}
A morphism $f\colon X \to Y$ in $\GSpc$ is a $G$-global equivalence if and only if for each compact Lie group $K$ and each exhaustive sequence of $K$-representations $\{V_i\}_{i \in \NN}$, the map
\[\tel_i f(V_i) \colon \tel_i X(V_i) \to \tel_i Y(V_i)\]
induced on the mapping telescopes of the sequences $X(V_i)$ and $Y(V_i)$ of $(K \times G)$-spaces and $(K \times G)$-equivariant maps is an $\fat(K, G)$-equivalence.
\end{prop}
\begin{proof}
First we assume that for each compact Lie group $K$ and each exhaustive sequence of orthogonal $K$-representations, $\tel_i f(V_i)$ is an $\fat(K, G)$-equivalence of $(K \times G)$-spaces. Any compact Lie group $K$ has an exhaustive sequence of representations $\{V_i\}_{i \in \NN}$, so for any $K$-representation $V$, any graph subgroup $\Gamma \in \fat(K, G)$ and any lifting problem $(\alpha, \beta)$ for $f(V)^\Gamma$, since $\{V_i\}_{i \in \NN}$ is exhaustive, we can embed $V$ into some $V_n$, and so we assume that $V= V_n$.

Now we fix some notation. Let
\[c_{X, n} \colon X(V_n) \to \tel_i X(V_i)\]
be the canonical $(K \times G)$-equivariant map. Let $\tel_{[0, n]} X(V_i)$ denote the truncated mapping telescope. Let
\[\pi_{X, n} \colon \tel_{[0, n]} X(V_i) \to X(V_n)\]
be the $(K \times G)$-equivariant canonical projection. Slightly abusing notation we also use $c_{X, n}$ for the canonical map
\[c_{X, n} \colon X(V_n) \to \tel_{[0, n]} X(V_i).\]
For $n \leqslant m$, let
\[c_{X, n, m} \colon \tel_{[0, n]} X(V_i) \to \tel_{[0, m]} X(V_i)\]
denote the inclusion of truncated mapping telescopes, and 
\[c_{X, n, \infty} \colon \tel_{[0, n]} X(V_i) \to \tel_i X(V_i)\]
the canonical map.

Taking fixed points commutes with the construction of the mapping telescopes by~\cite[Proposition~B.1]{global}, so $(\tel_i X(V_i))^\Gamma \cong \tel_i X(V_i)^\Gamma$ for each graph subgroup $\Gamma \in \fat(K, G)$. Since we assumed that $\tel_i f(V_i)^\Gamma$ is a weak homotopy equivalence, by~\cite[9.6 Lemma]{may1999concise} there is a map $\lambda$ associated to the lifting problem $(c^\Gamma_{X, n} \circ \alpha, c^\Gamma_{Y, n} \circ \beta)$ such that the upper-left triangle commutes and the lower-right one commutes up to homotopy relative $\partial D^l$, witnessed by a homotopy $H$.
\begin{equation*}
    % https://tikzcd.yichuanshen.de/#N4Igdg9gJgpgziAXAbVABwnAlgFyxMJZABgBpiBdUkANwEMAbAVxiRAB1206AnPRgAQARAHoMQAX1LpMufIRRkAjFVqMWbUeKkzseAkSXlV9Zq0QgAGgAoAagH0wAShGc0ACyyTpIDHvmGpCrUphoWAJp2ji5unt66cgYoAEzGIermHOw4MAz2WAI2DlgxXHE6vrL6CsipwWpmbJw5eQWRxaUeXhKqMFAA5vBEoABmPBAAtkhkIDgQSKkNYVmMHnTxIGOTSEaz84gAzOmNFpwARjA469QMdBcMAApVARY8WP3uOBtbU4gzcztjsssPZtD4fgtqADDhUIYcofsAKzUC5gKDTIGZADG9mA4VcZSwpAEYAkIBud1yT38SRAbw+X1h41+i2hABYUTA0RiltjcZYCV1iaTvsykBy9khkbymtlcvkBCMoiVBeVwWLELt2TcsGBMlA6HB3H1yTLTuxbhMzgbJBQJEA
\begin{tikzcd}
\partial D^l \arrow[r, "\alpha"] \arrow[d, "i_l", hook]      & X(V_n)^\Gamma \arrow[r, "{c^\Gamma_{X, n}}"] & \tel_i X(V_i)^\Gamma \arrow[d, "\tel_i f(V_i)^\Gamma"] \\
D^l \arrow[r, "\beta"'] \arrow[rru, "\lambda", dashed] & Y(V_n)^\Gamma \arrow[r, "{c^\Gamma_{Y, n}}"']          & \tel_i Y(V_i)^\Gamma                                
\end{tikzcd}
\end{equation*}
Both $\lambda$ and $H$ have compact domains, and since the $\Gamma$-fixed points of the mapping telescopes are colimits along the closed embeddings $c^\Gamma_{X, n, m}$, both $\lambda$ and $H$ factor through some stage $m \geqslant n$ with $\psi\colon V_n \to V_m$, giving
\[\lambda' \colon D^l \to \tel_{[0, m]} X(V_i)^\Gamma \, \text{and} \, H' \colon D^l \times [0, 1] \to \tel_{[0, m]} Y(V_i)^\Gamma.\]
Then $\pi^\Gamma_{X, m} \circ \lambda'$ and $\pi^\Gamma_{X, m} \circ H'$ satisfy the requirements for the lifting problem $(X(\psi)^\Gamma \circ \alpha, X(\psi)^\Gamma \circ \beta)$, so $f$ is a $G$-global equivalence.

Now assume that $f$ is a $G$-global equivalence. Fix a compact Lie group $K$, a graph subgroup $\Gamma\in\fat(K, G)$, and an exhaustive sequence of $K$-representations $\{V_i\}_{i \in \NN}$. We have to check that $\tel_i f(V_i)^\Gamma$ is a weak homotopy equivalence.

For a lifting problem $(\alpha, \beta)$ for $\tel_i f(V_i)^\Gamma$, since $\partial D^l$ and $D^l$ are compact, $\alpha$ and $\beta$ factor through some stage $n$, as
\[\alpha' \colon \partial D^l \to \tel_{[0, n]} X(V_i)^\Gamma \;\text{and} \;\beta'\colon D^l \to \tel_{[0, n]} Y(V_i)^\Gamma.\]

For each $n$, there is a homotopy from the identity on $\tel_{[0, n]} X(V_i)$ to $c_{X, n} \circ \pi_{X, n}$, which is $(K \times G)$-equivariant and natural in $X$, given by retracting the truncated mapping telescope. By~\cite[Lemma~1.1.5]{global} this means that if there is a solution of the lifting problem
\[(c^\Gamma_{X, n} \circ \pi^\Gamma_{X, n} \circ \alpha', c^\Gamma_{Y, n} \circ \pi^\Gamma_{Y, n} \circ \beta')\]
then there is a solution of the lifting problem $(\alpha', \beta')$. In this proof we use the terminology \emph{solution of a lifting problem}\footnote{To avoid confusion with the more common meaning of the terminology \emph{solution of a lifting problem}, we do not use it outside of this proof.} in the sense of \cite[Lemma~1.1.5]{global} and Definition~\ref{defiGglobal}, i.e. a map that makes the upper-left triangle commute and the lower-right triangle commute up to homotopy relative $\partial D^l$.  
\begin{equation*}
    % https://tikzcd.yichuanshen.de/#N4Igdg9gJgpgziAXAbVABwnAlgFyxMJZABgBpiBdUkANwEMAbAVxiRAB1206AnPRgAQARAHoMQAX1LpMufIRRkAjFVqMWbUeKkzseAkSXlV9Zq0Qd2OGAwD6wEqQFgKEgQA0AFADVbWAJQinGgAFliS0iAYevKGpCrUphoWnNZ2DmTOrgIAmj5+gcFhEbpyBigAzMaJ6uaWafaOWW5evgFBXMU6UbL6CshVCWpmbKk2jZkubnlthZ3h3dFl-QBM8Sa1bDO2YHOhC5FLfURrlDUjFq07e12qMFAA5vBEoABmPBAAtkhkIDgQSDWw2SlkYoToAHISiB3l8kEY-gDEFVgXVOAAjGA4SEgagMOiYhgABV6sQsPCwDxCOGhsO+iF+-3h5xBWFs2kidMB1CZyJZaKs4wyTimAle+XaRQObw+9JRvIAbHisGA6nAIAwsFBcajRlw2cAch19iKJDr8YSSTFyiAKVSad0uYggbyAOzK1VsdWa7X8vVoA3uY1hU202VId2IpAAFj9FgAxvYg1KnJ8zXiCTYrcs2AwYK8HZzw4glVHEABWOMgROG4NYVPpkAWrOkm126lhuGIWNlyu6lKC9JNUXi2Z1zv0yOKquj2yfG4LVxAA
\begin{tikzcd}
\partial D^l \arrow[r, "\alpha'"] \arrow[d, "i_l", hook] & {\tel_{[0, n]} X(V_i)^\Gamma} \arrow[d, "{\tel_{[0, n]} f(V_i)^\Gamma}"] \arrow[r, "{\pi^\Gamma_{X, n}}"] & X(V_n)^\Gamma \arrow[r, "{c^\Gamma_{X, n}}"] \arrow[d, "f(V_m)^\Gamma"] & {\tel_{[0, n]} X(V_i)^\Gamma} \arrow[d, "{\tel_{[0, n]} f(V_i)^\Gamma}"] \\
D^l \arrow[r, "\beta'"']                           & {\tel_{[0, n]} Y(V_i)^\Gamma} \arrow[r, "{\pi^\Gamma_{Y, n}}"']                                         & Y(V_n)^\Gamma \arrow[r, "{c^\Gamma_{Y, n}}"']                         & {\tel_{[0, n]} Y(V_i)^\Gamma}                                         
\end{tikzcd}
\end{equation*}

The new lifting problem $(\pi^\Gamma_{X, n} \circ \alpha', \pi^\Gamma_{Y, n} \circ \beta')$ has as solution $\lambda$ after evaluating at some larger $m \geqslant n$ with embedding $\psi\colon V_n \to V_m$, because $f$ is a $G$-global equivalence and $\{V_i\}_{i \in \NN}$ is an exhaustive sequence of $K$-representations. Then $c^\Gamma_{X, m} \circ \lambda$ is a solution of the lifting problem
\[(c^\Gamma_{X, m} \circ X(\psi)^\Gamma \circ \pi^\Gamma_{X, n} \circ \alpha', c^\Gamma_{Y, m} \circ Y(\psi)^\Gamma \circ \pi^\Gamma_{Y, n} \circ \beta')\]
and since $\pi^\Gamma_{X, m} \circ c^\Gamma_{X, n, m} = X(\psi)^\Gamma \circ \pi^\Gamma_{X, n}$, the map $c^\Gamma_{X, m} \circ \lambda$ is also a solution of
\[(c^\Gamma_{X, m} \circ \pi^\Gamma_{X, m} \circ c^\Gamma_{X, n, m} \circ \alpha', c^\Gamma_{Y, m} \circ \pi^\Gamma_{Y, m} \circ c^\Gamma_{Y, n, m} \circ \beta').\]

By \cite[Lemma~1.1.5]{global} and the previously mentioned homotopy from the identity on $\tel_{[0, m]} X(V_i)$ to $c_{X, m} \circ \pi_{X, m}$, the lifting problem $(c^\Gamma_{X, n, m} \circ \alpha', c^\Gamma_{Y, n, m} \circ \beta')$ has a solution $\lambda'$. Note that we did not obtain a solution of $(\alpha', \beta')$, but since $c^\Gamma_{X, m, \infty} \circ c^\Gamma_{X, n, m} = c^\Gamma_{X, n, \infty}$, the map $c^\Gamma_{X, m, \infty} \circ \lambda'$ is a solution of the original lifting problem
\[(\alpha, \beta)=(c^\Gamma_{X, m, \infty} \circ c^\Gamma_{X, n, m} \circ \alpha', c^\Gamma_{Y, m, \infty} \circ c^\Gamma_{Y, n, m} \circ \beta'). \]

Since any lifting problem for $\tel_i f(V_i)^\Gamma$ has a solution, by~\cite[9.6 Lemma]{may1999concise} the map $\tel_i f(V_i)^\Gamma$ is a weak homotopy equivalence.
\end{proof}

Recall that an orthogonal space $X$ is said to be \emph{closed} if for each linear isometric embedding $\psi$ the map $X(\psi)$ is a closed embedding. We similarly define a \emph{closed $G$-orthogonal space} to be a $G$-orthogonal space $X$ such that $X(\psi)$ is a closed embedding for each linear isometric embedding $\psi$. We have fixed a complete $K$-universe $\Uk$ for each compact Lie group $K$. The \emph{underlying $(K \times G)$-space} of a $G$-orthogonal space $X$ is the underlying $K$-space of $X$ as an orthogonal space with the induced $G$-action. This is precisely the colimit $X(\Uk) = \colim_{V \in s(\Uk)} X(V)$ over the finite-dimensional subrepresentations of $\Uk$. Therefore analogously to~\cite[Proposition~1.1.17]{global} we have the following simpler characterization of $G$-global equivalences.

\begin{prop}
\label{propglobalcolim}
A morphism $f\colon X \to Y$ in $\GSpc$ between closed $G$-orthogonal spaces is a $G$-global equivalence if and only if for each compact Lie group $K$ the map induced on the underlying $(K \times G)$-spaces
\[f(\Uk) \colon X(\Uk) \to Y(\Uk)\]
is an $\fat(K, G)$-equivalence of $(K \times G)$-spaces.
\end{prop}
\begin{proof}
The colimits that define $X(\Uk)$ and $Y(\Uk)$ can be written as sequential colimits
\[\colim_{V \in s(\Uk)} X(V) \cong \colim_{i \in \NN} X(V_i),\]
for a nested sequence of finite-dimensional subrepresentations $\{V_i\}_{i \in \NN}$ of $\Uk$ which cover all of $\Uk$. These are colimits of $(K \times G)$-spaces along closed embeddings because $X$ and $Y$ are closed. Then for each $\Gamma \in \fat(K, G)$, taking $\Gamma$-fixed points commutes with this colimit along closed embeddings (see~\cite[Proposition~B.1 ii)]{global}). Since additionally $\partial D^l$ and $D^l$ are compact, a lifting problem for $(\colim_{i \in \NN} f(V_i))^\Gamma$ factors through some stage $n$ of the sequential colimit. By~\cite[9.6 Lemma]{may1999concise} we obtain that if $f$ is a $G$-global equivalence the map $(\colim_{i \in \NN} f(V_i))^\Gamma$ is a weak homotopy equivalence.

For the other implication, assume that $f(\Uk)^\Gamma$ is a weak homotopy equivalence for each compact Lie group $K$ and each $\Gamma \in \fat(K, G)$. Let $V$ be a $K$-representation. Then $V$ embeds into $\Uk$, so we may fix an embedding $V \to \Uk$ and call it $\psi$. Given any lifting problem
\[(\alpha \colon \partial D^l \to X(V)^\Gamma, \beta \colon D^l \to Y(V)^\Gamma)\]
for $f(V)^\Gamma$, consider the lifting problem
\[(X(\psi)^\Gamma \circ \alpha, Y(\psi)^\Gamma \circ \beta)\]
for $f(\Uk)^\Gamma$. By~\cite[9.6 Lemma]{may1999concise} since $f(\Uk)^\Gamma$ is a weak homotopy equivalence there exists some $\lambda \colon D^l \to X(\Uk)^\Gamma$ such that $\lambda \circ i_l = X(\psi)^\Gamma \circ \alpha$ and $f(\Uk)^\Gamma \circ \lambda$ is homotopic relative $\partial D^l$ to $Y(\psi)^\Gamma \circ \beta$, and we denote this homotopy by $H$.

Both $\lambda$ and $H$ factor through some stage of the colimits as
\[\lambda' \colon D^l \to X(V_j)^\Gamma \;\text{and}\; H' \colon D^l \times [0, 1] \to Y(V_j)^\Gamma.\]
We can choose $j$ so that $V_j$ contains the image of $\psi$ since $\{V_i\}_{i \in \NN}$ covers $\Uk$. Then $\lambda'$ and the homotopy $H'$ witness that $f$ is a $G$-global equivalence.
\end{proof}

 \begin{rem}
 \label{remwhygraph1}
 When defining $G$-global equivalences, we could have decided to look at all subgroups instead of only the graph subgroups of $K \times G$. This would give a strictly smaller class of $G$-global equivalences.
 
 We consider only the graph subgroups because between $G$-free orthogonal spaces, the graph subgroups tell the whole story, since the fixed points of any non-graph subgroup are empty. This means that looking at this bigger class of $G$-global equivalences is enough for the proof of Theorem~\ref{thmoperadmodelcat}. It also leads to Theorem~\ref{thmquillen}, which states that a map $f$ of operads in $(\Spc, \boxtimes)$ gives a Quillen equivalence if and only if each $f_n$ is a $\Sigma_n$-global equivalence in the sense of the graph subgroups, even if the operads themselves are not $\Sigma_n$-free.
 \end{rem}\vspace{3mm}

The $G$-global equivalences are in fact a part of a model structure on $\GSpc$, which we call the $G$-global model structure. We include in this section the basic facts about $G$-global equivalences, as well as the results about $G$-global equivalences that are most relevant to the proofs of Section~\ref{sectionmainresults}. We relegate the construction of this $G$-global model structure to Appendix~\ref{appendixmodel}.

The characterization of $G$-global equivalences given by Proposition~\ref{propglobaltel} makes it simple to check the following general properties of $G$-global equivalences.

\begin{lemm}
\label{lemmGglobal}

For compact Lie groups $G$ and $H$ we have the following properties:
\begin{enumerate}[label = \roman*)]
    \item (2-out-of-6) Consider three composable morphisms of $G$-orthogonal spaces $f$, $g$ and $h$. If $g \circ f$ and $h \circ g$ are $G$-global equivalences, then $f$, $g$, $h$ and $h \circ g \circ f$ are $G$-global equivalences.
    \item A retract of a $G$-global equivalence is a $G$-global equivalence.
    \item If $f\colon X \to Y$ is a $G$-global equivalence and $g$ is homotopic to $f$ through morphisms of $G$-orthogonal spaces, then $g$ is a $G$-global equivalence.
    \item For a $G$-orthogonal space $X$, and an $H$-global equivalence $f\colon Y \to Z$ between $H$-orthogonal spaces, the morphism $X \times f$ is a $(G \times H)$-global equivalence.
    \item For a $G$-global equivalence $f\colon X \to Y$ and an $H$-global equivalence $f'\colon X' \to Y'$, the morphism $f \times f'$ is a $(G \times H)$-global equivalence.
    \item For a $G$-global equivalence $f\colon X \to Y$ and a continuous homomorphism $\eta \colon H \to G$ the restriction $\eta^\ast f$ is an $H$-global equivalence.
\end{enumerate}
\end{lemm}
\begin{proof}
Let $K$ be a compact Lie group and $\{V_i\}_{i \in \NN}$ an exhaustive sequence of $K$-representations.
\begin{enumerate}[label = \roman*)]
    \item  By Proposition~\ref{propglobaltel} the maps 
    \[\qquad \tel_i (g \circ f)(V_i)= \tel_i g(V_i) \circ \tel_i f(V_i) \;\text{and}\; \tel_i (h \circ g)(V_i)= \tel_i h(V_i) \circ \tel_i g(V_i)\]
    are $\fat(K, G)$-equivalences. Since the class of $\fat(K, G)$-equivalences satisfies the 2-out-of-6 property, by Proposition~\ref{propglobaltel} again we obtain that $f$, $g$, $h$ and $h \circ g \circ f$ are $G$-global equivalences.
    \item As before, if $g$ is a retract of $f$ then $\tel_i g(V_i)$ is a retract of $\tel_i f(V_i)$, and $\fat(K, G)$-equivalences are closed under retracts. 
    \item If $H\colon X \times [0, 1] \to Y$ is a homotopy through morphisms of $G$-orthogonal spaces then it induces a homotopy through $(K \times G)$-equivariant maps on mapping telescopes. Since a map $(K \times G)$-homotopic to an $\fat(K, G)$-equivalence is an $\fat(K, G)$-equivalence, we see that $g$ is also a $G$-global equivalence.
    \item The canonical map
    \[\tel_i (X \times Y)(V_i) \to \tel_i X(V_i) \times \tel_i Y(V_i)\]
    is a $(K \times G \times H)$-equivalence. Consider a graph subgroup $\Gamma_\phi \in \fat(K, G \times H)$ associated to a homomorphism $\phi$. Let $\pi_G \colon G \times H \to G$ and $\pi_H \colon G \times H \to H$ denote the respective projections. Then
    \[(\tel_i X(V_i) \times \tel_i f(V_i))^{\Gamma_\phi} \cong \tel_i X(V_i)^{\Gamma_{\pi_G \circ \phi}} \times \tel_i f(V_i)^{\Gamma_{\pi_H \circ \phi}}\]
    where $\Gamma_{\pi_H \circ \phi}$ is the graph subgroup associated to $\pi_H \circ \phi$. Since $\tel_i f(V_i)$ is an $\fat(K, H)$-equivalence, $\tel_i X(V_i) \times \tel_i f(V_i)$ is an $\fat(K, G \times H)$-equivalence and by the 2-out-of-6 property so is $\tel_i (X \times f)(V_i)$.
    \item We have $f \times f'=(Y \times f') \circ (f \times X')$ and each of these is a $(G \times H)$-global equivalence by iv).
    \item Consider a graph subgroup $\Gamma_\phi \in \fat(K, H)$ associated to a homomorphism $\phi$. Then
    \[(\tel_i \eta^* f(V_i))^{\Gamma_\phi} = (\eta^*(\tel_i f(V_i)))^{\Gamma_\phi} = (\tel_i f(V_i))^{\Gamma_{\eta \circ \phi}}\]
    which is a weak homotopy equivalence. \qedhere
\end{enumerate}
\end{proof}

We now turn to the box product of $G$-orthogonal spaces, to check that it preserves $G$-global equivalences. The box product of orthogonal spaces is fully homotopical with respect to the global equivalences, that is, the box product of two global equivalences is a global equivalence, and this does not require any cofibrancy assumptions on the morphisms or the orthogonal spaces involved. Our goal is to check that the box product is also fully homotopical with respect to the $G$-global equivalences.

Given compact Lie groups $G$ and $H$, a $G$-orthogonal space $X$ and an $H$-orthogonal space $Y$, $X \boxtimes Y$ has a canonical action by $G \times H$, and so does $X \times Y$. The natural morphism $\rho_{X, Y}$ of Remark~\ref{remoperadboxtocat} is $(G \times H)$-equivariant. By \cite[Theorem~1.3.2 (i)]{global} this $\rho_{X, Y}$ is a global equivalence of underlying orthogonal spaces. We now adapt that proof to show that it is additionally a $(G \times H)$-global equivalence. We start with a technical lemma.

\begin{lemm}
\label{lemmLendofunctor}
Given $F \colon \Lcat \to \Lcat$ a continuous endofunctor, a natural transformation $\eta \colon Id \Rightarrow F$, and a $G$-orthogonal space $X$, the morphism $X \circ \eta \colon X \to X \circ F$ is a $G$-global equivalence.
\end{lemm}
\begin{proof}
We use the fact that for each compact Lie group $K$ and each $K$-representation $V$ the two embeddings
\[F(\eta_V), \eta_{F(V)}\colon F(V) \to F(F(V))\]
are homotopic relative to $\eta_V \colon V \to F(V)$  through $K$-equivariant linear isometric embeddings. This is shown in the proof of the equivalent result where $X$ is just an orthogonal space in~\cite[Theorem~1.1.10]{global}.

Given a compact Lie group $K$, a $K$-representation $V$, a graph subgroup $\Gamma \in \fat(K, G)$ and a lifting problem $(\alpha, \beta)$ as in the following diagram, the linear isometric embedding $\eta_V$ and the map $\beta$ together witness that $X \circ \eta$ is a $G$-global equivalence. 
\begin{equation*}
    % https://tikzcd.yichuanshen.de/#N4Igdg9gJgpgziAXAbVABwnAlgFyxMJZABgBpiBdUkANwEMAbAVxiRAB1206AnPRgAQARAHoMQAX1LpMufIRRkAjFVqMWbUeKkzseAkSXlV9Zq0QgAGgAoAagEoRACUnSQGPfMOkV1UxosbADE7e0cXHXdZfQVkACZjP3VzK2sQh3DXXTkDFATfNTM2YLTQsOdJVRgoAHN4IlAAMx4IAFskMhAcCCQEwoCOdkY0AAs6LJBmtqQjLp7EAGYkootOACMYHHHIqfbETu6Z5YGsAH1tN13e6kPF45SbTk26U4cK6gY6DYYABWivCw8LA1EY4CZXRB9W4AFnuxWsTy2r0yOxaeyWcyQsJADCwYBSUAgTDWDFY1BGMDoUDYkHxIBudCw4gstNYqOmiGxtwArHDAgj2M9TsB0vYJCjLmikBieXzUiFES8Mu8cV8YL9-rkQECQWCJBQJEA
\begin{tikzcd}
\partial D^l \arrow[r, "\alpha"] \arrow[d, "i_l", hook] & X(V)^\Gamma \arrow[d, "X(\eta_V)^\Gamma"'] \arrow[r, "X(\eta_V)^\Gamma"]              & X(F(V))^\Gamma \arrow[d, "X(\eta_{F(V)})^\Gamma"] \\
D^l \arrow[r, "\beta"]                            & X(F(V))^\Gamma \arrow[ru, Equal] \arrow[r, "X(F(\eta_V))^\Gamma"'] & X(F(F(V)))^\Gamma                           
\end{tikzcd}
\end{equation*}
The upper left trapezoid commutes by construction. For the lower right triangle, $F(\eta_V)$ and $\eta_{F(V)}$ are homotopic through $K$-equivariant linear isometric embeddings, therefore $X(F(\eta_V))$ and $X(\eta_{F(V)})$ are homotopic through $(K \times G)$-equivariant maps, and $X(F(\eta_V))^\Gamma$ and $X(\eta_{F(V)})^\Gamma$ are homotopic. Since the original homotopy was relative to $\eta_V$, and $\beta \circ i_l = X(\eta_V)^\Gamma \circ \alpha$, the obtained homotopy between $X(F(\eta_V))^\Gamma \circ \beta$ and $X(\eta_{F(V)})^\Gamma \circ \beta$ is relative $i_l$. Thus $X \circ \eta$ is a $G$-global equivalence.
\end{proof}

\begin{prop}
\label{proprhoGeq}
Given a $G$-orthogonal space $X$ and an $H$-orthogonal space $Y$, the morphism of $(G \times H)$-orthogonal spaces $\rho_{X, Y}$ is a $(G \times H)$-global equivalence.
\end{prop}
\begin{proof}
Consider the endofunctor $\sh \colon \Lcat \to \Lcat$ that sends $V$ to $V \oplus V$. We have two natural transformations $\iota_1, \iota_2 \colon Id \Rightarrow\sh$ given by the embeddings into the first and second factor respectively. We also denote by $\sh$ the functor of orthogonal spaces given by precomposing with $\sh$, $\sh(X)= X \circ \sh$.

The universal bimorphism $i$ that exhibits $X \boxtimes Y$ as the box product of $X$ and $Y$ gives a morphism of orthogonal spaces $\lambda \colon X \times Y \to \sh(X \boxtimes Y)$ through the maps
\[i_{V, V} \colon X(V) \times Y(V) \to (X \boxtimes Y)(V \oplus V) = (\sh(X \boxtimes Y))(V).\]
We need to check that $\lambda \circ \rho_{X, Y}$ and $\sh(\rho_{X, Y}) \circ \lambda$ in the following diagram
\begin{equation*}
    % https://tikzcd.yichuanshen.de/#N4Igdg9gJgpgziAXAbVABwnAlgFyxMJZABgBpiBdUkANwEMAbAVxiRAA0ACAHW4CMIADzwBbeJwCaIAL6l0mXPkIoAjOSq1GLNl16jxU2fOx4CRAEzrq9Zq0QhecABYAKXfyH64kgJQy5IBgmSkQAzFaatmyOru5evjIaMFAA5vBEoABmAE4QIkhkIDgQSGqR2va82U4QAPrA7KSS0v5ZufmIZcVIluV2DtwMdCJ8UHStIDl5PdTdiOF90dzOLlU19Y3NftIU0kA
\begin{tikzcd}
X \boxtimes Y \arrow[r, "{\rho_{X, Y}}"] & X \times Y \arrow[r, "\lambda"] & \sh(X \boxtimes Y) \arrow[r, "{\sh(\rho_{X, Y})}"] & \sh(X \times Y)
\end{tikzcd}
\end{equation*}
are $(G \times H)$-global equivalences, and then we can use the 2-out-of-6 property~\ref{lemmGglobal}~i) to obtain that $\rho_{X, Y}$ is a $(G \times H)$-global equivalence.

We have that $\sh(\rho_{X, Y}) \circ \lambda$ evaluated at $V$ is the same as the map associated to $\rho_{X, Y}$ at level $(V, V)$ given in (\ref{eqrhodef}) of Remark~\ref{remoperadboxtocat}, by the constructions of $\lambda$ and $\rho_{X, Y}$. This means that
\[\sh(\rho_{X, Y}) \circ \lambda= X(\iota_1)\times Y(\iota_2),\]
where each morphism on the right is a $G$-global equivalence or an $H$-global equivalence respectively by Lemma~\ref{lemmLendofunctor}, and so their product is a $(G \times H)$-global equivalence by Lemma~\ref{lemmGglobal}~v).

Next we use that $\lambda \circ \rho_{X, Y}$ is homotopic through $(G \times H)$-equivariant morphisms to $(X \boxtimes Y)(\iota_1)$, since the homotopy between them given in the proof of~\cite[Theorem~1.3.2 (i)]{global} is through $(G \times H)$-equivariant morphisms. Additionally $(X \boxtimes Y)(\iota_1)$ is a $(G \times H)$-global equivalence by Lemma~\ref{lemmLendofunctor}, so by Lemma~\ref{lemmGglobal}~iii) $\lambda \circ \rho_{X, Y}$ is a $(G \times H)$-global equivalence.
\end{proof}

\begin{coro}
\label{coroboxgglobal}
For a $G$-global equivalence $f\colon X \to Y$ and an $H$-global equivalence $f'\colon X' \to Y'$, the morphism $f \boxtimes f'$ is a $(G \times H)$-global equivalence. If $H=G$ then $f \boxtimes f'$ is a $G$-global equivalence. Therefore for any $X \in \GSpc$, the functor $X \boxtimes -$ preserves $G$-global equivalences.
\end{coro}
\begin{proof}
First,
\[\rho_{Y, Y'} \circ (f \boxtimes f') = (f \times f') \circ \rho_{X, X'},\]
and $\rho_{Y, Y'}$ and $\rho_{X, X'}$ are $(G \times H)$-global equivalences by Proposition~\ref{proprhoGeq}. Since $f \times f'$ is also a $(G \times H)$-global equivalence by Lemma~\ref{lemmGglobal}~v), by the 2-out-of-6 property so is $f \boxtimes f'$.

If $H=G$ by restricting along the diagonal homomorphism $\Delta \colon G \to G \times G$ and using Lemma~\ref{lemmGglobal}~vi) we obtain that $f \boxtimes f'$ is a $G$-global equivalence and therefore $X \boxtimes -$ preserves $G$-global equivalences.
\end{proof}

Now we proceed with a technical lemma which we use to prove the two subsequent propositions. The first one discusses what happens to $G$-global equivalences between $G$-free orthogonal spaces when taking orbits, if $G$ is finite. The second one shows that $G$-global equivalences are preserved by inducing from a finite subgroup.

\begin{lemm}
\label{lemmfreegweqorbit}
Let $H$ be a finite group and $K$ and $G$ compact Lie groups. Assume that we have equivariant maps of $(K \times G \times H)$-spaces $f \colon X \to Y$ and $g\colon Y \to Z$ such that $Z$ is Hausdorff and $H$-free. Then the map on orbits $f/H\colon X/H \to Y/H$ is an $\fat(K, G)$-equivalence if and only if $f$ is an $\fat(K, G \times H)$-equivalence.
\end{lemm}
\begin{proof}
First note that since $Z$ is $H$-free, so are $X$ and $Y$. For any graph subgroup $\Gamma_\phi \in \fat(K, G)$ given by a continuous homomorphism $\phi \colon L \to G$,~\cite[Proposition~B.17]{global} gives a natural homeomorphism for $X$, $Y$ and $Z$,
\[\coprod_{[\psi]} X^{\Gamma_\psi}/C(\psi) \to (X/H)^{\Gamma_\phi}.\]
The disjoint union on the left is indexed by the conjugacy classes of continuous homomorphisms $\psi \colon \Gamma_\phi \to H$. Here $C(\psi)$ denotes the centralizer of the image of $\psi$ in $H$. 

Fix a graph subgroup $\Gamma_\phi \in \fat(K, G)$. A homomorphism $\psi \colon \Gamma_\phi \to H$, as a subgroup of $K \times G \times H$, has elements $(k, \phi(k), \psi(k, \phi(k))$ for $k \in L$, so $\Gamma_\psi \in \fat(K, G \times H)$. Conversely, for a graph subgroup $\Gamma_\psi \in \fat(K, G \times H)$, let $\phi$ be the homomorphism $\pi_G \circ \psi \colon L \to G$ where $\pi_G \colon G \times H \to G$ is the projection. Then $\Gamma_\psi$ is a graph subgroup of $\Gamma_\phi \times H$, so that $\psi$ can be seen as a homomorphism $\Gamma_\phi \to H$.

Therefore the map on orbits $f/H$ is an $\fat(K, G)$-equivalence if and only if for each $\Gamma_\psi \in \fat(K, G \times H)$ the map $f^{\Gamma_\psi}/C(\psi)$ is a weak homotopy equivalence.

For each $\Gamma_\psi \in \fat(K, G \times H)$, the centralizer of the image of $\psi$, $C(\psi)\leq H$, is finite. Additionally, $Z^{\Gamma_\psi}$ is $C(\psi)$-free and a subspace of $Z$ so Hausdorff. Therefore the $C(\psi)$-action on $Z^{\Gamma_\psi}$ is properly discontinuous, and since $f^{\Gamma_\psi}$ and $g^{\Gamma_\psi}$ are $C(\psi)$-equivariant, the $C(\psi)$-actions on $X^{\Gamma_\psi}$ and $Y^{\Gamma_\psi}$ are also properly discontinuous.

This means that
\[X^{\Gamma_\psi} \to X^{\Gamma_\psi}/C(\psi) \; \text{and}\; Y^{\Gamma_\psi} \to Y^{\Gamma_\psi}/C(\psi)\]
are covering maps, and since $f^{\Gamma_\psi}$ is $C(\psi)$-equivariant, it induces a map of coverings. Then we consider the long exact sequence of homotopy groups for these covering maps. We obtain that $f^{\Gamma_\psi}/C(\psi)$ is a weak homotopy equivalence if and only if $f^{\Gamma_\psi}$ is a weak homotopy equivalence. For $\pi_n$ for $n \geqslant 2$ this can be seen by using the five lemma and for $\pi_0$ and $\pi_1$ it can be checked explicitly.

Thus we finally obtain that $f/H$ is an $\fat(K, G)$-equivalence if and only if $f$ is an $\fat(K, G \times H)$-equivalence.
\end{proof}

This next proposition is similar to~\cite[Lemma~8.1]{sagave}.

\begin{prop}
\label{propfreefiniteorbitspc}
Let $H$ be a finite group and $G$ a compact Lie group. Consider two morphisms of $(G \times H)$-orthogonal spaces $f\colon X \to Y$ and $g\colon Y \to Z$, where for $Z$ we know that for each inner product space $V$ the space $Z(V)$ is Hausdorff and $H$-free. Then $f/H\colon X/H \to Y/H$ is a $G$-global equivalence if and only if $f$ is a $(G \times H)$-global equivalence.
\end{prop}
\begin{proof}
By Proposition~\ref{propglobaltel} we know that $f/H\colon X/H \to Y/H$ is a $G$-global equivalence if and only if for each compact Lie group $K$ and exhaustive sequence of $K$-representations $\{V_i\}_{i \in \NN}$ the map
\[\tel_i f/H(V_i) \colon \tel_i X/H(V_i) \to \tel_i Y/H(V_i)\]
is an $\fat(K, G)$-equivalence.

Taking $H$-orbits commutes with colimits and product with $[0, 1]$, so it commutes with taking mapping telescopes. Therefore $\tel_i f/H(V_i) \cong \tel_i f(V_i)/H$. Now $f$ and $g$ induce $(K \times G \times H)$-equivariant maps on mapping telescopes
\begin{equation*}
    % https://tikzcd.yichuanshen.de/#N4Igdg9gJgpgziAXAbVABwnAlgFyxMJZABgBpiBdUkANwEMAbAVxiRAB12cYGB9LAAQANABQA1fgEoQAX1LpMufIRQBGclVqMWbTtz6CAmuKmz5IDNjwEiAJg3V6zVog5ce-AQC0TWaTM0YKABzeCJQADMAJwgAWyQyEBwIJHUtZ113AwEI339zaLjU6mSke3SdVz0PQWC82QoZIA
\begin{tikzcd}
\tel_i X(V_i) \arrow[r, "\tel_i f(V_i)"] & \tel_i Y(V_i) \arrow[r, "\tel_i g(V_i)"] & \tel_i Z(V_i).
\end{tikzcd}
\end{equation*}
Since each $Z(V)$ is Hausdorff and $H$-free, so is $\tel_i Z(V_i)$. By Proposition~\ref{propglobaltel} again $f$ is a $(G \times H)$-global equivalence if and only if $\tel_i f(V_i)$ is an $\fat(K, G \times H)$-equivalence for each $K$ and $\{V_i\}_{i \in \NN}$. By Lemma~\ref{lemmfreegweqorbit} $\tel_i f(V_i)/H$ is an $\fat(K, G)$-equivalence if and only if $\tel_i f(V_i)$ is an $\fat(K, G \times H)$-equivalence, which yields the result.
\end{proof}

\begin{prop}
\label{propinduceorbitspc}
For a compact Lie group $G$, a finite subgroup $H \leq G$, and an $H$-global equivalence $f\colon X \to Y$, the morphism $G \times_H f$ is a $G$-global equivalence.
\end{prop}
\begin{proof}
We first need to check that $G \times f$ is a $(G \times H)$-global equivalence, for the action where $G$ acts on the left on the $G$ factor, and $H$ acts both on the right on the $G$ factor and on the left on the $f$ factor.

Consider a compact Lie group $K$ and an exhaustive sequence of $K$-representations $\{V_i\}_{i \in \NN}$. The functor $G \times -$ commutes with colimits and the functor $- \times [0, 1]$, so it commutes with taking mapping telescopes. Therefore it suffices to check that $G \times \tel_i f(V_i)$ is an $\fat(K, G \times H)$-equivalence. 

For any graph subgroup $\Gamma_\phi\in\fat(K, G \times H)$, the image of $\Gamma_\phi$ under the projection
\[\pi_{K \times H} \colon K \times G \times H \to K \times H\]
is the graph subgroup $\Gamma_{\pi_H \circ \phi}$. Therefore
\[(\tel_i f(V_i))^{\Gamma_\phi}= (\tel_i f(V_i))^{\Gamma_{\pi_H \circ \phi}},\]
and the latter is a weak homotopy equivalence since $\tel_i f(V_i)$ is an $\fat(K, H)$-equivalence. Then
\[(G \times \tel_i f(V_i))^{\Gamma_\phi} = G^{\Gamma_\phi} \times \tel_i f(V_i)^{\Gamma_\phi}\]
is also a weak homotopy equivalence.

Lastly, the projection $G \times Y \to G$ is a $(G \times H)$-equivariant map, where again $G$ acts on $G$ on the left and $H$ acts on the right on $G$ and on the left on $Y$. With this action $G$ is $H$-free and Hausdorff, so by Proposition~\ref{propfreefiniteorbitspc}, $G \times_H f$ is a $G$-global equivalence.
\end{proof}
%Note that this statement can be turned into an "if and only if" with a small adjustment to the proof. Just need to note that for \phi' : K \to H, the unit of G is fixed by \Gamma_{\phi' \times \phi'} \in \fat(K, G \times H). So G has non-empty fixed points by this. This shows that if G \times f is an eq, f is an eq.

In Appendix~\ref{appendixmodel} we further explore some more technical aspects of $G$-orthogonal spaces. In particular, we construct the $G$-global model structure on $\GSpc$. The $G$-flat cofibrations are the cofibrations of this model structure. However for our admissibility results on operads in $\Spc$ we need to work with a bigger class of morphisms than that of the $G$-flat cofibrations. This is why we now introduce the class of $G$-$h$-cofibrations of $\GSpc$. In Appendix~\ref{appendixmodel} we also study the compatibility of $G$-global equivalences and $G$-$h$-cofibrations.

The category $\GSpc$ is tensored over $\Topcat$. Thus we can define what a homotopy of morphisms of $G$-orthogonal spaces is in the usual way using the interval. We can also similarly define what a $G$-homotopy equivalence of $G$-orthogonal spaces is.

\begin{defi}[$G$-$h$-cofibration]
\label{defiGhcof}
A morphism in $\GSpc$ is an \emph{$h$-cofibration} if it has the homotopy extension property. A morphism $f \colon X \to Y$ has the homotopy extension property if and only if there is a retraction in $\GSpc$ for the induced morphism
\[X \times [0, 1] \cup_X Y \to Y \times [0, 1].\]
We call these morphisms the \emph{$G$-$h$-cofibrations}.
\end{defi}

\begin{lemm}
\label{lemmGhcofclosed}
The class of $G$-$h$-cofibrations is closed under coproducts, transfinite compositions, cobase changes and retracts. Additionally each $G$-flat cofibration is a $G$-$h$-cofibration.
\end{lemm}
\begin{proof}
On a category tensored and cotensored over $\Topcat$ the $h$-cofibrations can be equivalently defined as those morphisms that have the left lifting property with respect to $\ev_0\colon X^{[0, 1]} \to X$ for all objects $X$ (see \cite[Definition~A.28]{global}. This shows the first part.

The $G$-level model structure for $\GSpc$ that we construct in Theorem~\ref{thmlevelmodel} is topological, and all objects are fibrant so by~\cite[Corollary~A.30 (iii)]{global} each $G$-flat cofibration is a $G$-$h$-cofibration.
\end{proof}

\begin{lemm}
\label{lemmpropertiesofhcofs}
Let $G$ be a compact Lie group. 
\begin{enumerate}[label = \roman*)]
    \item Consider a closed normal subgroup $H \leq G$. For a $G$-$h$-cofibration of $G$-orthogonal spaces $f\colon X \to Y$, the morphism on orbits $f/H \colon X/H \to Y/H$ is a $(G/H)$-$h$-cofibration.
    \item Consider a continuous homomorphism $\alpha \colon H \to G$ between compact Lie groups. For a $G$-$h$-cofibration of $G$-orthogonal spaces $f \colon X \to Y$, the morphism $\alpha^\ast f \colon \alpha^\ast(X) \to \alpha^\ast(Y)$ is an $H$-$h$-cofibration.
    \item Consider a compact Lie group $H$ and an $H$-orthogonal space $Z$. For a $G$-$h$-cofibration of $G$-orthogonal spaces $f\colon X \to Y$, the morphisms $Z \boxtimes f$ and $Z \times f$ are $(H \times G)$-$h$-cofibrations.
    \item Consider a closed subgroup $H \leq G$. For an $H$-$h$-cofibration of $H$-orthogonal spaces $f\colon X \to Y$, the morphism $G \times_H f \colon G \times_H X \to G \times_H Y$ is a $G$-$h$-cofibration.
\end{enumerate}
\end{lemm}
\begin{proof}
\begin{enumerate}[label = \roman*)]
    \item Suppose that we have a retraction in $\GSpc$
    \[r\colon Y \times [0, 1] \to X \times [0, 1] \cup_X Y.\]
    Taking orbits commutes with pushouts and the product with $[0, 1]$.
    
    Thus the morphism
    \[r/H \colon Y/H \times [0, 1] \to (X \times [0, 1] \cup_X Y)/H \cong X/H \times [0, 1] \cup_{X/H} Y/H,\]
    is the retraction that witnesses that $f/H$ is a $G/H$-$h$-cofibration.
    \item As before, the functor $\alpha^\ast$ commutes with pushouts and the product with $[0, 1]$, and the morphism $\alpha^\ast r$ is the retraction that witnesses that $\alpha^\ast f$ is an $H$-$h$-cofibration.
    \item The functors $Z \boxtimes -$ and $Z \times -$ commute with pushouts and the product with $[0, 1]$. The $(H \times G)$-equivariant morphisms $Z \boxtimes r$ and $Z \times r$ witness that $Z \boxtimes f$ and $Z \times f$ are $(H \times G)$-$h$-cofibrations respectively.
    \item This follows from i), ii) and iii).\qedhere
\end{enumerate}
\end{proof}

\section{Main Results for Operads in \texorpdfstring{$(\Spc, \boxtimes)$}{(Spc)}}
\label{sectionmainresults}
\subsection{Lifting the positive global model structure to \texorpdfstring{$\algop$}{Alg(O)}}
\label{subsectionTheoremI}

In this subsection, our goal is to prove Theorem~\ref{thmoperadmodelcat}, that states that any operad in $(\Spc, \boxtimes)$ is admissible. By this we mean that for any operad $\Op$ in $(\Spc, \boxtimes)$, the positive global model structure on $\Spc$ lifts through $U_{\algop} \colon \algop \to \Spc$ to give a model structure on $\algop$.

The condition that we need to check to obtain that any operad is admissible is the following.

\begin{cond}
\label{condkey}
For any $Z \in \SnSpc$ and any generating cofibration $i$ of $\Spc$, the morphism $Z \boxtimes_{\Sigma_n} i^{\square n}$ is an $h$-cofibration. For any $Z \in \SnSpc$ and any generating acyclic cofibration $j$ of $\Spc$, the morphism $Z \boxtimes_{\Sigma_n} j^{\square n}$ is an $h$-cofibration and a global equivalence.
\end{cond}

Note that the fact that we consider any possible $Z$ here is crucial in removing any cofibrancy assumptions on the operad. The symbol $\square$ denotes the pushout product of two morphisms and $i^{\square n}$ denotes the $n$-th iterated pushout product of $i$ with itself.

\begin{rem}
Condition~\ref{condkey} is strongly related to the property named \emph{symmetric h-monoidality} defined in \cite[Definition~4.2.4]{PavlovSchPowers}. Note that there are two different definitions of \emph{$h$-cofibrations} in the literature. The one used in \cite{PavlovSchPowers} and \cite{PavlovSchAdmissibility} was first given in \cite[Definition~1.1]{BataninBergerHcof}, and it is weaker than the definition we used for $\Spc$ and $\GSpc$.

The property of $\Spc$ being symmetric $h$-monoidal is not directly related to Condition~\ref{condkey}, however the spirit of it is the same. In \cite[Theorem~5.11]{PavlovSchAdmissibility} it is proven that in a category which satisfies certain technical assumptions and is symmetric $h$-monoidal each operad is admissible. Using Condition~\ref{condkey} instead of symmetric h-monoidality simplifies some arguments in the case of orthogonal spaces. Most of this subsection is dedicated to checking Condition~\ref{condkey}.
\end{rem}
\vspace{3mm}

First of all, in order to check Condition~\ref{condkey} we should give an explicit description of the generating (acyclic) cofibrations of the positive global model structure on $\Spc$. They can also be obtained from the generating (acyclic) cofibrations of the $G$-global model structure described in Theorem~\ref{thmlevelmodel} and Construction~\ref{constrKg} by setting $G=e$ and adding everywhere the requirement that $V \neq 0$ (this $V \neq 0$ requirement is the difference between the \emph{positive} global model structure and the global model structure).

\begin{rem}[Generating (acyclic) cofibrations of the positive global model structure]
\label{remgenerating}
In $\Spc$ we have a semifree orthogonal space for each compact Lie group $G$ and each $G$-representation $V$, given by $L_{G, V} = \Lcat(V, -)/G$. This semifree orthogonal space is the representing object for the functor $(-)(V)^G$ given by evaluating at $V$ and then taking $G$-fixed points.

Recall that $i_l$ denotes the boundary map $i_l \colon \partial D^l \rightarrow D^l$ in $\Topcat$, for each $l \geq 0$. Similarly let $j_l$ denote the inclusion $j_l \colon D^l \cong D^l \times \{0\} \to D^l \times [0, 1]$ for $l \geq 0$. We use this notation throughout the paper.

The morphisms in $\Is$, the generating cofibrations of the positive global model structure, are of the form $L_{G, V} \times i_l$ for a compact Lie group $G$, a faithful $G$-representation $V \neq 0$, and $l \geqslant 0$.

The generating acyclic cofibrations are $\Js \cup \Ks$, where morphisms in $\Js$ are of the form $L_{G, V} \times j_l$ for a compact Lie group $G$, a faithful $G$-representation $V \neq 0$, and $l \geqslant 0$. Morphisms in $\Ks$ are of the form $\iota_{\rho_{G, V, W}} \square i_l$ for a compact Lie group $G$, a faithful $G$-representation $V \neq 0$, a $G$-representation $W$, and $l \geqslant 0$. The morphism $\rho_{G, V, W} \colon L_{G, V \oplus W} \to L_{G, V}$ is given by restriction to $V$, and $\iota_{\rho_{G, V, W}}$ is the mapping cylinder inclusion of $\rho_{G, V, W}$.
\end{rem}
\vspace{3mm}

The generating acyclic cofibrations in the set $\Ks$ are more complex. Before checking Condition~\ref{condkey} for them, we need to prove several auxiliary lemmas. We first deal with the case of the morphisms in $\Is$ and $\Js$.

\begin{prop}
\label{propgeneratingcomplete1}
Let $K$ be a compact Lie group, $n\geqslant 1$, and let $Z$ be a $(K \times\Sigma_n)$-orthogonal space. For a generating cofibration $i \in \Is$, the morphism $Z\boxtimes_{\Sigma_n} i^{\square n}$ is a $K$-$h$-cofibration. For a generating acyclic cofibration $j$ in the set $\Js$, the morphism $Z\boxtimes_{\Sigma_n} j^{\square n}$ is a $K$-$h$-cofibration and a $K$-global equivalence.
\end{prop}
\begin{rem}
This Proposition is stated in more generality than Condition~\ref{condkey} so that we can also use it later in the proof of Theorem~\ref{thmquillen}.
\end{rem}
\begin{proof}
Let $i=L_{G, V} \times i_l \in \Is$. Then
\[Z\boxtimes i^{\square n}=Z\boxtimes L_{G, V}^{\boxtimes n} \times i_l^{\square n}\]
which is a $(K \times \Sigma_n)$-$h$-cofibration because $i_l^{\square n}$is a $\Sigma_n$-$h$-cofibration of $\Sigma_n$-spaces. Then by Lemma~\ref{lemmpropertiesofhcofs}~i) $Z\boxtimes_{\Sigma_n} i^{\square n}$ is a $K$-$h$-cofibration.

Let $j=L_{G, V} \times j_l \in \Js$. By the same argument as before we obtain that $Z\boxtimes_{\Sigma_n} j^{\square n}$ is a $K$-$h$-cofibration. Since $j_l^{\square n}$ is a $\Sigma_n$-homotopy equivalence of $\Sigma_n$-spaces, we also obtain that $Z\boxtimes L_{G, V}^{\boxtimes n} \times j_l^{\square n}$ is a $(K \times \Sigma_n)$-homotopy equivalence of orthogonal spaces, so $Z\boxtimes_{\Sigma_n} j^{\square n}$ is a $K$-homotopy equivalence. Therefore it is a $K$-level equivalence, and thus a $K$-global equivalence.
\end{proof}

\begin{prop}
\label{propboxtosquare}
Let $f\colon X \to Y$ be a morphism of orthogonal spaces such that for each $n \geqslant 1$ the morphism $f^{\boxtimes n}$ is a $\Sigma_n$-global equivalence, and such that for each $n \geqslant 1$ the morphism $f^{\square n}$ is a $\Sigma_n$-$h$-cofibration. Then for each $n \geqslant 1$ the morphism $f^{\square n}$ is a $\Sigma_n$-global equivalence.
\end{prop}
\begin{proof}
We use strong induction. For the base case, $f^{\square 1}= f^{\boxtimes 1}=f$ is a global equivalence.

Assume that the result holds for each $i < n$. We decompose $f^{\boxtimes n}$ by applying~\cite[Lemma~A.8]{sagave} to the pushout diagram given by $X = X \to Y$, obtaining
\begin{equation*}
    % https://tikzcd.yichuanshen.de/#N4Igdg9gJgpgziAXAbVABwnAlgFyxMJZABgBpiBdUkANwEMAbAVxiRAA0A9YAHR4CMIADzwBbeAAIwAXwC8ARU5gA+sQAUAMwCUIaaXSZc+QigCM5KrUYs2ilac069B7HgJEATBer1mrRCB8UBA4CM4gGK7GRADM3lZ+tkrKwGAAtKbSjrr6EYZuJsgALPG+NgF2ymCOshIAmtx8giJY4nBS0rqWMFAA5vBEoBoAThCiSGQgOBBImbkjY7PU00ge4QvjiF5TM4gx66ObcTtIJQnlIBqNPHAAjkx0wzAdXdJAA
\begin{tikzcd}
X^{\boxtimes n}=Q^n_0(f) \arrow[r] & Q^n_1(f) \arrow[r] & \dots \arrow[r] & Q^n_{n-1}(f) \arrow[r, "f^{\square n}"] & Q^n_n(f)= Y^{\boxtimes n}.
\end{tikzcd}
\end{equation*}
Note that the last step of this decomposition is precisely $f^{\square n}$. In the rest of this article we also use $Q^n_{n-1}(f)$ to denote the source of the $n$-fold pushout product of $f$, following the notation of \cite{sagave}, originally introduced in \cite[Section~12]{ElmendorfMandell}. 

For each step $1 \leqslant i < n$ there is a $\Sigma_n$-equivariant pushout diagram of orthogonal spaces
\begin{equation*}
    % https://tikzcd.yichuanshen.de/#N4Igdg9gJgpgziAXAbVABwnAlgFyxMJZAFgBoBGAXVJADcBDAGwFcYkQBFAPTAH0sAFADMAlCAC+pdJlz5CKAAwVqdJq3bc+wLAFpy44WMnTseAkSUKVDFm0QgAOg4DKWAOYBberzAACJ3ge8LzATq6e3sBgOlji-g6B8PHhXvxxABpcoQ4ARhAAHolwvtGx8XmFWEHF3FghuvqGElIgGKZyRGRWNDbq9mHuqX4BVcHZKZGlcSPVyYPeZZnZFUUlMdO5BasAmllOK6PFsRIqMFBu8ESgQgBOEB5I5DQ4EEgKxiC392-Pr4gATB8vg8Ab9HkC7iD-mDEABmHpqOyOFzzHzxIohAYREJTdGHObYxZ7TaVWa4-ZbfFCYlwACOzHoNxgvmOEO+cJh70o4iAA
\begin{tikzcd}
\Sigma_n \times_{\Sigma_{n-i} \times \Sigma_i} X^{\boxtimes n-i} \boxtimes Q^i_{i-1}(f) \arrow[d] \arrow[rrrr, "\Sigma_n \times_{\Sigma_{n-i} \times \Sigma_i} X^{\boxtimes n-i} \boxtimes f^{\square i}"] &  &  &  & \Sigma_n \times_{\Sigma_{n-i} \times \Sigma_i} X^{\boxtimes n-i} \boxtimes Y^{\boxtimes i} \arrow[d] \\
Q^n_{i-1}(f) \arrow[rrrr]                                                                                                                                                                                  &  &  &  & Q^n_i(f). \arrow[llllu, phantom, "\ulcorner", very near start]
\end{tikzcd}
\end{equation*}
By Corollary~\ref{coroboxgglobal} and the induction hypothesis $X^{\boxtimes n-i} \boxtimes f^{\square i}$ is a $(\Sigma_{n-i} \times \Sigma_i)$-global equivalence. Then by Proposition~\ref{propinduceorbitspc}
\[\Sigma_n \times_{\Sigma_{n-i} \times \Sigma_i} X^{\boxtimes n-i} \boxtimes f^{\square i}\]
is a $\Sigma_n$-global equivalence. Additionally by Lemma~\ref{lemmpropertiesofhcofs} it is a $\Sigma_n$-$h$-cofibration.

By Corollary~\ref{corocobase} this means that $Q^n_{i-1}(f) \to Q^n_i(f)$ is a $\Sigma_n$-global equivalence for each $1 \leqslant i < n$. Since so is $f^{\boxtimes n}$, by the 2-out-of-6 property for $\Sigma_n$-global equivalences $f^{\square n}$ is a $\Sigma_n$-global equivalence.
\end{proof}

\begin{lemm}
\label{lemmhtyeq}
For $f \colon X \to Y$ a homotopy equivalence between orthogonal spaces and $n \geqslant 1$, the morphism $f^{\boxtimes n}$ is a $\Sigma_n$-homotopy equivalence of orthogonal spaces, and therefore a $\Sigma_n$-global equivalence.
\end{lemm}
\begin{proof}
Let $g \colon Y \to X$ be a homotopy inverse to $f$ and $H$ a homotopy between $f \circ g$ and $Id_X$. Then for each $n \geqslant 1$, 
\[H^{\boxtimes n} \circ (X ^{\boxtimes n} \times \Delta) \colon X ^{\boxtimes n} \times [0, 1] \to Y^{\boxtimes n}\]
is a $\Sigma_n$-equivariant homotopy between $(f \circ g) ^{\boxtimes n}$ and $Id_X ^{\boxtimes n}$, where $\Delta \colon [0, 1] \to [0, 1]^n$ is the diagonal. The same can be done for $g \circ f$.

Then we obtain that $f^{\boxtimes n}$ is a $\Sigma_n$-homotopy equivalence. Therefore it is a $\Sigma_n$-level equivalence, and thus a $\Sigma_n$-global equivalence.
\end{proof}

\begin{prop}
\label{propiotasquarecof}
For each generating acyclic cofibration $k \in \Ks$, the morphism $k^{\square n}$ is a $\Sigma_n$-$h$-cofibration. Concretely, let $G$ be a compact Lie group, consider a faithful $G$-representation $V\neq 0$, a $G$-representation $W$, $n \geqslant 1$ and $l \geqslant 0$. Let $\iota_{\rho_{G, V, W}}$ be the morphism given in Remark~\ref{remgenerating}. Then $k^{\square n} =  (\iota_{\rho_{G, V, W}}\square i_l)^{\square n}$ is a $\Sigma_n$-$h$-cofibration. In particular, since $\iota_{\rho_{G, V, W}}\square i_0 = \iota_{\rho_{G, V, W}}$ we get that $\iota_{\rho_{G, V, W}}^{\square n}$ is a $\Sigma_n$-$h$-cofibration.
\end{prop}
\begin{proof}
Consider the decomposition of $\iota_{\rho_{G, V, W}}$ given by $g \circ i_{L_{G, V \oplus W}}$ in the following diagram

% https://q.uiver.app/?q=WzAsNSxbMiwwLCJNX3tcXHJob197RywgViwgV319Il0sWzAsMCwiTF97RywgViBcXG9wbHVzIFd9Il0sWzEsMCwiTF97RywgViBcXG9wbHVzIFd9IFxcY29wcm9kIExfe0csIFZ9Il0sWzEsMSwiTF97RywgViBcXG9wbHVzIFd9IFxcY29wcm9kIExfe0csIFYgXFxvcGx1cyBXfSJdLFsyLDEsIkxfe0csIFYgXFxvcGx1cyBXfSBcXHRpbWVzIFswLCAxXSJdLFsxLDIsImlfe0xfe0csIFYgXFxvcGx1cyBXfX0iXSxbMiwwLCJnIl0sWzMsNCwiTF97RywgViBcXG9wbHVzIFd9IFxcdGltZXMgaV8xIl0sWzMsMl0sWzQsMF1d
\[\begin{tikzcd}
	{L_{G, V \oplus W}} & {L_{G, V \oplus W} \amalg L_{G, V}} && {M_{\rho_{G, V, W}}} \\
	& {L_{G, V \oplus W} \amalg L_{G, V \oplus W}} && {L_{G, V \oplus W} \times [0, 1].}
	\arrow["{i_{L_{G, V \oplus W}}}", from=1-1, to=1-2]
	\arrow["g", from=1-2, to=1-4]
	\arrow["{L_{G, V \oplus W} \times i_1}", from=2-2, to=2-4]
	\arrow[from=2-2, to=1-2]
	\arrow[from=2-4, to=1-4]
	\arrow[from=1-4, to=2-2, phantom, "\llcorner", very near start]
\end{tikzcd}\]

We use results from \cite{GORCHINSKIY2016707} that deal with the interaction between the pushout product and operations on morphisms like composition and cobase change. The structure of this proof is convoluted because in general we cannot prove that the pushout product of two $\Sigma_n$-$h$-cofibrations is a $\Sigma_n$-$h$-cofibration. This forces us to carry the $- \square i_l^{\square n}$ term around. We also have to decompose $(\iota_{\rho_{G, V, W}}\square i_l)^{\square n}$ into simpler morphisms which we can prove to be $\Sigma_n$-$h$-cofibrations, using cobase changes and compositions, because these operations preserve the class of $\Sigma_n$-$h$-cofibrations.

By \cite[Lemma~15]{GORCHINSKIY2016707} we can decompose $\iota_{\rho_{G, V, W}}^{\square n}$ into $f_0 \circ f_1 \circ \dots \circ f_n$. Here for each $0 \leqslant j \leqslant n$, the morphism $f_j$ is a $\Sigma_n$-equivariant cobase change of
\[\Sigma_n \times_{\Sigma_{n-j} \times \Sigma_j} g^{\square n-j} \square i_{L_{G, V \oplus W}}^{\square j}.\]

By \cite[Lemma~17]{GORCHINSKIY2016707} we can write $\iota_{\rho_{G, V, W}}^{\square n} \square i_l^{\square n}$ as $f_0' \circ f_1' \circ \dots \circ f_n'$, where each $f_j'$ is a cobase change of $f_j \square i_l^{\square n}$.

We also have \cite[Lemma~13]{GORCHINSKIY2016707}, which states that for any $h_0$ if a morphism $h_1$ is a cobase change of $h_2$, then $h_1 \square h_0$ is a cobase change of $h_2 \square h_0$. By iterating this result, and using the associativity of the pushout product, we obtain that $g^{\square n-j}$ is a cobase change of $(L_{G, V \oplus W} \times i_1)^{\square n-j}$. Similarly $i_{L_{G, V \oplus W}}^{\square j}$ is a cobase change of $\emptyset \to L_{G, V}^{\boxtimes j}$. Furthermore these cobase changes can be checked to be through equivariant maps.

Note that $i_1^{\square n-j} \cong i_{n-j}$. For each $0 \leqslant j \leqslant n$ we can apply \cite[Lemma~13]{GORCHINSKIY2016707} again to obtain that $g^{\square n-j} \square i_{L_{G, V \oplus W}}^{\square j}$ is a $(\Sigma_{n-j} \times \Sigma_j)$-equivariant cobase change of
\[(L_{G, V \oplus W}^{\boxtimes n-j} \times i_1^{\square n-j}) \square (\emptyset \to L_{G, V}^{\boxtimes j}) \cong L_{G, V \oplus W}^{\boxtimes n-j} \boxtimes L_{G, V}^{\boxtimes j} \times i_{n-j}.\]

We want to check that for each $0 \leqslant j \leqslant n$,
\begin{equation}
\label{propiotasquarecofequation}
    (\Sigma_n \times_{\Sigma_{n-j} \times \Sigma_j} (L_{G, V \oplus W}^{\boxtimes n-j} \boxtimes L_{G, V}^{\boxtimes j} \times i_{n-j})) \square i_l^{\square n}
\end{equation}
is a $\Sigma_n$-$h$-cofibration. As a morphism of $(\Sigma_n \times \Sigma_n)$-orthogonal spaces this is isomorphic to
\[\Sigma_n \times_{\Sigma_{n-j} \times \Sigma_j} ((L_{G, V \oplus W}^{\boxtimes n-j} \boxtimes L_{G, V}^{\boxtimes j} \times i_{n-j}) \square i_l^{\square n}).\]

The map $i_{n-j} \square i_l^{\square n} \cong i_{n-j + l n}$ is a $(\Sigma_{n - j} \times \Sigma_n)$-$h$-cofibration of spaces. Therefore by Lemma~\ref{lemmpropertiesofhcofs}
\[L_{G, V \oplus W}^{\boxtimes n-j} \boxtimes L_{G, V}^{\boxtimes j} \times (i_{n-j} \square i_l^{\square n})\]
is a $(\Sigma_{n - j} \times \Sigma_j \times \Sigma_n)$-$h$-cofibration and the morphism (\ref{propiotasquarecofequation}) is a $\Sigma_n$-$h$-cofibration.

Recall that for each $0 \leqslant j \leqslant n$  the morphism $g^{\square n-j} \square i_{L_{G, V \oplus W}}^{\square j}$ is a cobase change of
\[L_{G, V \oplus W}^{\boxtimes n-j} \boxtimes L_{G, V}^{\boxtimes j} \times i_{n-j}.\]
Thus applying \cite[Lemma~13]{GORCHINSKIY2016707} again we obtain that $f_j \square i_l^{\square n}$ is a cobase change of (\ref{propiotasquarecofequation}) so it is also a $\Sigma_n$-$h$-cofibration. We are using the fact that the induction functor $\Sigma_n \times_{\Sigma_{n-j} \times \Sigma_j} -$ preserves pushouts.

Finally, each $f_j'$ was a cobase change of $f_j \square i_l^{\square n}$, so it is also a $\Sigma_n$-$h$-cofibration. Thus their composition
\[k^{\square n} =  (\iota_{\rho_{G, V, W}}\square i_l)^{\square n} = \iota_{\rho_{G, V, W}}^{\square n} \square i_l^{\square n}\]
is a $\Sigma_n$-$h$-cofibration. 
\end{proof}

\begin{prop}
\label{propiotasquareeq}
For each generating acyclic cofibration $k \in \Ks$, the morphism $k^{\square n}$ is a $\Sigma_n$-global equivalence.
\end{prop}
\begin{proof}
The generating acyclic cofibration $k$ is of the form $\iota_{\rho_{G, V, W}} \square i_l$ for a compact Lie group $G$, a faithful $G$-representation $V \neq 0$, a $G$-representation $W$, and $l \geqslant 0$.

We first check that
\[\rho_{G, V, W}^{\boxtimes n} \colon L_{G, V \oplus W}^{\boxtimes n} \to L_{G, V}^{\boxtimes n}\]
is a $\Sigma_n$-global equivalence. By~\cite[Example 1.3.3]{global} the orthogonal space $L_{G, V}^{\boxtimes n}$ is isomorphic to $L_{G^n, V^n}$ and thus closed. The $(\Sigma_n \wr G)$-representation $V^n$ is faithful, so for each compact Lie group $K$ by~\cite[Proposition~1.1.26 (ii)]{global} the restriction map
\[\rho_{V^n, W^n}(\Uk) \colon \Lcat(V^n \oplus W^n, \Uk) \to \Lcat(V^n, \Uk)\]
is a $(K \times (\Sigma_n \wr G))$-homotopy equivalence. Using that
\[\Lcat(V^n, \Uk) \cong \colim_{V' \in s(\Uk)} \Lcat(V^n, V')\]
and the fact that $-/G^n$ preserves colimits, we can obtain that
\[\rho_{G, V, W}^{\boxtimes n}(\Uk) \cong \rho_{V^n, W^n}(\Uk)/G^n\]
is a $(K \times \Sigma_n)$-homotopy equivalence. Therefore $\rho_{G, V, W}^{\boxtimes n}(\Uk)$ is an $\fat(K, \Sigma_n)$-equivalence, and so $\rho_{G, V, W}^{\boxtimes n}$ is a $\Sigma_n$-global equivalence.

Now we use the mapping cylinder to decompose $\rho_{G, V, W}$ as $\pi_{\rho_{G, V, W}} \circ \iota_{\rho_{G, V, W}}$. Since $\pi_{\rho_{G, V, W}}$ is a homotopy equivalence, by Lemma~\ref{lemmhtyeq} $\pi_{\rho_{G, V, W}}^{\boxtimes n}$ is a $\Sigma_n$-global equivalence, and then so is $\iota_{\rho_{G, V, W}}^{\boxtimes n}$.

We use Proposition~\ref{propboxtosquare} and Proposition~\ref{propiotasquarecof} to obtain that for each $n \geqslant 1$ the morphism $\iota_{\rho_{G, V, W}}^{\square n}$ is a $\Sigma_n$-global equivalence. Finally by Corollary~\ref{coropushoutproducteq} and Proposition~\ref{propiotasquarecof} again we get that $k^{\square n} = \iota_{\rho_{G, V, W}}^{\square n} \square i_l^{\square n}$ is a $\Sigma_n$-global equivalence.
\end{proof}

\begin{prop}
\label{propgeneratingcomplete2}
Let $n\geqslant 1$ and let $Z$ be a $\Sigma_n$-orthogonal space. For each generating acyclic cofibration $k \in \Ks$, the morphism $Z\boxtimes_{\Sigma_n} k^{\square n}$ is an $h$-cofibration and a global equivalence.
\end{prop}
\begin{proof}
First,
\[Z \boxtimes k^{\square n}= Z \boxtimes (\iota_{\rho_{G, V, W}} \square i_l)^{\square n}\]
is a $\Sigma_n$-$h$-cofibration by Proposition~\ref{propiotasquarecof}, and a $\Sigma_n$-global equivalence by Proposition~\ref{propiotasquareeq} and Corollary~\ref{coroboxgglobal}.

Consider the $\Sigma_n$-orthogonal space $L_{G, V}^{\boxtimes n} \cong L_{G^n, V^n}$. For each inner product space $U$ the group $G^n$ acts freely (since $V$ is faithful), smoothly and properly (since $G^n$ is compact) on $\Lcat(V^n, U)$, as long as $\lvert U \rvert \geqslant \lvert V \rvert ^n$. Therefore
\[L_{G^n, V^n}(U)=\Lcat(V^n, U)/G^n\]
is Hausdorff, and since $V^n$ is a faithful $\Sigma_n$-representation, $L_{G, V}^{\boxtimes n}(U) \cong L_{G^n, V^n}(U)$ is also $\Sigma_n$-free.

If $\lvert U \rvert < \lvert V \rvert ^n$ then $\Lcat(V^n, U)$ is empty, so in particular $L_{G^n, V^n}(U)$ is still Hausdorff and $\Sigma_n$-free.

The morphism $\pi_{\rho_{G, V, W}}$ induces a $\Sigma_n$-equivariant map of orthogonal spaces from the target of $Z \boxtimes k^{\square n}$ to $\ast \boxtimes L_{G, V}^{\boxtimes n} \times \ast$, and so by Proposition~\ref{propfreefiniteorbitspc} $Z \boxtimes_{\Sigma_n} k^{\square n}$ is a global equivalence. It is an $h$-cofibration by Lemma~\ref{lemmpropertiesofhcofs}~i).
\end{proof}

Note that the fact that $L_{G, V}^{\boxtimes n}(U)$ is $\Sigma_n$-free in this last proof is important. It lets us avoid the assumption that the components $\Opn$ of the operad $\Op$ are $\Sigma_n$-free in the following theorem.

\begin{thm}[Theorem~\ref{thmintrooperadmodelcat}]
\label{thmoperadmodelcat}
Let $\Op$ be any operad in $(\Spc, \boxtimes)$ the category of orthogonal spaces, with the positive global model structure and the symmetric monoidal structure given by the box product. Then there is a cofibrantly generated model category structure on $\algop$ the category of algebras over $\Op$, where the forgetful functor $U_{\algop}$ creates the weak equivalences and fibrations, and sends cofibrations in $\algop$ to $h$-cofibrations in $\Spc$.
\end{thm}

\begin{proof}
Let $\Hcof$ be the class of $h$-cofibrations in $\Spc$. It satisfies conditions a), b) and c) of Theorem~\ref{thmschsh2} by Lemma~\ref{lemmGhcofclosed}, Lemma~\ref{lemmsmallsources}, and Corollary~\ref{corotransfinitegglobal} respectively, with $G = e$ in all of them.

Consider a morphism $i \colon X\to Y$ in $\Spc$ and a pushout in $\algop$ of the form
\begin{equation*}
% https://tikzcd.yichuanshen.de/#N4Igdg9gJgpgziAXAbVABwnAlgFyxMJZABgBpiBdUkANwEMAbAVxiRADEB9YAHR4Fs6OABZoAXgGNgAQQYBzAL4AKPgHk0ASmUANDSAWl0mXPkIoAjOSq1GLNl14ChoyTPnK1m5QE09Bo9h4BERk5tb0zKyIINL6hiAYgaZElmHUEXbRAEL61jBQcvBEoABmAE4Q-EhkIDgQSJY2kfbcfIIi4lKyiio86lpKWH7x5ZXV1HVIAEz+IKNViFMT9YgAzOm2UXNxpRULjZNrs-NI67UrxAoUCkA
\begin{tikzcd}
F_{\mathpzc{Alg}(\Op)}(X) \arrow[r, "F_{\mathpzc{Alg}(\Op)}(i)"] \arrow[d] & F_{\mathpzc{Alg}(\Op)}(Y) \arrow[d] \\
A \arrow[r, "f"]                                                           & B. \arrow[lu, phantom, "\ulcorner", very near start]                       
\end{tikzcd}
\end{equation*}

We use the filtration of~\cite[Proposition~A.16]{sagave}, originally introduced in the proof of \cite[Theorem~12.4]{ElmendorfMandell}, with $k=0$, where $U_0^\Op = U_{\algop}$. We obtain a decomposition of $U_{\algop}(f)$ as the infinite composition of morphisms
\[f_n\colon P_{n-1} U_{\algop}(B) \to P_{n} U_{\algop}(B)\]
for $n \geqslant 1$, with $P_0 U_{\algop}(B)=U_{\algop}(A)$. For each $n \geqslant 1$, \cite[Proposition~A.16]{sagave} gives the following pushout in $\Spc$
\begin{equation*}
    % https://tikzcd.yichuanshen.de/#N4Igdg9gJgpgziAXAbVABwnAlgFyxMJZABgBoBGAXVJADcBDAGwFcYkQAFAfWCwFpyAXwAEAVS7EAegB1pAeTQAKAEIBKEINLpMufIRQAmCtTpNW7brxHipshSvWbt2PASJliJhizaIQ4rBl5JQBBVWFZACMIAA88AFt4HlkAZSwAc3j6LiwRAEVJLB5+IUUAa0ctEAwXPSIjTxpvcz8AoPswiOlouKxEuGTpNMzs3OFFAE1VSWAo2IT4YVyNExgodPgiUAAzACcIeKQjEBwIJGInED2DpABmGlOkISrrw8QyE7PEciazXyuchoXvs3j9PkdLq8jg8vvdTD52G07KFwnNev1BsMsjkRGUZrI4ABHZj0XYwJaCFaCIA
\begin{tikzcd}
U_n^\Op(A) \boxtimes_{\Sigma_n} Q^n_{n-1}(i) \arrow[d] \arrow[rr, "U_n^\Op(A) \boxtimes_{\Sigma_n} i^{\square n}"] &  & U_n^\Op(A) \boxtimes_{\Sigma_n} (Y)^{\boxtimes n} \arrow[d] \\
P_{n-1} U_0^\Op(B) \arrow[rr, "f_n"]                                                                               &  & P_{n} U_0^\Op(B). \arrow[llu, phantom, "\ulcorner", very near start]                               
\end{tikzcd}
\end{equation*}
Both the class $\Hcof$, and the class of morphisms in $\Spc$ which are both $h$-cofibrations and global equivalences, are closed under infinite composition and cobase change (see the results of Appendix~\ref{appendixmodel}). Propositions~\ref{propgeneratingcomplete1} and~\ref{propgeneratingcomplete2} imply that if $i$ is a generating cofibration $U_n^\Op(A) \boxtimes_{\Sigma_n} i^{\square n}$ is an $h$-cofibration, and if $i$ is a generating acyclic cofibration then $U_n^\Op(A) \boxtimes_{\Sigma_n} i^{\square n}$ is an $h$-cofibration and global equivalence. Therefore all the conditions of Theorem~\ref{thmschsh2} hold, and $\algop$ is a cofibrantly generated model category where $U_{\algop}$ creates the weak equivalences and fibrations. Furthermore $U_{\algop}$ sends cofibrations in $\algop$ to $h$-cofibrations in $\Spc$.
\end{proof}

\subsection{Characterizing which morphisms of operads induce Quillen equivalences}

We study now morphisms of operads and the associated functors between their respective categories of algebras, with the goal of classifying which morphisms of operads in orthogonal spaces induce Quillen equivalences between the respective categories of algebras.

Consider for now a general symmetric monoidal category $(\cat, \otimes, \ast)$, where the tensor product preserves all colimits in both variables. Let $g \colon \Op \to \Pop$ be a morphism of operads, understood as a morphism of monoids in $(\sob, \circ)$. The morphism $g$ induces an adjoint pair of functors
\[\adjunction{g!}{\algop}{\algpop}{g^*},\]
called the extension functor and the restriction functor respectively. The specific details can be found in~\cite[Section~3.3.5]{fresse} for example.

We use $\theta \colon \FF(\Op) \Rightarrow \FF(\Pop)$ to denote the natural transformation induced by $g$ between the monads $\FF(\Op)$ and $\FF(\Pop)$, associated to the operads $\Op$ and $\Pop$ respectively. For $X$ an algebra over $\Pop$, we use $\zeta_X \colon \FF(\Pop)(X) \to X$ to denote its structure map. Then $g^*(X)$ is just $X$ with structure map $\zeta_X \circ \theta_X$. Additionally, since $\Uop \circ g^\ast=\Upop$, we have that $g_! \circ \Fop$ is left adjoint to $\Upop$, so it is naturally isomorphic to $\Fpop$. This is the only information about the extension functor that we need.

For the proof of Theorem~\ref{thmquillen} we need to consider again the functors $U^\Op_k$ for $k \geqslant 0$ from~\cite[Proposition~10.1]{sagave}, originally introduced in the proof of \cite[Theorem~12.4]{ElmendorfMandell}. The functor $U^\Op_k$ goes from $\algop$ to $\Skcatfull$, the category of $\Sigma_k$-objects in $\cat$, and $U^\Op_0 = U_{\algop}$.

\begin{constr}
\label{construk}
Let $\Op$ and $\Pop$ be two operads, and let $g$ be a morphism of operads $g \colon \Op \to \Pop$. For a general $\Op$-algebra $X$, a $\Pop$-algebra $Y$, and a map of $\Op$-algebras $\gamma \colon X \to g^\ast(Y)$, we construct certain maps
\[g_{k, \gamma} \colon U^\Op_k (X) \to U^\Pop_k(Y)\]
in $\Skcatfull$ for each $k \geqslant 0$, in a way that is natural in $\gamma$ and preserves filtered colimits. It is important to note that the morphism $g_{k, \gamma}$ is not $U^\Op_k(\gamma)$ unless $k = 0$. In fact, $U^\Op_k \circ g^\ast$ is not $U^\Pop_k$ for $k \neq 0$, so $g_{k, \gamma}$ and $U^\Op_k(\gamma)$ do not have the same target for $k \neq 0$.

Consider the functors
\[\Op(-, k) \colon \cat \to \Skcatfull\]
constructed in~\cite[Section~A.9]{sagave}, which for an operad $\Op$ and $k \geqslant 0$ are given by 
\[\Op(X, k) = \coprod_{n \in \NN} \Op(n + k) \otimes_{\Sigma_n} X^{\otimes n}.\]
Note that $\Op(-, 0) = \FF(\Op)$. The construction of the functors $U_k^\Op$ in~\cite[Definition~A.10]{sagave} is given by the coequalizer
\begin{equation*}
    % https://tikzcd.yichuanshen.de/#N4Igdg9gJgpgziAXAbVABwnAlgFyxMJZAJgBoAGAXVJADcBDAGwFcYkQBVAPQB0eB5NAH0A1gAoAGgEoQAX1LpMufIRQBGCtTpNW7PoMmkABCJnzF2PASLlNNBizaIQ+tGNeGj5KcdNytMFAA5vBEoABmAE4QALZIGiA4EEjk5iBRsUhkicmIarKUskA
\begin{tikzcd}
{\Op(\Op(X, 0), k)} \ar[r,shift left=.75ex,"\partial_0"] \ar[r,shift right=.75ex,swap,"\partial_1"] & {\Op(X, k)} \arrow[r] & U^\Op_k(X).
\end{tikzcd}
\end{equation*}
The morphism of operads $g$ and the map of $\Op$-algebras $\gamma$ together induce a $\Sigma_k$-equivariant morphism of coequalizer diagrams. The induced morphism between the coequalizers $U^\Op_k(X)$ and $U^\Pop_k(Y)$ is our desired $g_{k, \gamma}$.

This construction preserves filtered colimits because tensor powers preserve them, and thus so do the functors $\Op(-, k)$.
\end{constr}\vspace{5mm}

Now we restrict ourselves to the case of operads in $(\Spc, \boxtimes)$, where we have the model structures on $\algop$ obtained in Theorem~\ref{thmoperadmodelcat}.

\begin{prop}
\label{propquillenadjunction}
For any morphism $g \colon \Op \to \Pop$ of operads in $(\Spc, \boxtimes)$, the restriction functor $g^\ast$ preserves and reflects fibrations and weak equivalences. Thus the pair $(g_!, g^\ast)$ is a Quillen adjunction.
\end{prop}
\begin{proof}
The functors $\Uop$ and $\Upop$ preserve and reflect fibrations and weak equivalences, and $\Uop \circ g^\ast=\Upop$.
\end{proof}

\begin{thm}[Theorem~\ref{thmintroquillen}]
\label{thmquillen}
Let $g \colon \Op \to \Pop$ be a morphism of operads in $(\Spc, \boxtimes)$ the category of orthogonal spaces, with the positive global model structure and the symmetric monoidal structure given by the box product. Then the induced adjunction $(g_!, g^\ast)$ is a Quillen equivalence between the respective categories of algebras if and only if for each $n \geqslant 0$ the morphism $g_n \colon\Opn \to \Popn$ is a $\Sigma_n$-global equivalence.
\end{thm}
\begin{proof}
The right adjoint $g^*$ preserves and reflects weak equivalences. Therefore the pair $(g_!, g^\ast)$ is a Quillen equivalence if and only if for each cofibrant $A \in \algop$ the unit $\eta_A \colon A \to g^\ast(g_!(A))$ is a weak equivalence in $\algop$, that is, a global equivalence of underlying orthogonal spaces (see for example \cite[Lemma~3.3]{Erdal_2019} for a proof).

We first assume that each $g_n$ is a $\Sigma_n$-global equivalence, and check that for each cofibrant $A \in \algop$ the unit $\eta_A \colon A \to g^\ast(g_!(A))$ is a global equivalence.
 
First assume that the cofibrant algebra $A$ is the colimit of a $\lambda$-sequence of morphisms $\{f_\beta\}_{\beta \in \lambda}$ beginning at $A_0=\Op_0$, for a limit ordinal $\lambda$. Note that the initial object of $\algop$ is $\Op_0$, since it is $F_{\algop}(\emptyset)$. Assume that each $f_\beta$ is a cobase change of a morphism of the form $\Fop(i_\beta)$, for $i_\beta \in \Is \; \; i_\beta \colon X_\beta \to Y_\beta$ a generating positive flat cofibration of orthogonal spaces. We want to check that $\Uop(\eta_A)$ is a global equivalence.

By evaluating the unit of the adjunction $\eta$ on the $\lambda$-sequence that gives rise to $A$, we obtain the following diagram
\vspace*{-3mm}
\begin{equation*}
    % https://tikzcd.yichuanshen.de/#N4Igdg9gJgpgziAXAbVABwnAlgFyxMJZABgBpiBdUkANwEMAbAVxiRAEEB9YgAgF4eAHUEB5NNxABfUuky58hFAEZyVWoxZsuSqTJAZseAkQBMq6vWatEIYVAg4E02YYVEAzOfVWtnYQCMYHDpdF3ljFAAWL0tNGzsHJz0DcMUSUiU1WOtbQTQACywAPWE6OBwACmECrE4AQgquYgBKZtD9OSM0lUyLDRzqwpLBMsrB2obtVvaUro8MrP62ceHRqrzC+sa-QUDg6ecO1wjkM17vONz7RxnOtyiFvp94wWuksLmUAFYYpZtb45pH7nbJsKRqGBQADm8CIoAAZgAnCAAWyQZBAOAgSBUFxy8IkhyRqJx1CxSDMeLYBJ0RORaMQlPJiHcdJJLLJ2MQ0SpNgJASCITZDIAbJykAAOJ6XFalcrrGpbGmtNrCyXixAAdjVWo1AE5pQMNsU5WNjUqdns6Ac9MSGT9MVyxbzcjVVvLxhaWqrbfSkDzmUpiDqDY6cUodRjmQ7QS9BZxgE1JO07aSw4hnbHcvHE5wIym-Rz05rDctBDmuALgsnJBRJEA
\begin{tikzcd}
A_0 = \Op_0 \arrow[rr, "f_0"] \arrow[d, "\eta_{A_0}"]      & & A_1 \arrow[rr, "f_1"] \arrow[d, "\eta_{A_1}"]              &  & \dots \arrow[r] & A_\beta \arrow[rr, "f_\beta"] \arrow[d, "\eta_{A_\beta}"]         &  & \dots  \\
g^\ast(g_!(A_0)) \arrow[rr, "g^\ast(g_!(f_0))"] & & g^\ast(g_!(A_1)) \arrow[rr, "g^\ast(g_!(f_1)))"] & & \dots \arrow[r] & g^\ast(g_!(A_\beta)) \arrow[rr, "g^\ast(g_!(f_\beta))"] & & \dots .
\end{tikzcd}
\end{equation*}
We apply $\Uop$ to the whole diagram. By Theorem~\ref{thmoperadmodelcat}, $\Uop$ sends cofibrations to $h$-cofibrations, so $\Uop(f_\beta)$ is an $h$-cofibration. The morphism $g_!(f_\beta)$ is a cofibration, so
\[\Uop(g^\ast(g_!(f_\beta)))=\Upop(g_!(f_\beta))\]
is also an $h$-cofibration. Since $\Uop$ preserves filtered colimits, $\Uop(\eta_A)$ is
\[\colim_{\beta \in \lambda} \Uop(\eta_{A_\beta}).\]

We have to check that each $\Uop(\eta_{A_\beta})$ is a global equivalence, and for this we follow the proof of the similar statement in~\cite[Lemma~9.13]{sagave}. We prove this by induction, but we in fact need to work with a stronger property. For each $\beta$ and each $k \geqslant 0$, let $g_{k, \beta}$ be the morphism $g_{k, \eta_{A_\beta}}$ given in Construction~\ref{construk}. We check by transfinite induction on $\beta$ that for each $k \geqslant 0$ the morphism
\[g_{k, \beta} \colon U^\Op_k(A_\beta) \to U^\Pop_k(g_!(A_\beta))\]
is a $\Sigma_k$-global equivalence. For $k=0$ this reduces to our desired result.

The base case concerns $A_0 = \Op_0 = F_{\algop}(\emptyset)$. By~\cite[Lemma~A.13]{sagave} the $\Sigma_k$-orthogonal space $U^\Op_k(F_{\algop}(\emptyset))$ is isomorphic to $\Op(\emptyset, k)$, and  $\Op(\emptyset, k)$ equals $\Op_k$. Similarly $g_!(F_{\algop}(\emptyset))$ is isomorphic to $F_{\algpop}(\emptyset)$, and then
\[U^\Pop_k(F_{\algpop}(\emptyset)) \cong \Pop(\emptyset, k) = \Pop_k,\]
and under these identifications, the morphism $g_{k, 0}$ corresponds to $g_k$, which is a $\Sigma_k$-global equivalence by the condition of the theorem. Remarkably, this is the only part of the proof where this condition is used.

Then we check the induction step for a successor ordinal $\beta + 1$. For this we use the filtration of~\cite[Proposition~A.16]{sagave}, originally introduced in the proof of \cite[Theorem~12.4]{ElmendorfMandell}, in the same way that it is used in the proof of~\cite[Lemma~9.13]{sagave}.
\tikzset{font={\fontsize{9pt}{12}\selectfont}}
\begin{equation}
\label{thmquillen1}
    % https://tikzcd.yichuanshen.de/#N4Igdg9gJgpgziAXAbVABwnAlgFyxMJZABgBpiBdUkANwEMAbAVxiRAFUA9AHW4Hk0AfQDWACgCCg3gCMYOOgEoAvADFBxAARdeAkRMHAZcuhoDUGgIwBfBSCul0mXPkIoL5KrUYs2ai1p5+ITFJQ25ZeTNLGzsHEAxsPAIiACYPanpmVkQQXigIHAR7R0SXIgBmdK8stl4AYwgGLABbAywNXiwwDu4AOV6rDTV27SC9UKNI82sFDSUAnWD9MIiTUxnYkudklDILT0yfHNGABQgl3jQACyxBAEJ9ScVlNU1T871Lm-vlp6iZ2zFeJOJKuZDufYZbzZEB+BbcM4XbjXW4PCbhYz-GyAuIJbZgyqQ6pHXLcBpNVrAdqdbq8fqDYbwxGfZHfNEGP7TbFzJkfMRfVG-DFTAGbYGlHbINJEw4wvIFIqeGBQADm8CIoAAZgAnCDNJBkEA4CBIAAsUJqORVBmEpB6qysYp1eqQ7iNJsQAFYgc79Yg0u6kAB2H26v2VQOIABsFpJ1uAtvtmK5TrDruoxqQKVDLv9GY95RzfqD+aQUaLSE9pcQIbivrN1e9dbTiENmcQ1goViAA
\begin{tikzcd}
U^\Op_k(A_\beta)=F_0 U^\Op_k(A_{\beta + 1}) \arrow[d, "{g_{k, \beta}}"] \arrow[r] & F_1 U^\Op_k(A_{\beta + 1}) \arrow[d] \arrow[r] & \dots \arrow[r] & \colim\limits_{j \in \NN} F_j U^\Op_k(A_{\beta + 1}) = U^\Op_k(A_{\beta +1}) \arrow[d, "{g_{k, \beta + 1}}"] \\
U^\Pop_k(g_!(A_\beta))=F_0 U^\Pop_k(g_!(A_{\beta + 1})) \arrow[r]           & F_1 U^\Pop_k(g_!(A_{\beta + 1})) \arrow[r]  & \dots \arrow[r]           & \colim\limits_{j \in \NN} F_j U^\Pop_k(g_!(A_{\beta + 1})) = U^\Pop_k(g_!(A_{\beta +1}))              
\end{tikzcd}
\end{equation}
\tikzset{font={\fontsize{11pt}{12}\selectfont}}
Assume that for each $k \geqslant 0$ the morphism $g_{k, \beta}$ is a $\Sigma_k$-global equivalence. Each horizontal map is a cobase change of
\[U_{j + k}^\Op(A_{\beta}) \boxtimes_{\Sigma_j} i_\beta^{\square j} \; \text{or} \; U_{j + k}^\Pop(g_!(A_{\beta})) \boxtimes_{\Sigma_j} i_\beta^{\square j},\]
which are $\Sigma_k$-$h$-cofibrations by Proposition~\ref{propgeneratingcomplete1}.

Each vertical map is obtained from the previous by the following morphism of pushout diagrams,
\begin{equation*}
    % https://tikzcd.yichuanshen.de/#N4Igdg9gJgpgziAXAbVABwnAlgFyxMJZABgBpiBdUkANwEMAbAVxiRADEB9YLAWgEYAvgAIAqgD0AOpIDyaTgGsAFAEFu0gEYwcdYQGphQgJQhBpdJlz5CKMvyq1GLNlx4CRE6QAUI85dLQACyxOAEJVdUktHX1DQSMTMwtsPAIifnIHemZWRBBRbixYhUEpWTQI4E1tOnjhTQgADzwAW3hIgGUsAHMWuk4sEQBFcRC3ISUQ6p1E8xAMFOt00ntqbOc8gp5i0u9fJQDgsMrp2oT6qKbW9qrJLt7+weERsb4Jqaia2eSrNJQAJkyaycuXyhR2ZTkJ0+OjqDWaWDacE6PT6AxESgAmpxTkZxLcNFdEfBhINTHMFr8bMhAatHDk2FsigYSmUfBVDiFwmoCTV4kYLoSEUiUQ90cIsTiYXQ8QSiUjSYJTA4YFBuvAiKAAGYAJwgLSQZBAOAgSCEc11+qQgONpsQABYkiBLQbEDaTYanS6kABmagexAAVi9etdfttZpDVrd-rtPqjrvtsaQwYogiAA
\begin{tikzcd}
F_{j-1} U^\Op_k(A_{\beta + 1}) \arrow[d] & U_{j + k}^\Op(A_{\beta}) \boxtimes_{\Sigma_j} Q^j_{j-1}(i_\beta) \arrow[rrr, "U_{j + k}^\Op(A_{\beta}) \boxtimes_{\Sigma_j} i_\beta^{\square j}"] \arrow[l] \arrow[d, "g_{j + k, \beta} \boxtimes_{\Sigma_j} Q^j_{j-1}(i_\beta)"] & & & U_{j + k}^\Op(A_{\beta}) \boxtimes_{\Sigma_j} (Y_\beta)^{\boxtimes j} \arrow[d, "g_{j + k, \beta} \boxtimes_{\Sigma_j} (Y_\beta)^{\boxtimes j}"'] \\
F_{j-1} U^\Pop_k(g_!(A_{\beta + 1}))  & U_{j + k}^\Pop(g_!(A_{\beta})) \boxtimes_{\Sigma_j} Q^j_{j-1}(i_\beta) \arrow[rrr, "U_{j + k}^\Pop(g_!(A_{\beta})) \boxtimes_{\Sigma_j} i_\beta^{\square j}"] \arrow[l]  & & & U_{j + k}^\Pop(g_!(A_{\beta})) \boxtimes_{\Sigma_j} (Y_\beta)^{\boxtimes j}.
\end{tikzcd}
\end{equation*}
By the induction hypothesis the morphism 
\[g_{j + k, \beta} \colon U_{j + k}^\Op(A_{\beta}) \to U_{j + k}^\Pop(g_!(A_{\beta}))\]
is a $\Sigma_{j + k}$-global equivalence. Here $Y_\beta=L_{G, V} \times i_l$, so we can project to $L_{G, V}^{\boxtimes j}$ and use Corollary~\ref{coroboxgglobal} and Proposition~\ref{propfreefiniteorbitspc} as in the proof of Proposition~\ref{propgeneratingcomplete2} to check that the two rightmost vertical maps are $\Sigma_k$-global equivalences.

Then we can use induction on $j$ and the Gluing~Lemma~\ref{lemmgluing} to obtain that each vertical map of (\ref{thmquillen1}) is also a $\Sigma_k$-global equivalence. Finally by Lemma~\ref{lemmtransfinitent} $g_{k, \beta + 1}$ is a $\Sigma_k$-global equivalence.

If $\beta$ is a limit ordinal, we just need to use Lemma~\ref{lemmtransfinitent}, and the fact that the construction of $g_{k, \beta}$ preserves filtered colimits.

We have proven that $g_{k, \beta}$ is a $\Sigma_k$-global equivalence for each $k$ and $\beta$. Setting $k=0$ we have our original intended result that $\Uop(\eta_{A_\beta})$ is a global equivalence for each $\beta$. By Lemma~\ref{lemmtransfinitent} with $G = e$ the morphism $U_{\algop}(\eta_A)$ is a global equivalence.

If $A \in \algop$ is cofibrant, then it is a retract of an algebra $A'$ of the kind we were considering at the beginning of this proof, and the unit $\eta_A$ is a retract of $\eta_{A'}$. Since retracts preserve weak equivalences, $\eta_A$ is a weak equivalence in $\algop$. Therefore $(g_!, g^\ast)$ is a Quillen equivalence.

\vspace{8mm}

We prove the other implication now. Assume that $(g_!, g^\ast)$ is a Quillen equivalence. We want to prove that for each $n \geqslant 0$ the morphism $g_n$ is a $\Sigma_n$-global equivalence. Consider the free orthogonal space $L_\RR=\Lcat(\RR, -)$, which is positively flat. Then $F_{\algop}(L_\RR)$ is cofibrant in $\algop$. Since $(g_!, g^\ast)$ is a Quillen equivalence the unit
\[\eta_{F_{\algop}(L_\RR)} \colon F_{\algop}(L_\RR) \to g^\ast(g_!(F_{\algop}(L_\RR)))\]
is a weak equivalence, so its underlying morphism of orthogonal spaces is a global equivalence.

The $\Pop$-algebra $g_!(F_{\algop}(L_\RR))$ is naturally isomorphic to $F_{\algpop}(L_\RR)$. After postcomposing $\eta_{F_{\algop}(L_\RR)}$ with $g^\ast$ of this isomorphism, we obtain a morphism
\[F_{\algop}(L_\RR) \to g^\ast(F_{\algpop}(L_\RR)),\]
whose underlying morphism of orthogonal spaces is precisely
\[\theta_{L_\RR} = \coprod_{n \in \NN} g_n \boxtimes_{\Sigma_n} L_\RR^{\boxtimes n}.\]

Since $\theta_{L_\RR}$ is a global equivalence, each $g_n \boxtimes_{\Sigma_n} L_\RR^{\boxtimes n}$ is a global equivalence. If $n=0$ we obtain that $g_0$ is a global equivalence. For each $n \geqslant 1$, $L_\RR^{\boxtimes n}\cong L_{\RR^n}$, and the orthogonal space $L_{\RR^n}$ is $\Sigma_n$-free and Hausdorff at each inner product space $V$. Thus by Proposition~\ref{propfreefiniteorbitspc} the morphism $g_n \boxtimes L_{\RR^n}$ is a $\Sigma_n$-global equivalence for each $n \geqslant 1$.

The morphisms $\rho_{\Opn, L_{\RR^n}}$ and $\rho_{\Popn, L_{\RR^n}}$ are $\Sigma_n$-global equivalences by Proposition~\ref{proprhoGeq} and Lemma~\ref{lemmGglobal}~vi). By the 2-out-of-6 property of $\Sigma_n$-global equivalences we obtain that $g_n \times L_{\RR^n}$ is a $\Sigma_n$-global equivalence.
\begin{equation*}
    % https://tikzcd.yichuanshen.de/#N4Igdg9gJgpgziAXAbVABwnAlgFyxMJZABgBpiBdUkANwEMAbAVxiRAB12B5NMAAk4AjCAA88AW3h8AMgH1gMUgPYAlFQD0wAXz4gtpdJlz5CKMgEYqtRizace-ThKlyFSzms069BkBmx4BETm5Fb0zKyIHOwAChC8ysJiWJJwMvKKyp7auvqGASbBpJbU4bZRnHEJTikuGe6qGjl6VjBQAObwRKAAZgBOEOJIZCA4EEgh1hF27H0AFhDy9rxKrpkeTVpaINQMdIIwDHEFQVF9WO1zOD69A0OIAEzUY0gAzKU2kdHzi8CV8WBVvUspttnkQP1BhNnuNEO8puUQO1ZI52M40msGtlvODIfcRi9Hh9plFkaikuj0m4QV5chQtEA
\begin{tikzcd}
{\Opn \boxtimes L_{\RR^n} } \arrow[d, "{\rho_{\Opn, L_{\RR^n}}}"'] \arrow[rr, "{g_n \boxtimes L_{\RR^n} }"] & & {\Popn \boxtimes L_{\RR^n} } \arrow[d, "{\rho_{\Popn, L_{\RR^n}}}"] \\
{\Opn \times L_{\RR^n} } \arrow[rr, "{g_n \times L_{\RR^n} }"]                                                 & & {\Popn \times L_{\RR^n} }                                             
\end{tikzcd}
\end{equation*}
By Proposition~\ref{propglobaltel} for each compact Lie group $K$ and each exhaustive sequence of $K$-representations $\{V_i\}_{i \in \NN}$ the map
\[\tel_i (g(V_i) \times L_{\RR^n}(V_i)) \colon \tel_i (\Op_n(V_i) \times L_{\RR^n}(V_i)) \to \tel_i (\Pop_n(V_i) \times L_{\RR^n}(V_i))\]
is an $\fat(K, \Sigma_n)$-equivalence. The canonical map
\[\tel_i (\Op_n(V_i) \times L_{\RR^n}(V_i)) \to (\tel_i \Op_n(V_i)) \times (\tel_j L_{\RR^n}(V_j))\]
is also an $\fat(K, \Sigma_n)$-equivalence, and the same holds for $\Pop_n$. Therefore $(\tel_i g_n(V_i)) \times (\tel_j L_{\RR^n}(V_j))$ is an $\fat(K, \Sigma_n)$-equivalence.

For each $\Gamma_\phi \in \fat(K, \Sigma_n)$, we can pull the $\Sigma_n$-action on $\RR_n$ through $\phi \colon H \to \Sigma_n$ to get an $H$-action. Then the $H$-representation $\RR^n$ embeds into some $K$-representation (see \cite[Theorem~III. 4.5]{diecklie}), which in turn embeds into some $V_i$, so $(\tel_j L_{\RR^n}(V_j))^{\Gamma_\phi}$ is non-empty. Thus $\tel_i g_n(V_i)$ is an $\fat(K, \Sigma_n)$-equivalence, and $g_n$ is a $\Sigma_n$-global equivalence for each $n \geqslant 0$.
\end{proof}

\begin{rem}
The previous theorem generalizes, in the setting of algebras over operads in $(\Spc, \boxtimes)$, the classical result that a morphism $g$ between cofibrant operads induces a Quillen equivalence if the underlying morphism of each $g_n$ is a weak equivalence (see \cite[12.5.A]{fresse} for example). For orthogonal spaces, and a morphism $g$ between operads which are \emph{not necessarily cofibrant}, by the previous theorem we require the stronger condition that each $g_n$ is not just a global equivalence, but also a $\Sigma_n$-global equivalence.

The question of which morphisms between more general operads induce Quillen equivalences was also answered in \cite[Theorem~7.5]{PavlovSchAdmissibility}. The key property there is whether the morphisms $g_n$ are \emph{symmetric flat} weak equivalences as defined in \cite[Definition~2.1 vii)]{PavlovSchAdmissibility}. However, $\Sigma_n$-global equivalences are not necessarily symmetric flat.
\end{rem}

\begin{rem}
\label{remcofibrantreplacement}
Given $\Op$ an operad in $(\Spc, \boxtimes)$, we could take a cofibrant replacement of it in the $J$-semi model category $\OpeSpc$ of operads in $(\Spc, \boxtimes)$, constructed in \cite[Theorem~3]{spitzweck2001operads}. This would be a cofibrant operad $\Op'$ and a morphism of operads $g \colon \Op' \to \Op$ such that each $g_n$ is a global equivalence. But as we just saw in Theorem~\ref{thmquillen}, this $g$ does not induce a Quillen equivalence between the categories of algebras of $\Op$ and $\Op'$ unless each $g_n$ is additionally a $\Sigma_n$-global equivalence. This means that simply taking a cofibrant replacement $\Op'$ in $\OpeSpc$ of an operad $\Op$, and looking at the model structure on $\mathpzc{Alg}(\Op')$ does not give the correct homotopy theory of the algebras over $\Op$.

Additionally, we cannot have a functor $F^c \colon \OpeSpc \to \OpeSpc$, with a natural transformation $\eta \colon F^c \Rightarrow Id_{\OpeSpc}$ such that each $\eta(\Op)_n$ is a $\Sigma_n$-global equivalence, and $F^c(\Op)$ is cofibrant in the $J$-semi model structure of \cite[Theorem~3]{spitzweck2001operads}. Assume that we had such a functor $F^c$, then consider a morphism of operads $g \colon \Op \to \Op'$ which satisfies that each $g_n$ is a global equivalence, but does not satisfy that each $g_n$ is a $\Sigma_n$-global equivalence. An example of such a morphism is given by the unique morphism from one of the naive global $E_\infty$-operads of Remark~\ref{remconstantEinfinity} to the terminal operad $\Comm$.

In that case each $F^c(g)_n$ would be a global equivalence by the 2-out-of-6 property, so $F^c(g)$ induces a Quillen equivalence between $\mathpzc{Alg}(F^c(\Op))$ and $\mathpzc{Alg}(F^c(\Op'))$ because $F^c(\Op)$ and $F^c(\Op')$ are cofibrant operads. The morphisms of operads $\eta(\Op)$ and $\eta(\Op')$ would also induce Quillen equivalences by Theorem~\ref{thmquillen}, But this would imply that $g$ induces a Quillen equivalence between the categories of algebras, which contradicts the only if part of Theorem~\ref{thmquillen}.

This means that in order to study the genuine homotopy theory of algebras over operads in $(\Spc, \boxtimes)$, we cannot restrict ourselves to looking only at cofibrant operads.
\end{rem}

\section{Global \texorpdfstring{$E_\infty$}{E-infinity}-operads}
\label{sectionEinfinity}

Let $\Comm$ be the terminal operad in $(\Spc, \boxtimes)$, where each $\Comm_n$ is the constant one-point orthogonal space. Algebras over $\Comm$ are precisely the commutative monoids in $\Spc$ with respect to the box product, which are called commutative orthogonal monoid spaces or ultra-commutative monoids in~\cite[Definition~1.4.14]{global}. The unit and multiplication maps imply that a commutative monoid in $(\Spc, \boxtimes)$ is precisely a lax symmetric monoidal functor $(\Lcat, \oplus) \to (\Topcat, \times)$.

\begin{defi}
A \emph{global $E_\infty$-operad} in $(\Spc, \boxtimes)$ is an operad $\Op$ in $(\Spc, \boxtimes)$ such that each $\Opn$ is $\Sigma_n$-globally equivalent to $\ast$ with the trivial $\Sigma_n$-action.
\end{defi}

\begin{rem}
By Theorem~\ref{thmquillen}, if $\Op$ is a global $E_\infty$-operad in $(\Spc, \boxtimes)$ and $g$ is the unique morphism of operads $g \colon \Op \to \Comm$, then the induced Quillen adjunction $(g_!, g^\ast)$ is a Quillen equivalence between $\algop$ and $\algcomm$, the category of ultra-commutative monoids. This justifies why we gave the previous definition of a global $E_\infty$-operad. 

Furthermore, the algebras over a global $E_\infty$-operad are endowed with plenty of additional structure, just like ultra-commutative monoids. It is also relatively simple to characterize when a given operad in $\Spc$ (like the ones constructed in Subsection~\ref{subsectionexamples}) is a global $E_\infty$-operad.
\end{rem}

\begin{prop}
\label{propeinftyzigzag}
Let $\Op$ be a global $E_\infty$-operad in $(\Spc, \boxtimes)$, and let $g \colon \Op \to \Comm$ be the unique morphism of operads. There is a homotopical functor $R \colon \algop \to \algcomm$ and a zigzag of natural weak equivalences between $g^\ast \circ R$ and the identity on $\algop$.

For $A \in \algop$, $R(A)$ is an ultra-commutative monoid, thus $R$ is a functor that rectifies algebras over global $E_\infty$-operads into ultra-commutative monoids.
\end{prop}
\begin{proof}
Let $C \colon \algop \to \algop$ be a cofibrant replacement functor in $\algop$ constructed via the small object argument, and let $\alpha \colon C \Rightarrow \id_{\algop}$ be the associated natural weak equivalence. Then $\Uop(\alpha_A)$ is a global equivalence for each $A \in \algop$. Furthermore, the adjunction unit for $C(A)$, the morphism $\eta_{C(A)} \colon C(A) \to g^\ast(g_!(C(A)))$, is a global equivalence in $\Spc$ because the right adjoint $g^\ast$ preserves and reflects weak equivalences (see \cite[Lemma~3.3]{Erdal_2019}). Then $R = g_! \circ C$ is the desired functor, and $\alpha$ and $\eta$ form the desired zigzag of natural weak equivalences.
\end{proof}

\begin{lemm}
\label{lemmadmissiblefortimes}
The operads $\Ld$ and $\KK$ constructed in Examples~\ref{exdisks} and~\ref{exsteiner} respectively are reduced ($\Ld_0= \KK_0 = \ast$). For each $n \geqslant 0$, the orthogonal spaces $\Ld_n$ and $\KK_{\:n}$ are closed, and for each $V \in \Lcat$, the $\Sigma_n$-spaces $\Ld_n(V)$ and $\KK_{\:n}(V)$ are $\Sigma_n$-free and Hausdorff.
\end{lemm}
\begin{proof}
This follows from the properties of the little disks operads $\Ld(V)$ and Steiner operads $\KK(V)$ for an inner product space $V$. By construction they are reduced, and for each $n \geqslant 0$ they are $\Sigma_n$-free and Hausdorff, so the same is true for $\Ld$ and $\KK$. For a linear isometric embedding $\psi \colon V \to W$, the maps $\Ld_n(\psi)$ and $\KK_{\:n}(\psi)$ are closed embeddings, so the operads $\Ld$ and $\KK$ are closed.
\end{proof}

We now give several examples of global $E_\infty$-operads. To check that $\Ld$ and $\KK$ are global $E_\infty$-operads we first need the following technical lemma.

\begin{lemm}
\label{lemmequivconf}
Let $K$ be a compact Lie group, $\Uk$ a $K$-universe (not necessarily complete), $L \leqslant K$, and $T$ an $L$-set. Let $\conf^L_T(\Uk)$ denote the space of $L$-equivariant $T$-configurations in $\Uk$, that is, $L$-equivariant embeddings of $T$ in $\Uk$. Then $\conf^L_T(\Uk)$ is either empty or contractible.
\end{lemm}
\begin{proof}
Decompose $\Uk$ as
\[\Uk \cong \bigoplus_{\lambda \in \Lambda} \bigoplus_{n \in \NN} \lambda \cong \bigoplus_{n \in \NN} \bigoplus_{\lambda \in \Lambda} \lambda = \bigoplus_{n \in \NN} U_n\]

where $\Lambda$ is a set of finite-dimensional irreducible $K$-representations.

Let $P$ be the linear isometric embedding
\[\bigoplus _{n \in \NN} U_n \to \bigoplus _{n \in \NN} U_n, \quad (u_0, u_1, \dots) \mapsto (0, u_0, u_1, \dots).\]
Then $P$ is a $K$-equivariant non-surjective linear isometric embedding.

We give a homotopy $H$ between the identity and
\[P \circ - \colon \conf^L_T(\Uk) \to \conf^L_T(\Uk),\]
the map given by postcomposition with $P$. For each $s \in [0, 1]$, $f \in \conf^L_T(\Uk)$ and $t \in T$ this homotopy $H$ is given by
\[H_s(f)(t)= (1-s)(f(t)) + s(P(f(t))).\]
It is easy to see that each $H_s(f)$ is $L$-equivariant. If $H_s(f)(t)= H_s(f)(t')$, arguing coordinatewise we see that $t=t'$.

Now assume that $\conf^L_T(\Uk)$ is non-empty, so that $f_0 \in \conf^L_T(\Uk)$. There is an $m \geqslant 0$ such that the image of $f_0$ is contained in $\bigoplus_{n \leqslant m} U_n$. Then 
\[H'_s(f)(t)= s(P^{\circ m+1}(f(t))) + (1-s)(f_0(t))\]
gives a homotopy between the constant map with value $f_0$ and the map
\[P^{\circ m+1} \circ - \colon \conf^L_T(\Uk) \to \conf^L_T(\Uk).\]

We can easily see that $H'_s(f)$ is $L$-equivariant, and we can check that it is an embedding by looking at the projection to $\bigoplus_{n \leqslant m} U_n$ and to its orthogonal complement separately.

With $H$ we can obtain a homotopy from the identity to $P^{\circ m+1} \circ -$, and combining that homotopy with $H'$ we obtain that $\conf^L_T(\Uk)$ is contractible.
\end{proof}

\begin{prop}
The operads $\Ld$ and $\KK$ in $(\Spc, \boxtimes)$ are global $E_\infty$-operads.

\end{prop}
\begin{proof}
Both operads $\Ld$ and $\KK$ are closed by Lemma~\ref{lemmadmissiblefortimes}. For each compact Lie group $K$ and each $K$-representation $V$, the $(K \times \Sigma_n)$-spaces $\Ld(V)_n$ and $\KK(V)_n$ are $(K \times \Sigma_n)$-homotopy equivalent to $\conf_n(V)$, the configuration space of $n$ points in $V$, where $K$ acts on $V$ and $\Sigma_n$ acts by permuting the points. This is \cite[Lemma~1.2]{MayGuillou} and \cite[Proposition~1.5]{MayGuillou} respectively\footnote{In the proof of \cite[Lemma~1.2]{MayGuillou}, one has to add a small condition to ensure that the little disks are contained in the unit disk and that the constructed maps are well defined.}.

We have that 
\[\conf_n(\Uk) \cong \colim_{V \in s(\Uk)} \conf_n(V).\]
Consider any graph subgroup $\Gamma_\phi \in \fat(K, \Sigma_n)$, with $\phi \colon H \to \Sigma_n$ and $H \leqslant K$. Let $T_\phi$ be the set with $n$ elements and an $H$-action given by $\phi$. Since taking fixed points commutes with filtered colimits along closed embeddings, and the poset $s(\Uk)$ has a cofinal subsequence, $\Ld_n(\Uk)^{\Gamma_\phi}$ and $\KK_{\:n}(\Uk)^{\Gamma_\phi}$ are weakly homotopy equivalent to
\[\conf_n(\Uk)^{\Gamma_\phi} \cong \conf^H_{T_\phi}(\Uk),\]
where $\conf^H_{T_\phi}(\Uk)$ is the space of $H$-equivariant $T_\phi$-configurations in $\Uk$, that is, $H$-equivariant embeddings of $T_\phi$ in $\Uk$. Since $\Uk$ is a complete universe, $\conf^H_{T_\phi}(\Uk)$ is non-empty, so by Lemma~\ref{lemmequivconf} it is contractible. Thus $\Ld_n$ and $\KK_{\:n}$ are $\Sigma_n$-globally contractible, and $\Ld$ and $\KK$ are global $E_\infty$-operads.
\end{proof}

Recall that $L_V = \Lcat(V, -)$ is the orthogonal space represented by $V$.

\begin{prop}
For any $V \in \Lcat$ with $V \neq 0$, the endomorphism operad $\End(L_V)$ in $(\Spc, \boxtimes)$ is a global $E_\infty$-operad.
\end{prop}
\begin{proof}
We have to check that $\Hom(L_V^{\boxtimes n}, L_V) \to \ast$ is a $\Sigma_n$-global equivalence.

Let $\Uk$ be a complete $K$-universe for $K$ a compact Lie group. Then the underlying $K$-space of $\End(L_V)_n$ is

\begin{align*}
    \Hom(L_V^{\boxtimes n}, L_V)(\Uk)= \colim\limits_{W \in s(\Uk)} \Hom(L_V^{\boxtimes n}, L_V)(W) \cong \colim\limits_{W \in s(\Uk)} \Spc(L_W, \Hom(L_V^{\boxtimes n}, L_V)) \cong  \\ \colim\limits_{W \in s(\Uk)} \Spc(L_W \boxtimes L_{V^n}, L_V) \cong \colim\limits_{W \in s(\Uk)} \Lcat(V, W \oplus V^n) \cong  \Lcat(V, \Uk  \oplus V^n).
\end{align*}

The first three isomorphisms are induced by a chain of isomorphisms for each $W$, and we have to check that they are natural in $W$. For the second isomorphism this holds by the naturality of the $\boxtimes$-$\Hom$ adjunction, and for the first and third because of the naturality of the enriched Yoneda lemma. By the same reason they are $(K \times \Sigma_n)$-equivariant.

For any $\Gamma \in \fat(K, \Sigma_n)$ we consider the $\Gamma$-fixed points
\[\Hom(L_V^{\boxtimes n}, L_V)(\Uk)^{\Gamma} \cong \Lcat(V, \Uk  \oplus V^n)^{\Gamma} \cong \Lcat(V, (\Uk  \oplus V^n)^{\Gamma}).\]
Since $\Uk$ is a complete $K$-universe, $(\Uk  \oplus V^n)^{K \times \Sigma_n}$ is infinite-dimensional, and thus so is $(\Uk  \oplus V^n)^{\Gamma}$, so
\[\Lcat(V, (\Uk  \oplus V^n)^{\Gamma}) \cong \Lcat(V, \RR^\infty) \simeq \ast.\qedhere\]
\end{proof}

\begin{rem}
For any $V \in \Lcat$ with $V \neq 0$, the orthogonal space $L_V$ is an algebra over the global $E_\infty$-operad $\End(L_V)$. However, $L_V$ cannot be given the structure of a commutative monoid over the box product (ultra-commutative monoid). In particular, $L_V(0)=\Lcat(V, 0) = \emptyset$, so there are no morphisms of orthogonal spaces from $\ast$ to $L_V$, and thus we cannot give it a unit.
\end{rem}

\begin{rem}
\label{remconstantEinfinity}
Let $\Op$ be an $E_\infty$-operad in $\Topcat$. This is an operad such that $\Opn$ is $\Sigma_n$-free and weakly contractible for each $n \geqslant 0$. Let $\overline{\Op}$ be the constant operad in orthogonal spaces associated to $\Op$, which is closed. The space $\overline{\Op}_n(\Uk)$ is just $\Opn$ with the trivial $K$-action, which means that for $n \geqslant 2$ and a graph subgroup $\Gamma \in \fat(K, \Sigma_n)$ not contained in $K \times \{e\}$, we have that $(\overline{\Op}_n(\Uk))^\Gamma = \emptyset$. Therefore $\overline{\Op}_n$ is not $\Sigma_n$-globally equivalent to $\ast$ for $n \geqslant 2$, and so the constant operad $\overline{\Op}$ is not a global $E_\infty$-operad.

A similar situation occurs in the classical world of equivariant operads. A non-equivariant $E_\infty$-operad given the trivial $G$-action is not a good example of an $E_\infty$-operad in $G$-spaces. Instead one wants to look at $E_\infty$-$G$-operads, the ones for which $\Opn$ is $\fat(G, \Sigma_n)$-equivalent to a point (first defined in \cite[Definition~VII.1.2]{ESHT}).
\end{rem}

\begin{rem}
In the $G$-equivariant setting, there is a hierarchy of non-equivalent operads between a non-equivariant operad given the trivial $G$-action and an $E_\infty$-$G$-operad. These in-between operads are called \emph{$N_\infty$-operads}, and were introduced in \cite{BLUMBERG2015658}. They codify various levels of commutativity, by imposing the existence of certain additive transfers/multiplicative norms, which exist for commutative monoids in $G$-spaces and commutative $G$-ring spectra respectively.

In the global setting, there is also a hierarchy of operads between the naive global $E_\infty$-operads of Remark~\ref{remconstantEinfinity} and the global $E_\infty$-operads. These operads in orthogonal spaces are the global analogs of $N_\infty$-operads. A classification of them will appear in a separate article \cite{globalNin}.
\end{rem}

\appendix
\section{More about \texorpdfstring{$G$}{G}-orthogonal spaces}
\label{appendixmodel}

We now construct the $G$-global model structure on the category of $G$-orthogonal spaces, for $G$ a compact Lie group. The process is the same as the one used in \cite[Section~1.2]{global} to construct the global model structure on $\Spc$, and most of the proofs are almost identical. We first construct a level model structure on $\GSpc$ applying the results from \cite[Appendix C]{global}, and then we consider the left Bousfield localization with the $G$-global equivalences as the weak equivalences. We include here several technical results needed to construct this $G$-global model structure, or used in the main part of this paper.

In this appendix we assume the definitions and results of Section~\ref{sectionGorthogonal}.

\subsection{\texorpdfstring{$G$}{G}-level model structure}

For each compact Lie group $G$ there is an isomorphism of $\Topcat$-enriched categories
\begin{equation*}
    \Fu(\Lcat \times \Gcat, \Topcat) \cong \Fu(\Gcat , \Fu(\Lcat, \Topcat)) = \Gcat\text{-}\Spc.
\end{equation*}

We have that $\dat = \Lcat \times \Gcat$ is a skeletally small $\Topcat$-enriched category. On $\dat$ we have a dimension function on the objects $\lvert - \rvert$ given by the dimension of the inner product space of $\Lcat$. This function satisfies that if $\lvert V \rvert > \lvert W \rvert$ then $\dat(V, W)= \emptyset$ and if $\lvert V \rvert = \lvert W \rvert$ then $V$ and $W$ are isomorphic in $\dat$. We need to choose for each $m \geqslant 0$ an object of $\dat$ of dimension $m$, and we pick $\RR^m$. 

As input to obtain the $G$-level model structure, we have to consider for each $m \geqslant 0$ a model structure on the category of spaces with an action by $\dat(\RR^m, \RR^m)= O(m) \times G$. We take the model structure given by the graph subgroups, the $\fat(O(m), G)$-projective model structure, where a morphism $f$ of $\OmGTopcat$ is a weak equivalence (respectively a fibration) if and only if $f^\Gamma$ is a weak homotopy equivalence (respectively a Serre fibration) for each $\Gamma \in \fat(O(m), G)$. Recall that $\fat(O(m), G)$ is the set of graph subgroups of $O(m) \times G$, those $\Gamma \in O(m) \times G$ such that $\Gamma \cap (\{e_{O(m)}\} \times G) = \{e_{O(m) \times G}\}$

This $\fat(O(m), G)$-projective model structure is proper, cofibrantly generated, and topological. See for example \cite[Proposition~B.7]{global} for the construction. We call the weak equivalences of this model structure the $\fat(O(m), G)$-equivalences, and we do the same for the fibrations, acyclic fibrations, cofibrations and acyclic cofibrations. We can use the results from \cite[Appendix C]{global} to construct a level model structure on $\GSpc$ based on the graph subgroups. For this, the following consistency condition, described in~\cite[Definition~C.22]{global}, has to be satisfied.

\begin{lemm}[Consistency condition]
For each $m, n \geqslant 0$ and each $\fat(O(m), G)$-acyclic cofibration $i$, the morphism
\[(\Lcat(\RR^m, \RR^{m+n})\times G) \times_{O(m)\times G} i\]
is an $\fat(O(m + n), G)$-acyclic cofibration.
\end{lemm}
\begin{proof}
The functor
\[(\Lcat(\RR^m, \RR^{m+n})\times G) \times_{O(m)\times G} -\]
is a left adjoint to the functor given by
\[\Map(\Lcat(\RR^m, \RR^{m+n})\times G, -)^{O(m + n)\times G}.\]
Therefore we only need to check that it sends the generating acyclic cofibrations to acyclic cofibrations. 

The generating acyclic cofibrations of the $\fat(O(m), G)$-projective model structure are of the form $((O(m) \times G)/\Gamma)\times j_l$, for $\Gamma \in \fat(O(m), G)$ and $l \geqslant 0$. Then the functor mentioned at the beginning takes this generating acyclic cofibration to
\[((\Lcat(\RR^m, \RR^{m+n})\times G)/\Gamma) \times j_l.\]

The left $G$-action on $G$ is free, and because $\Gamma$ is a graph subgroup and the $O(m)$-action on $\Lcat(\RR^m, \RR^{m+n})$ is free, the left $G$-action on $(\Lcat(\RR^m, \RR^{m+n})\times G)/\Gamma$ is also free. We consider now $\Lcat(\RR^m, \RR^{m+n})\times G$ as an $(O(m+n) \times G \times O(m) \times G)$-space, where the component $O(m+n) \times G$ acts on the left, and $O(m) \times G$ originally acts on the right so we precompose with $(-)^{-1}$ to obtain a left action. The space $\Lcat(\RR^m, \RR^{m+n})\times G$ is a smooth $(O(m+n) \times G \times O(m) \times G)$-manifold. Illman's theorem \cite[7.2]{Illman} provides an $(O(m+n) \times G \times O(m) \times G)$-equivariant triangulation, so $\Lcat(\RR^m, \RR^{m+n})\times G$ is cofibrant in the projective model structure with respect to all subgroups of $O(m+n) \times G \times O(m) \times G$. 

By \cite[B.14 (i)]{global} and \cite[B.14 (iii)]{global}, $(\Lcat(\RR^m, \RR^{m+n})\times G)/\Gamma$ is cofibrant in the projective model structure with respect to all subgroups of $O(m+n) \times G$. This in particular means that it is a retract of a generalized $(O(m+n) \times G)$-CW-complex. Each cell $(O(m+n) \times G)/\Delta \times D^{l'}$ for a subgroup $\Delta \leqslant O(m+n) \times G$ and $l' \geqslant 0$ that appears in this equivariant CW-structure induces a $(O(m+n) \times G)$-equivariant map
\[f \colon (O(m+n) \times G)/\Delta \times D^{l'} \to (\Lcat(\RR^m, \RR^{m+n})\times G)/\Gamma.\]
Since the target of $f$ is $G$-free, so is the source, and this implies that $\Delta$ is a graph subgroup. As only graph subgroups can appear in this CW-structure, $(\Lcat(\RR^m, \RR^{m+n})\times G)/\Gamma$ is also cofibrant in the $\fat(O(m + n), G)$-projective model structure. Recall that the $\fat(O(m + n), G)$-projective model structure is topological. Thus $((\Lcat(\RR^m, \RR^{m+n})\times G)/\Gamma) \times j_l$ is the product of a cofibrant object with an acyclic cofibration of $\Topcat$, so it is an acyclic cofibration.
\end{proof}

\begin{thm}[$G$-level model structure]
\label{thmlevelmodel}
There is a topological cofibrantly generated model structure on the category $\GSpc$ of $G$-orthogonal spaces, which we call the $G$-level model structure. The weak equivalences (respectively the fibrations) are those morphisms $f$ such that for each $m\geqslant0$ and each graph subgroup $\Gamma \in \fat(O(m), G)$, the map $f(\RR^m)^\Gamma$ is a weak homotopy equivalence (respectively a Serre fibration). We call these respectively the \emph{$G$-level equivalences} and the \emph{$G$-level fibrations}.

A set of generating cofibrations is
\[\Is_G= \{\,((\Lcat(\RR^m, -)\times G)/\Gamma) \times i_l \mid m,l \geqslant 0, \Gamma \in \fat(O(m), G) \,\}.\]
A set of generating acyclic cofibrations is
\[\Js_G = \{\,((\Lcat(\RR^m, -)\times G)/\Gamma) \times j_l \mid m,l \geqslant 0, \Gamma \in \fat(O(m), G) \,\}.\]
We call the cofibrations of this model structure the \emph{$G$-flat cofibrations}.
\end{thm}
\begin{proof}
Such a model structure exists by~\cite[Proposition~C.23 (i)]{global}. It is cofibrantly generated by~\cite[Proposition~C.23 (iii)]{global} because each of the chosen model structures on $\OmGTopcat$ is cofibrantly generated.

The functor
\[(-)(\RR^m) \colon \GSpc \to \OmGTopcat\]
given by evaluation at $\RR^m$ has a left adjoint, which we denote by $F_m$, and it is given by
\[F_m(A) = (\Lcat(\RR^m, -) \times G) \times_{O(m) \times G} A.\]
The generating cofibrations obtained from \cite[Proposition~C.23 (iii)]{global} are those of the form $F_m(i)$ where $i$ is a generating cofibration of $\OmGTopcat$, which are of the form $((O(m) \times G)/\Gamma) \times i_l$ for $\Gamma \in \fat(O(m), G)$ and $l \geqslant 0$. Similarly the generating acyclic cofibrations are of the form $F_m(j)$ for $j$ a generating acyclic cofibration of $\OmGTopcat$.

Each $G$-orthogonal space of the form $(\Lcat(\RR^m, -)\times G)/\Gamma$ is cofibrant in this $G$-level model structure, because $F_m(((O(m) \times G)/\Gamma) \times i_0)$ is a generating cofibration. Using~\cite[Proposition~B.5]{global} we obtain that this model structure is topological, taking $\mathpzc{G}$ and $\mathpzc{Z}$ in the statement of that proposition to be
\[\mathpzc{G}= \{\, (\Lcat(\RR^m, -)\times G)/\Gamma  \mid m \geqslant 0, \Gamma \in \fat(O(m), G) \,\} \; \text{and} \; \mathpzc{Z}= \emptyset.\qedhere\]\end{proof}

Note that we should call this model structure on $\GSpc$ the "$G$-graph level model structure" to distinguish it from other possible model structures on $\GSpc$. In particular, there is at least the level model structure that we would obtain by considering all subgroups of $O(m) \times G$, and not just the graph subgroups. There is also a projective model structure on $\Fu(\Gcat , \Spc)$. However since we never talk about these two other model structures on $\GSpc$, we omit the adjective "graph" everywhere.

\begin{lemm}
\label{lemmlevelglobal}
If $f \colon X \to Y$ is a $G$-level equivalence then for any compact Lie group $K$ and any faithful $K$-representation $V$ the map $f(V)$ is an $\fat(K, G)$-equivalence. In particular, $f$ is also a $G$-global equivalence.
\end{lemm}
\begin{proof}
As a finite-dimensional inner product space, $V$ is isomorphic to $\RR^l$ for some $l \geqslant 0$. Let $\Gamma \in \fat(K, G)$ be a graph subgroup. Then the image of $\Gamma$ under the homomorphism $K \times G \to O(\RR^l) \times G$ induced by said isomorphism is a graph subgroup $\Gamma'$. Then $X(V)^\Gamma$ is naturally (on $X$) homeomorphic to $X(\RR^l)^{\Gamma'}$. Since $f(\RR^l)^{\Gamma'}$ is a weak homotopy equivalence, so is $f(V)^\Gamma$.
\end{proof}

\begin{rem}
\label{remGsemifree}
For an inner product space $V$ and a closed subgroup $H \leqslant O(V) \times G$, the $G$-orthogonal space
\[\dat(V, -)/H=(\Lcat(V, -) \times G)/H,\]
which we denote by $L_{H, V ;G}$, is special. It has a certain freeness condition, namely it is the representing object for the functor $(-)(V)^H$ given by evaluating at $V$ and then taking $H$-fixed points. We refer to these as the \emph{semifree $G$-orthogonal spaces}, since they have the same property as the semifree orthogonal spaces $L_{H, V}$.

Explicitly the natural isomorphism between the functors
\[\GSpc(L_{H, V ;G}, -), \: (-)(V)^H \: \colon \GSpc \to \Topcat\]
is given by $f \mapsto f(V)([\id_V, e])$. The inverse isomorphism is given on $Y \in \GSpc$ by sending a point $y_0 \in Y(V)^H$to the morphism of $G$-orthogonal spaces $f$ given by
\begin{alignat*}{3}
    (\Lcat(V, W) \times G)/H     \to & Y(W) \\ [\psi, g]  \mapsto & Y(\psi)(g y_0).
\end{alignat*}
\end{rem}
\vspace{4mm}

Analogously to the case of the semifree orthogonal spaces, the box product of a semifree $G$-orthogonal space and a semifree $K$-orthogonal space has a nice structure. As a $(G \times K)$-orthogonal space it is isomorphic to a semifree $(G \times K)$-orthogonal space, and this can be deduced from the result for orthogonal spaces. Note however that the box product of two semifree $G$-orthogonal spaces with the $G$-action given by restriction along the diagonal is not a semifree $G$-orthogonal space in general.

\begin{lemm}
\label{lemmgsemibox}
For compact Lie groups $G$ and $K$, inner product spaces $V$ and $V'$, and closed subgroups $H \leqslant O(V) \times G$ and $H' \leqslant O(V') \times K$, we have that $L_{H, V ; G} \boxtimes L_{H', V'; K}$ is isomorphic as a $(G \times K)$-orthogonal space to $L_{H \times H', V \oplus V'; G \times K}$.
\end{lemm}
\begin{proof}
Since the box product preserves colimits in both variables, we have that
\begin{alignat*}{2}
&(\Lcat(V, -) \times G)/H \boxtimes (\Lcat(V', -) \times K)/H' \cong ((\Lcat(V, -) \times G)\boxtimes (\Lcat(V', -) \times K))/(H \times H')\\
& \cong (\Lcat(V \oplus V', -) \times G \times K)/(H \times H').
\end{alignat*}
Here we also used the isomorphism $\Lcat(V, -) \boxtimes \Lcat(V', -) \cong \Lcat(V \oplus V', -)$ from \cite[Example 1.3.3]{global} and its naturality on $V$ and $V'$.
\end{proof}

\begin{lemm}
\label{lemmpushoutprodgflat}
The pushout product of a $G$-flat cofibration (recall that these are the cofibrations of the $G$-level model structure) and a $K$-flat cofibration is a $(G \times K)$-flat cofibration.
\end{lemm}
\begin{proof}
Given a generating $G$-flat cofibration $f=L_{\Gamma, \RR^m ; G} \times i_l$ and a generating $K$-flat cofibration $g=L_{\Gamma', \RR^n ; K} \times i_k$, their pushout product is by Lemma~\ref{lemmgsemibox} isomorphic to
\[L_{\Gamma \times \Gamma', \RR^{m + n}; G \times K} \times (i_l \square i_k)\]
as a morphism of $(G \times K)$-orthogonal spaces. The subgroup
\[\Gamma \times \Gamma' \leqslant O(m) \times O(n) \times G \times K \leqslant O(m + n) \times G \times K\]
is a graph subgroup because both $\Gamma$ and $\Gamma'$ are graph subgroups. Additionally $i_l \square i_k$ is homeomorphic to $i_{l + k}$, and so $f \square g$ is a generating $(G \times K)$-flat cofibration.

Since the box product of orthogonal spaces is closed, \cite[Lemma~4.2.4]{hovey2007model} implies that the pushout product of a $G$-flat cofibration and a $K$-flat cofibration is a $(G \times K)$-flat cofibration.
\end{proof}

\begin{lemm}
\label{lemmlevelproper}
The $G$-level model structure is proper.
\end{lemm}
\begin{proof}
First we check right properness. Consider the following pullback diagram
% https://q.uiver.app/?q=WzAsNCxbMCwwLCJBIl0sWzAsMSwiQiJdLFsxLDAsIlgiXSxbMSwxLCJZIl0sWzAsMiwiZyJdLFswLDFdLFsyLDMsImYiXSxbMSwzLCJoIiwyXV0=
    \[\begin{tikzcd}
	A & X \\
	B & Y
	\arrow["g", from=1-1, to=1-2]
	\arrow[from=1-1, to=2-1]
	\arrow["f", from=1-2, to=2-2]
	\arrow["h"', from=2-1, to=2-2]
	\arrow[from=1-1, to=2-2, phantom, "\lrcorner", very near start]
    \end{tikzcd}\]
where f is a $G$-level fibration and $h$ is a $G$-level equivalence. Let $m \leqslant 0$. Evaluating at $\RR^m$ yields a diagram of $(O(m) \times G)$-spaces, which is a pullback diagram because limits of $G$-orthogonal spaces and $(O(m) \times G)$-spaces are computed in $\Topcat$. Then $f(\RR^m)$ is an $\fat(O(m), G)$-fibration and $h(\RR^m)$ is an $\fat(O(m), G)$-equivalence, and since the $\fat(O(m), G)$-projective model structure is right proper by \cite[B.7]{global}, $g(\RR^m)$ is also an $\fat(O(m), G)$-equivalence. Thus $g$ is a $G$-level equivalence.

To check left properness one can use the dual argument. We additionally need to use that if a morphism $f$ of $G$-orthogonal spaces is a $G$-flat cofibration, then it is a $G$-$h$-cofibration (see Lemma~\ref{lemmGhcofclosed}), which means that each $f(\RR^m)$ is an $h$-cofibration of $(O(m) \times G)$-spaces, and then we need to use the Gluing lemma \cite[B.6]{global}.
\end{proof}

\subsection{\texorpdfstring{$G$-$h$}{G-h}-cofibrations and \texorpdfstring{$G$}{G}-global equivalences}

We now check that $G$-global equivalences are preserved by various constructions along $G$-$h$-cofibrations. We use these results to finish the construction of the $G$-global model structure, and in the main part of this paper.

\begin{lemm}[Gluing Lemma]
\label{lemmgluing}
Given a commutative diagram of $G$-orthogonal spaces
\begin{equation*}
    % https://tikzcd.yichuanshen.de/#N4Igdg9gJgpgziAXAbVABwnAlgFyxMJZARgBoAGAXVJADcBDAGwFcYkQANEAX1PU1z5CKMsWp0mrdhwDkPPiAzY8BIuQriGLNohABNef2VC1pMTS1TdeubyODVKAEwaLknSABahxQJXDkF3MJbXZPW3EYKABzeCJQADMAJwgAWyR1EBwIJBcQqxAEkBpGegAjGEYABT8TXSSsaIALHB9ktKQyLJzEAGY3UN0EuRLyyprjRxAG5ta7QpT0xEzspAAWAYLotsXOmlXEAFZNj2jbBXalvIP+-I8AHXuKnHodjuX9nq7LB-umNCar3ml3WnyQxzu7Ee0XoqVSQMo3CAA
\begin{tikzcd}
Y \arrow[d, "\beta"] & X \arrow[l, "f"'] \arrow[r, "g"] \arrow[d, "\alpha"] & Z \arrow[d, "\gamma"] \\
Y'                   & X' \arrow[l, "f'"'] \arrow[r, "g'"]                  & Z'                   
\end{tikzcd}
\end{equation*}
where $\alpha$, $\beta$ and $\gamma$ are $G$-global equivalences, and $f$ and $f'$ are $G$-$h$-cofibrations, the morphism induced on the pushouts $Y \cup_X Z \to Y' \cup_{X'} Z'$ is a $G$-global equivalence.
\end{lemm}
\begin{proof}
Consider a compact Lie group $K$ and an exhaustive sequence of $K$-representations $\{V_i\}_{i \in \NN}$. We have the following diagram of equivariant morphisms of $(K \times G)$-spaces
\begin{equation*}
    % https://tikzcd.yichuanshen.de/#N4Igdg9gJgpgziAXAbVABwnAlgFyxMJZARgBoAGAXVJADcBDAGwFcYkQAdDnGRgfSwACABoAKAGoCAlCAC+pdJlz5CKMsWp0mrdlx78hwgOQTpchSAzY8BIuQqaGLNok7deAwQE1TWGfMVrFTtSDRonHVc9DyEvE0k-c0DlWxQAJgdw7Rc3fU8ALV9-CysU1WQMsK1nXXcDQXz4s1lNGCgAc3giUAAzACcIAFskexAcCCQM6sjcmMEeopAaRnoAI14ABSUbVRA+rHaACxwkkH6hpDIxicQAZiyaqLrPHqbE5bXN7eDXfaOTgJnAbDRCjcZIAAsDxm0Xq7UWgPOIKu4MQAFZoTlYZ52m9ir1gZMaKj7tMsc8hFx1jh6AiLEiRsSblcIuS8pSOEw0IdaQl8UCLogodckBiybV2YIuO16INBrzmpRZEA
\begin{tikzcd}
\tel_i Y(V_i) \arrow[d, "\tel_i \beta(V_i)"] & \tel_i X(V_i) \arrow[l, "\tel_i f(V_i)"'] \arrow[r, "\tel_i g(V_i)"] \arrow[d, "\tel_i \alpha(V_i)"] & \tel_i Z(V_i) \arrow[d, "\tel_i \gamma(V_i)"] \\
\tel_i Y'(V_i)                               & \tel_i X'(V_i) \arrow[l, "\tel_i f'(V_i)"'] \arrow[r, "\tel_i g'(V_i)"]                              & \tel_i Z'(V_i).                              
\end{tikzcd}
\end{equation*}
Here by Proposition~\ref{propglobaltel} $\tel_i \alpha(V_i)$, $\tel_i \beta(V_i)$ and $\tel_i \gamma(V_i)$ are $\fat(K, G)$-equivalences, and the formation of mapping telescopes commutes with pushouts, retracts and $- \times [0, 1]$, so $\tel_i f(V_i)$ and $\tel_i f'(V_i)$ are $h$-cofibrations of $(K \times G)$-spaces. Therefore by the Gluing Lemma for $\fat(K, G)$-equivalences (see for example~\cite[Proposition~B.6]{global}) the induced map on the pushouts of the mapping telescopes is also an $\fat(K, G)$-equivalence. Since again taking mapping telescopes commutes with pushouts, this means that $Y \cup_X Z \to Y' \cup_{X'} Z'$ is a $G$-global equivalence.
\end{proof}

\begin{coro}
\label{corocobase}
For a pushout diagram of $G$-orthogonal spaces
\begin{equation*}
    % https://tikzcd.yichuanshen.de/#N4Igdg9gJgpgziAXAbVABwnAlgFyxMJZABgBpiBdUkANwEMAbAVxiRAA0QBfU9TXfIRRkAjFVqMWbdgHJuvEBmx4CREeXH1mrRCACa8vssFrSY6lqm69cruJhQA5vCKgAZgCcIAWyRkQOBBIAEwWkjogbiDUDHQARjAMAAr8KkIgDDBuOIaRXr6I6gFBiADMYdpsbnIx8Ykpxqq6mdm5nj4h1IFIpTzu+X5dJSJ9eR1lQ352XEA
\begin{tikzcd}
X \arrow[r, "f"] \arrow[d, "g"] & Y \arrow[d]   \\
X' \arrow[r, "f'"]         & Y' \arrow[lu, phantom, "\ulcorner", very near start]
\end{tikzcd}
\end{equation*}
where $f$ is a $G$-global equivalence and either $f$ or $g$ is a $G$-$h$-cofibration, $f'$ is a $G$-global equivalence.
\end{coro}
\begin{proof}
Apply the previous lemma to the diagram
\begin{equation*}
    % https://tikzcd.yichuanshen.de/#N4Igdg9gJgpgziAXAbVABwnAlgFyxMJZARgBpiBdUkANwEMAbAVxiRAA0QBfU9TXfIRQAmclVqMWbdgHJuvEBmx4CRAAxjq9Zq0QgAmvL7LBRMmvHapezj2MDVKDRa2TdHI4v4qhyUS4kdaTkucRgoAHN4IlAAMwAnCABbJA0QHAgkUUDrEFiQagY6ACMYBgAFb1M9eKwIgAscTwTk1OoMpGI7PMSUxABmdszEABZulr6RoazXIL188d6kQfThtUXWgenEAFYNvp3troouIA
\begin{tikzcd}
X \arrow[d, "f"] & X \arrow[-,double line with arrow={-,-}]{l} \arrow[-,double line with arrow={-,-}]{d} \arrow[r, "g"] & X' \arrow[-,double line with arrow={-,-}]{d} \\
Y                & X \arrow[l, "f"'] \arrow[r, "g"]     & X'.
\end{tikzcd}
\end{equation*}
\end{proof}

\begin{coro}
\label{coropushoutproducteq}
For morphisms of $G$-orthogonal spaces $f \colon X_1 \to Y_1$ and $g \colon X_2 \to Y_2$ such that $f$ is a $G$-global equivalence and either $f$ or $g$ is a $G$-$h$-cofibration, their pushout product $f \square g$ is a $G$-global equivalence.

Similarly, assume instead that $f \colon X_1 \to Y_1$ is a morphism of $G$-orthogonal spaces and $g \colon X_2 \to Y_2$ is a map of $G$-spaces. If either $f$ is a $G$-global equivalence or $g$ is a $G$-equivalence, and either $f$ or $g$ is a $G$-$h$-cofibration, their pushout product $f \square g$ is a $G$-global equivalence.
\end{coro}
\begin{proof}
By Lemma~\ref{lemmGglobal}~iv) $f \boxtimes X_2$ and $f \boxtimes Y_2$ are $G$-global equivalences. Depending on the hypothesis, either $f \boxtimes X_2$ or $X_1 \boxtimes g$ is a $G$-$h$-cofibration, so by Corollary~\ref{corocobase} the morphism $\alpha$ is a $G$-global equivalence, and by the 2-out-of-6 property so is $f \square g$.

\begin{equation*}
    % https://tikzcd.yichuanshen.de/#N4Igdg9gJgpgziAXAbVABwnAlgFyxMJZABgBpiBdUkANwEMAbAVxiRAA0B9ARgAIAdfgCMIADzwBbeLy4AmEAF9S6TLnyEU3clVqMWbLn0EjxWKXF4BNTvKUrseAkTLcd9Zq0QhrR4WMnScorKIBgO6kSypLJuep7ePAJ+puZWNsH2ak6apK7U7vpeijowUADm8ESgAGYAThASSAAs1DgQSADM1AxYYPFQdHAAFqUg+XFs1UlwAI5MdLUwvGUZIHUNSFEgbZ3UQjBgUEgAtB1kuh5sPkkmARYr3XT7DAAKqo4aILVYZUM4q+tGohzjtEFsCvEpsZ-GZAulHs83uFsl8fn8AfUgVptu1EF0QPtDp1zhDJjcYalrLYQoCkCDcdjSV5DOSUtIVnY1pikNjQS0LoUQIJGGghnQMRswa1cU1ObTEPzQcQFBQFEA
\begin{tikzcd}
X_1 \boxtimes X_2 \arrow[d, "f \boxtimes X_2"'] \arrow[r, "X_1 \boxtimes g"] & X_1 \boxtimes Y_2 \arrow[rdd, "f \boxtimes Y_2", bend left] \arrow[d, "\alpha"] &                   \\
Y_1 \boxtimes X_2 \arrow[rrd, "Y_1 \boxtimes g"', bend right] \arrow[r]      & P \arrow[rd, "f \square g", dashed] \arrow[lu, phantom, "\ulcorner", very near start]                                 &                   \\
                                                                             &                                                                                 & Y_1 \boxtimes Y_2
\end{tikzcd}
\end{equation*}

The same is true if $g$ is a map of $G$-spaces, since the product of an orthogonal space with a space is the same as the box product with the associated constant orthogonal space. A $G$-equivalence between constant orthogonal spaces is a $G$-global equivalence, and similarly a $G$-$h$-cofibration of spaces is a $G$-$h$-cofibration between constant orthogonal spaces.
\end{proof}

\begin{lemm}
\label{lemmtransfinitent}
For a limit ordinal $\lambda$, consider two $\lambda$-sequences in $\GSpc$, which are colimit preserving functors $X \colon \lambda \to \GSpc$ and $Y \colon \lambda \to \GSpc$, and a natural transformation $f$ between them. If for each $\beta \in \lambda$ the morphisms $g_\beta \colon X_\beta \to X_{\beta + 1}$ and $h_\beta \colon Y_\beta \to Y_{\beta + 1}$ are $G$-$h$-cofibrations and the morphism $f_\beta \colon X_\beta \to Y_\beta$ is a $G$-global equivalence, the morphism induced on the colimits
\[\colim_{\beta \in \lambda} f_\beta \colon \colim_{\beta \in \lambda} X_\beta \to \colim_{\beta \in \lambda} Y_\beta\]
is a $G$-global equivalence.
\end{lemm}
\begin{proof}
By Proposition~\ref{propglobaltel} it is enough to check that for each compact Lie group $K$ and exhaustive sequence of $K$-representations $\{V_i\}_{i \in I}$ the map $\tel_i (\colim_{\beta \in \lambda} f_\beta)(V_i)$ is an $\fat(K, G)$-equivalence. The construction of the mapping telescopes commutes with taking colimits, so this map is isomorphic to $\colim_{\beta \in \lambda} (\tel_i f_\beta(V_i))$.

For each $\beta \in \lambda$ the map $\tel_i f_\beta(V_i)$ is an $\fat(K, G)$-equivalence, and the maps $\tel_i g_\beta(V_i)$ and $\tel_i h_\beta(V_i)$ are $h$-cofibrations of $(K \times G)$-spaces, and so in particular $h$-cofibrations of underlying compactly generated weak Hausdorff spaces, and therefore closed embeddings.

For each $\Gamma \in \fat(K, G)$ taking $\Gamma$-fixed points commutes with filtered colimits along closed embeddings (see~\cite[Proposition~B.1 ii)]{global}). Colimits with the shape of a filtered poset and built out of closed embeddings of compactly generated weak Hausdorff spaces can be computed in the category of all topological spaces (see~\cite[Proposition~A.14 (ii)]{global}). Weak Hausdorff spaces are $T_1$, so by~\cite[Proposition~2.4.2]{hovey2007model} maps from compact spaces ($\partial D^l$ and $D^l$ in this case) into the colimit of a $\lambda$-sequence of closed embeddings (for $\lambda$ a limit ordinal) factor through some stage $\beta \in \lambda$. Therefore compact spaces are finite in $\Topcat$ relative closed embeddings.

This implies that, for the $\lambda$-sequences given by $(\tel_i g_\beta(V_i))^\Gamma$ and $(\tel_i h_\beta(V_i))^\Gamma$, which consist of closed embeddings, and the natural transformation between them given by the maps $(\tel_i f_\beta(V_i))^\Gamma$ which are weak homotopy equivalences, the map induced on the colimits
\[\colim_{\beta \in \lambda} (\tel_i f_\beta(V_i))^\Gamma \cong (\colim_{\beta \in \lambda} (\tel_i f_\beta(V_i)))^\Gamma\]
is a weak homotopy equivalence. Therefore $\tel_i (\colim_{\beta \in \lambda} f_\beta)(V_i)$ is an $\fat(K, G)$-equivalence.
\end{proof}

\begin{coro}
\label{corotransfinitegglobal}
A transfinite composition of morphisms in $\GSpc$ that are $G$-$h$-cofibrations and $G$-global equivalences is a $G$-global equivalence.
\end{coro}
\begin{proof}
We check this via transfinite induction on the ordinal $\lambda$. Let $Y \colon \lambda \to \GSpc$ be a $\lambda$-sequence such that for each $\beta \in \lambda$ the morphism $h_\beta \colon Y_\beta \to Y_{\beta+1}$ is a $G$-$h$-cofibration and a $G$-global equivalence. The base case and the case where $\lambda$ is a successor ordinal hold because composition of two $G$-global equivalences is a $G$-global equivalence.

If $\lambda$ is a limit ordinal, set $X \colon\lambda \to \GSpc$ as the constant functor $X_\beta=Y_0$. Define a natural transformation $f \colon X \Rightarrow Y$ by letting $f_\beta$ be the morphism $Y_0 \to Y_\beta$. This is the transfinite composition of $Y$ restricted to $\beta + 1$. Then by the induction hypothesis $f_\beta$ is a $G$-global equivalence for each $\beta \in \lambda$. Then we use Lemma~\ref{lemmtransfinitent} to obtain that $\colim_{\beta \in \lambda} f_\beta$ is a $G$-global equivalence, but this morphism is precisely the transfinite composition of $Y$.
\end{proof}

\subsection{\texorpdfstring{$G$}{G}-global model structure}

We now go back to constructing the $G$-global model structure, starting with the fibrations.

\begin{defi}[$G$-global fibration]
\label{defiGglobalfib}
A morphism of $G$-orthogonal spaces $f \colon X \to Y$ is a \emph{$G$-global fibration} if it is a $G$-level fibration, and for each compact Lie group $K$, every graph subgroup $\Gamma \in \fat(K, G)$, and every linear isometric embedding of $K$-representations $\psi \colon V \to W$ with $V$ faithful, the induced map $X(V)^\Gamma \to Y(V)^\Gamma \times_{Y(W)^\Gamma} X(W)^\Gamma$ is a weak homotopy equivalence. Since $f$ is a $G$-level fibration, and so $f(V)^\Gamma$ and $f(W)^\Gamma$ are Serre fibrations, this is equivalent to the following square being homotopy cartesian
% https://q.uiver.app/?q=WzAsNCxbMCwwLCJYKFYpXlxcR2FtbWEiXSxbMCwxLCJZKFYpXlxcR2FtbWEiXSxbMSwwLCJYKFcpXlxcR2FtbWEiXSxbMSwxLCJZKFcpXlxcR2FtbWEiXSxbMCwxLCJmKFYpXlxcR2FtbWEiLDJdLFswLDIsIlgoXFxwc2kpXlxcR2FtbWEiXSxbMiwzLCJmKFcpXlxcR2FtbWEiLDJdLFsxLDMsIlkoXFxwc2kpXlxcR2FtbWEiXV0=
\[\begin{tikzcd}
	{X(V)^\Gamma} & {X(W)^\Gamma} \\
	{Y(V)^\Gamma} & {Y(W)^\Gamma}.
	\arrow["{f(V)^\Gamma}"', from=1-1, to=2-1]
	\arrow["{X(\psi)^\Gamma}", from=1-1, to=1-2]
	\arrow["{f(W)^\Gamma}"', from=1-2, to=2-2]
	\arrow["{Y(\psi)^\Gamma}", from=2-1, to=2-2]
\end{tikzcd}\]
\end{defi}

\begin{constr}
\label{constrKg}
Fix a compact Lie group $G$. We now construct the set $\Ks_G$, where $\Js_G \cup \Ks_G$ is a set of generating acyclic cofibrations for the $G$-global model structure of Theorem~\ref{thmGglobalmodel}. Recall that $\Js_G$ is the set of generating acyclic cofibrations of the $G$-level model structure given in Theorem~\ref{thmlevelmodel}. Let $K$ be a compact Lie group, let $V$ be a faithful $K$-representation, let $W$ be a $K$-representation and let $\Gamma \in \fat(K, G)$ be a graph subgroup. We consider the following restriction morphism of $G$-orthogonal spaces
\[\rho_{\Gamma, V, W; G} \colon L_{\Gamma, V \oplus W; G} = (\Lcat(V \oplus W, -) \times G) / \Gamma \to (\Lcat(V, -) \times G) / \Gamma = L_{\Gamma, V; G}.\]
The morphism $\rho_{\Gamma, V, W; G}$ is a $G$-global equivalence because the semifree $G$-orthogonal spaces are closed and given a compact Lie group $L$, the map
\[\rho_{V, W} \colon \Lcat(V \oplus W, \Ul) \to \Lcat(V, \Ul)\]
is a $(K \times L)$-homotopy equivalence by \cite[1.1.26 ii)]{global} (recall that $\Ul$ here is a complete $L$-universe).

Now let $\kappa$ be a set of representatives of isomorphism classes of triples $(K, \Gamma, V, W)$ consisting of a compact Lie group $K$, a faithful $K$-representation $V$, a $K$-representation $W$, and a graph subgroup $\Gamma \in \fat(K, G)$. Let $\Ks_G$ be the set
\[\Ks_G= \bigcup_{(K, \Gamma, V, W) \in \kappa} \{\, \iota_{\rho_{\Gamma, V, W; G}} \square i_l \mid l \geqslant 0\, \}.\]

Recall that $\iota_{\rho_{\Gamma, V, W; G}}$ denotes the inclusion of the mapping cylinder $L_{\Gamma, V \oplus W; G} \to M_{\rho_{\Gamma, V, W; G}}$. Note that here we allow $V$ to be $0$. For the generating acyclic cofibrations of the positive global model structure on $\Spc$, we do require that $V \neq 0$. If we did that here, in Definition~\ref{defiGglobalfib}, and in Theorem~\ref{thmlevelmodel}, we would obtain the \emph{positive $G$-global model structure}.
\end{constr}

\begin{lemm}
\label{lemmJandK}
Any morphism in $\Js_G \cup \Ks_G$ is a $G$-global equivalence and a $G$-flat cofibration. Any morphism obtained from $\Js_G \cup \Ks_G$ by transfinite composition and cobase changes is also a $G$-global equivalence and a $G$-flat cofibration.
\end{lemm}
\begin{proof}
Any morphism $f \in \Js_G$ is an acyclic cofibration in the $G$-level model structure, so it is a $G$-flat cofibration and by Lemma~\ref{lemmlevelglobal} a $G$-global equivalence.

Fix a compact Lie group $K$, a faithful $K$-representation $V$, a $K$-representation $W$, a graph subgroup $\Gamma \in \fat(K, G)$, and $l \geqslant 0$. Consider $f = \iota_{\rho_{\Gamma, V, W; G}} \square i_l$ in $\Ks_G$. We saw in Construction~\ref{constrKg} that $\rho_{\Gamma, V, W; G}$ is a $G$-global equivalence. The projection $M_{\rho_{\Gamma, V, W; G}} \to L_{\Gamma, V; G}$ from the mapping cylinder of $\rho_{\Gamma, V, W; G}$ to its target is a homotopy equivalence in $\GSpc$. Therefore it is a $G$-level equivalence, and thus a $G$-global equivalence. By the 2-out-of-6 property $\iota_{\rho_{\Gamma, V, W; G}}$ is also a $G$-global equivalence.

The $G$-orthogonal spaces $L_{\Gamma, V \oplus W; G}$ and $L_{\Gamma, V; G}$ are $G$-flat orthogonal spaces because they are isomorphic to $L_{\Gamma', \RR^{n+m}; G}$ and $L_{\Gamma'', \RR^n; G}$ respectively, for some $n, m \geqslant 0$, $\Gamma' \in \fat(O(n+m), G)$ and $\Gamma'' \in \fat(O(n), G)$. Then we obtain that 
\[L_{\Gamma, V \oplus W; G} \to L_{\Gamma, V \oplus W; G} \amalg L_{\Gamma, V; G}\]
is a $G$-flat cofibration. Also since the $G$-level model structure of Theorem~\ref{thmlevelmodel} is topological, $L_{\Gamma, V \oplus W; G} \times i_1$ is a $G$-flat cofibration. Putting this together we obtain that $\iota_{\rho_{\Gamma, V, W; G}}$ is a $G$-flat cofibration, and again because the $G$-level model structure is topological so is $f$. By Corollary~\ref{coropushoutproducteq}, $f = \iota_{\rho_{\Gamma, V, W; G}} \square i_l$ is a $G$-global equivalence.

Using the closure properties of Corollary~\ref{corocobase} and Corollary~\ref{corotransfinitegglobal} we obtain the second part of the lemma.
\end{proof}

\begin{lemm}
\label{lemmsmallsources}
The sources of all morphisms in $\Is_G$, $\Js_G$ and $\Ks_G$ are finite (and thus small) with respect to the class of maps that are levelwise closed embeddings. Since $G$-$h$-cofibrations are levelwise closed embeddings, they are also finite with respect to the class of $G$-$h$-cofibrations.
\end{lemm}
\begin{proof}

We first check that for any compact Lie group $K$, faithful $K$-representation $V$, graph subgroup $\Gamma \in \fat(K, G)$, and compact space $A$, the $G$-orthogonal space $L_{\Gamma, V; G} \times A$ is finite with respect to morphisms which are levelwise closed embeddings.

We recalled in the proof of Lemma~\ref{lemmtransfinitent} that compact spaces are finite in $\Topcat$ relative closed embeddings. Taking $\Gamma$-fixed points commutes with filtered colimits along closed embeddings. Consider a limit ordinal $\lambda$, and a $\lambda$-sequence $X \colon \lambda \to \GSpc$ of levelwise closed embeddings. By the semifreeness property of $L_{\Gamma, V; G} \times A$, and since colimits in $\GSpc$ are computed levelwise, we have that 
\begin{align*}
    \GSpc(L_{\Gamma, V; G} \times A, \colim_{\beta \in \lambda} X_\beta) \cong \Topcat(A, (\colim_{\beta \in \lambda} X_\beta)(V)^\Gamma) \cong \Topcat(A, \colim_{\beta \in \lambda} (X_\beta(V)^\Gamma)) \cong \\ \colim_{\beta \in \lambda} \Topcat(A, (X_\beta(V)^\Gamma)) \cong \colim_{\beta \in \lambda} \GSpc(L_{\Gamma, V; G} \times A, X_\beta).
\end{align*}
So for a generating cofibration $i \in \Is_G$, its source is of the form $L_{\Gamma, \RR^m; G} \times \partial D^l$, so it is finite relative levelwise closed embeddings. Similarly the source of a generating acyclic cofibration $j \in \Js_G$ is $L_{\Gamma, \RR^m; G} \times D^l$, so it is also finite relative levelwise closed embeddings.

For a generating acyclic cofibration $k = \iota_{\rho_{\Gamma, V, W; G}} \square i_l$ in $\Ks_G$, its source is
\[L_{\Gamma, V \oplus W; G} \times D^l \cup_{L_{\Gamma, V \oplus W; G} \times \partial D^l} M_{\rho_{\Gamma, V, W; G}} \times \partial D^l.\]
The $G$-orthogonal space $M_{\rho_{\Gamma, V, W; G}} \times \partial D^l$ is a finite colimit of objects of the form $L_{\Gamma, V; G} \times A$. Therefore it is also finite relative levelwise closed embeddings, because in $\Setcat$ finite limits commute with filtered colimits. By the same argument, the source of $k$ is also finite relative levelwise closed embeddings.

$G$-$h$-cofibrations are levelwise $h$-cofibrations of spaces, which are closed embeddings in the category of compactly generated weak Hausdorff spaces. Therefore $G$-$h$-cofibrations are levelwise closed embeddings.
\end{proof}

\begin{lemm}
\label{lemmGfiblifting}
A morphism in $\GSpc$ is a $G$-global fibration if and only if it has the right lifting property with respect to $\Js_G \cup \Ks_G$.
\end{lemm}
\begin{proof}
Every linear isometric embedding of $K$-representations is isomorphic to an embedding of the form $i_V \colon V \to V \oplus W$. Thus Definition~\ref{defiGglobalfib} can be altered slightly to say that a morphism $f$ is a $G$-global fibration if and only if it is a $G$-level fibration and for each compact Lie group $K$, graph subgroup $\Gamma \in \fat(K, G)$, and $K$-representations $V$ and $W$, the square
% https://q.uiver.app/?q=WzAsNCxbMCwwLCJYKFYpXlxcR2FtbWEiXSxbMCwxLCJZKFYpXlxcR2FtbWEiXSxbMSwwLCJYKFcpXlxcR2FtbWEiXSxbMSwxLCJZKFcpXlxcR2FtbWEiXSxbMCwxLCJmKFYpXlxcR2FtbWEiLDJdLFswLDIsIlgoXFxwc2kpXlxcR2FtbWEiXSxbMiwzLCJmKFcpXlxcR2FtbWEiLDJdLFsxLDMsIlkoXFxwc2kpXlxcR2FtbWEiXV0=
\[\begin{tikzcd}
	{X(V)^\Gamma} & {X(V \oplus W)^\Gamma} \\
	{Y(V)^\Gamma} & {Y(V \oplus W)^\Gamma}
	\arrow["{f(V)^\Gamma}"', from=1-1, to=2-1]
	\arrow["{X(i_V)^\Gamma}", from=1-1, to=1-2]
	\arrow["{f(V \oplus W)^\Gamma}"', from=1-2, to=2-2]
	\arrow["{Y(i_V)^\Gamma}", from=2-1, to=2-2]
\end{tikzcd}\]
is homotopy cartesian. By Remark~\ref{remGsemifree}, the morphism $\rho_{\Gamma, V, W; G}$ represents the natural transformation
\[(-)(i_V)^\Gamma \colon (-)(V)^\Gamma \Rightarrow (-)(V \oplus W)^\Gamma.\] 
By applying \cite[Proposition~1.2.16]{global} to the $G$-level model structure we obtain that the previous square is homotopy cartesian if and only if  $f$ has the right lifting property with respect to $\iota_{\rho_{\Gamma, V, W; G}} \square i_l$ for all $l \geqslant 0$. The set $\Js_G$ is a set of generating acyclic cofibrations of the $G$-level model structure, so a morphism is a $G$-level fibration if and only if it has the right lifting property with respect to $\Js_G$. Therefore a morphism in $\GSpc$ is a $G$-global fibration if and only if it has the right lifting property with respect to $\Js_G \cup \Ks_G$.
\end{proof}

\begin{lemm}
\label{lemmpullbackfibGglobal}
A pullback of a $G$-global equivalence along a $G$-level fibration is also a $G$-global equivalence.
\end{lemm}
\begin{proof}
Consider the pullback square
% https://q.uiver.app/?q=WzAsNCxbMCwwLCJQIl0sWzAsMSwiWiJdLFsxLDEsIlkiXSxbMSwwLCJYIl0sWzMsMiwiZiJdLFsxLDIsImgiLDIseyJzdHlsZSI6eyJoZWFkIjp7Im5hbWUiOiJlcGkifX19XSxbMCwxLCJnIiwyXSxbMCwzLCJrIl0sWzAsMiwiIiwxLHsic3R5bGUiOnsibmFtZSI6ImNvcm5lciJ9fV1d
\begin{equation}
\begin{tikzcd}
\label{lemmpullbackfibGglobal1}
	P & X \\
	Z & Y
	\arrow["f", from=1-2, to=2-2]
	\arrow["h"', from=2-1, to=2-2]
	\arrow["g"', from=1-1, to=2-1]
	\arrow["k", from=1-1, to=1-2]
	\arrow["\lrcorner"{anchor=center, pos=0.125}, draw=none, from=1-1, to=2-2]
\end{tikzcd}
\end{equation}
where $f$ is a $G$-global equivalence, and $h$ is a $G$-level fibration. Consider a compact Lie group $K$, a $K$-representation $V$, a graph subgroup $\Gamma \in \fat(K, G)$, and a lifting problem given by $\alpha \colon \partial D^l \to P(V)^\Gamma$ and $\beta \colon D^l \to Z(V)^\Gamma$ with $g(V)^\Gamma \circ \alpha = \beta \circ i_l$. Since $f$ is a $G$-global equivalence, there is a $K$-representation $W$, a $K$-equivariant linear isometric embedding $\psi \colon V \to W$, and a morphism $\lambda \colon D^l \to X(W)^\Gamma$ such that
\[\lambda \circ i_l = X(\psi)^\Gamma \circ k(V)^\Gamma \circ \alpha\]
and there is a $\partial D^l$-relative homotopy $H$ from $Y(\psi)^\Gamma \circ h(V)^\Gamma \circ \beta$ to $f(W)^\Gamma \circ \lambda$ relative $\partial D^l$. Since $h(W)^\Gamma$ is a Serre fibration, there is a lift $H'$ in the following diagram
% https://q.uiver.app/?q=WzAsNCxbMCwwLCJEXmwgXFx0aW1lcyBcXHswXFx9IFxcY3VwIFxccGFydGlhbCBEXmwgXFx0aW1lcyBbMCwgMV0iXSxbMCwxLCJEXmwgXFx0aW1lcyBbMCwgMV0iXSxbNiwwLCJaKFcpXlxcR2FtbWEiXSxbNiwxLCJZKFcpXlxcR2FtbWEiXSxbMSwzLCJIIl0sWzEsMiwiSCciLDAseyJzdHlsZSI6eyJib2R5Ijp7Im5hbWUiOiJkYXNoZWQifX19XSxbMiwzLCJoKFcpXlxcR2FtbWEiLDAseyJzdHlsZSI6eyJoZWFkIjp7Im5hbWUiOiJlcGkifX19XSxbMCwxXSxbMCwyLCIoWihXKV5cXEdhbW1hIFxcY2lyYyBcXGJldGEpIFxcY3VwIChnKFcpXlxcR2FtbWEgXFxjaXJjIFAoXFxwc2kpXlxcR2FtbWEgXFxjaXJjIFxcYWxwaGEgXFx0aW1lcyBbMCwgMV0pIl1d
\[\begin{tikzcd}
	{D^l \times \{0\} \cup \partial D^l \times [0, 1]} &&&&&& {Z(W)^\Gamma} \\
	{D^l \times [0, 1]} &&&&&& {Y(W)^\Gamma}.
	\arrow["H", from=2-1, to=2-7]
	\arrow["{H'}", dashed, from=2-1, to=1-7]
	\arrow["{h(W)^\Gamma}", from=1-7, to=2-7]
	\arrow[from=1-1, to=2-1, hook]
	\arrow["{(Z(\psi)^\Gamma \circ \beta) \cup ((g(W)^\Gamma \circ P(\psi)^\Gamma \circ \alpha) \times [0, 1])}", from=1-1, to=1-7]
\end{tikzcd}\]
Since the square (\ref{lemmpullbackfibGglobal1}) is a pullback there is a unique $\lambda' \colon D^l \to P(W)^\Gamma$ such that $g(W)^\Gamma \circ \lambda' = H'(-, 1)$ and $k(W)^\Gamma \circ \lambda' = \lambda$. Also by the universal property of the pullback (\ref{lemmpullbackfibGglobal1}) we obtain that $\lambda' \circ i_l = P(\psi)^\Gamma \circ \alpha$. Since $H'$ is a homotopy relative $\partial D^l$ between $g(W)^\Gamma \circ \lambda'$ and $Z(\psi)^\Gamma \circ \beta$, we obtain that $g$ is a $G$-global equivalence.
\end{proof}

\begin{lemm}
\label{lemmGglobalandisGlevel}
If $f \colon X \to Y$ is a $G$-global equivalence and a $G$-global fibration then it is also a $G$-level equivalence.
\end{lemm}
\begin{proof}
Consider $m \geqslant 0$, a graph subgroup $\Gamma \in \fat(O(m), G)$ given by a homomorphism $H \to G$ with $H \leqslant O(m)$, and a lifting problem of the form
\[% https://tikzcd.yichuanshen.de/#N4Igdg9gJgpgziAXAbVABwnAlgFyxMJZABgBpiBdUkANwEMAbAVxiRAB1206AnPRgAQARAHoMQAX1LpMufIRRkAjFVqMWbUeKkzseAkSXlV9Zq0QgAGgAoAagEoRnNAAssk6SAx75h0iupTDQsATTtHZzdJVRgoAHN4IlAAMx4IAFskIxAcCCQAZh0QVIykACZqXILA9XNi8KcuKKKSzMQyHLzECrUzNk5GVzoPFLS2jqrEbJw6LHELFwgIAGtoiSA
\begin{tikzcd}
\partial D^l \arrow[r, "\alpha"] \arrow[d, hook, "i_l"] & X(\RR^m)^\Gamma \arrow[d, "f(\RR^m)^\Gamma"] \\
D^l \arrow[r, "\beta"]                                    & Y(\RR^m)^\Gamma.                       
\end{tikzcd}\]
Since $f$ is a $G$-global equivalence there is an embedding of $H$-representations $\psi \colon \RR^m \to V$ and a map $\lambda \colon D^l \to X(V)^\Gamma$ such that $\lambda \circ i_l = X(\psi)^\Gamma \circ \alpha$ and $f(V)^\Gamma \circ \lambda$ is homotopic relative $\partial D^l$ to $Y(\psi)^\Gamma \circ \beta$.

Since $f$ is a $G$-level fibration, $f(V)^\Gamma$ is a Serre fibration. By lifting against
\[D^l \times \{0\} \cup \partial D^l \times [0, 1] \to D^l \times [0, 1]\]
which is a cofibration of spaces we can replace $\lambda$ with a $\lambda'$ such that $\lambda' \circ i_l = X(\psi)^\Gamma \circ \alpha$ and $f(V)^\Gamma \circ \lambda' = Y(\psi)^\Gamma \circ \beta$.

Since $f$ is a $G$-global fibration,
\[(f(\RR^m)^\Gamma, X(\psi)^\Gamma) \colon X(\RR^m)^\Gamma \to Y(\RR^m)^\Gamma \times_{Y(V)^\Gamma} X(V)^\Gamma\]
is a weak homotopy equivalence. This means that by~\cite[9.6 Lemma]{may1999concise} there is a map $\lambda''$ in the following diagram
% https://tikzcd.yichuanshen.de/#N4Igdg9gJgpgziAXAbVABwnAlgFyxMJZABgBpiBdUkANwEMAbAVxiRAB1206AnPRgAQARAHoMQAX1LpMufIRRkAjFVqMWbUeKkzseAkSXlV9Zq0QgAGgAoAagEoRnNAAssk6SAx75h0iupTDQsATTtHZzdJVRgoAHN4IlAAMx4IAFskIxAcCCQAZh0QVIykACZqXILA9XNi8KcuKKKSzMQyHLzECrUzNk5GVzoPFLS2jqrEbJw6LHELFwgIAGtoiSA
\[\begin{tikzcd}
\partial D^l \arrow[r, "\alpha"] \arrow[d, hook, "i_l"] & X(\RR^m)^\Gamma \arrow[d, "{(f(\RR^m)^\Gamma, X(\psi)^\Gamma)}"] \\
D^l \arrow[r, "{(\beta, \lambda')}"'] \arrow[ru, "\lambda''", dashed] & Y(\RR^m)^\Gamma \times_{Y(V)^\Gamma} X(V)^\Gamma
\end{tikzcd}\]
such that the upper-left triangle commutes and the lower-right triangle commutes up to homotopy relative $\partial D^l$. Thus by~\cite[9.6 Lemma]{may1999concise} again $f(\RR^m)^\Gamma$ is a weak homotopy equivalence, and so $f$ is a $G$-level equivalence.\end{proof}

\begin{thm}[$G$-global model structure]
\label{thmGglobalmodel}
There is a proper topological cofibrantly generated model structure on the category $\GSpc$ of $G$-orthogonal spaces, with the $G$-global equivalences as the weak equivalences, the $G$-global fibrations as the fibrations, and the $G$-flat cofibrations of the $G$-level model structure as the cofibrations. We call this model structure the $G$-global model structure. 

$\Is_G$ is a set of generating cofibrations of this model structure. The set $\Js_G \cup \Ks_G$ is a set of generating acyclic cofibrations. Recall that $\Is_G$, $\Js_G$ and $\Ks_G$ were given in Theorem~\ref{thmlevelmodel} and Construction~\ref{constrKg}.
\end{thm}
\begin{proof}
$\GSpc$ is complete and cocomplete. The $G$-global equivalences satisfy the 2-out-of-6 property and are closed under retracts by Lemma~\ref{lemmGglobal}~i) and ii) respectively. The $G$-global fibrations and $G$-flat cofibrations are closed under retracts because they can be defined via lifting properties, see Lemma~\ref{lemmGfiblifting} and Theorem~\ref{thmlevelmodel} respectively. Now we have to check the lifting and factorization axioms.

 Given a morphism in $\GSpc$, we can use the $G$-level model structure of Theorem~\ref{thmlevelmodel} to decompose it into $f \circ i$ where $i$ is a $G$-flat cofibration and $f$ is a $G$-level fibration and a $G$-level equivalence, so it is also a $G$-global equivalence by Lemma~\ref{lemmlevelglobal}. Given $\psi \colon V \to W$ a linear isometric embedding of faithful $K$-representations, in the square
% https://q.uiver.app/?q=WzAsNCxbMCwwLCJYKFYpXlxcR2FtbWEiXSxbMCwxLCJZKFYpXlxcR2FtbWEiXSxbMSwwLCJYKFcpXlxcR2FtbWEiXSxbMSwxLCJZKFcpXlxcR2FtbWEiXSxbMCwxLCJmKFYpXlxcR2FtbWEiLDJdLFswLDIsIlgoXFxwc2kpXlxcR2FtbWEiXSxbMiwzLCJmKFcpXlxcR2FtbWEiLDJdLFsxLDMsIlkoXFxwc2kpXlxcR2FtbWEiXV0=
\[\begin{tikzcd}
	{X(V)^\Gamma} & {X(W)^\Gamma} \\
	{Y(V)^\Gamma} & {Y(W)^\Gamma}
	\arrow["{f(V)^\Gamma}"', from=1-1, to=2-1]
	\arrow["{X(\psi)^\Gamma}", from=1-1, to=1-2]
	\arrow["{f(W)^\Gamma}"', from=1-2, to=2-2]
	\arrow["{Y(\psi)^\Gamma}", from=2-1, to=2-2]
\end{tikzcd}\]
the two vertical morphisms are weak equivalences by Lemma~\ref{lemmlevelglobal}. Therefore this square is homotopy cartesian and $f$ is a $G$-global fibration. This gives one of the factorization axioms.

For the second factorization axiom, we apply Quillen's small object argument to the set $\Js_G \cup \Ks_G$, which we can do by Lemma~\ref{lemmsmallsources}. This factors any morphism into $f \circ j$, where by Lemma~\ref{lemmJandK} we know that $j$ is a $G$-flat cofibration and a $G$-global equivalence, and $f$ has the right lifting property with respect to $\Js_G \cup \Ks_G$, so by Lemma~\ref{lemmGfiblifting} it is a $G$-global fibration. This gives the second factorization axiom. Note for later that this $j$ by construction has the left lifting property with respect to $G$-global fibrations.

One of the lifting axioms can be obtained from the $G$-level model structure. By Lemma~\ref{lemmGglobalandisGlevel}, a morphism which is both a $G$-global fibration and a $G$-global equivalence is a $G$-level equivalence, so it has the right lifting property with respect to the $G$-flat cofibrations.

Lastly, consider a morphism $g$ which is both a $G$-flat cofibration and a $G$-global equivalence. We can use Quillen's small object argument on the set $\Js_G \cup \Ks_G$ again to decompose $g$ into $f \circ j$, where $f$ is a $G$-global fibration and $j$ is a $G$-global equivalence which has the left lifting property with respect to $G$-global fibrations. By the 2-out-of-6 property $f$ is also a $G$-global equivalence. Then by the previously proven lifting axiom $g$ is a retract of $j$, so it also has the left lifting property with respect to $G$-global fibrations.

This model structure is right proper by Lemma~\ref{lemmpullbackfibGglobal} ($G$-global fibrations are $G$-level fibrations) and left proper by Corollary~\ref{corocobase}. Using~\cite[Proposition~B.5]{global} we obtain that this model structure is topological, taking $\mathpzc{G}$ and $\mathpzc{Z}$ in that statement to be
\[\mathpzc{G}= \{\, L_{\Gamma, \RR^m; G}  \mid m \geqslant 0, \Gamma \in \fat(O(m), G) \,\} \,\text{and}\, \mathpzc{Z}= \{\, \iota_{\rho_{\Gamma, V, W; G}} \mid (K, \Gamma, V, W) \in \kappa \}.\qedhere\]
\end{proof}

\begin{rem}
\label{remwhygraph2}
As mentioned in Remark~\ref{remwhygraph1}, we can define a different class of $G$-global equivalences by checking the condition from Definition~\ref{defiGglobal} on all subgroups of $K \times G$ instead of only on the graph subgroups. We can do the same for all the results of this appendix, replacing $\fat(K, G)$ everywhere by the set of all closed subgroups of $K \times G$. We can take the $G$-level model structure given by all subgroups briefly mentioned right after Theorem~\ref{thmlevelmodel}, and localize it at this smaller class of $G$-global equivalences. This gives us a model structure with this smaller class of $G$-global equivalences as the weak equivalences, as well as fibrations and cofibrations that are similarly defined by looking at all subgroups instead of just the graph subgroups. However, as shown by the various results of this article, the $G$-global model structure constructed in this appendix is more relevant when looking at operads in $\Spc$.
\end{rem}

\printbibliography

\end{document}